\documentclass[12pt,a4paper]{article}%
\usepackage[utf8]{inputenc}
\usepackage{hyperref}
\usepackage{amsmath}
\usepackage{amsfonts}
\usepackage{amssymb}
\usepackage{xcolor}
\usepackage{graphicx}%
\usepackage{array}
\usepackage{lineno}
\usepackage{caption}
\providecommand{\U}[1]{\protect\rule{.1in}{.1in}}
\newtheorem{theorem}{Theorem}

\newtheorem{example}[theorem]{Example}

\newtheorem{lemma}[theorem]{Lemma}

\newtheorem{problem}[theorem]{Problem}
\newtheorem{proposition}[theorem]{Proposition}
\newtheorem{remark}[theorem]{Remark}

\newenvironment{proof}[1][Proof]{\noindent\textbf{#1.} }{\ \hfill \rule{0.5em}{0.5em}\bigskip}
\graphicspath{{../Slike/}{Slike/}}

\begin{document}

\title{\large Spreads of degrees in maximal planar graphs -- an exposition}
\author{Yair Caro$^{1}$, Riste \v{S}krekovski$^{2,3,4}$, Christina Zarb$^{5}$\\[0.3cm]
{\small $^{1}$ \textit{University of Haifa-Oranim, Kiryat Tiv'on, Israel}}\\[0.1cm]
{\small $^{2}$ \textit{University of Ljubljana, Faculty of Mathematics and Physics,
Ljubljana, Slovenia}}\\[0.1cm]
{\small $^{3}$ \textit{Faculty of Information Studies, Novo Mesto, Slovenia}}\\[0.1cm]
{\small $^{4}$ \textit{Rudolfovo -- Science and Technology Centre Novo Mesto, Slovenia}}\\[0.1cm]
{\small $^{5}$ \textit{Department of Mathematics, University of Malta, Malta}}}
\date{}
\date{}
\maketitle

\begin{abstract}
For a graph $G$ and a set $B\subseteq V(G)$, the spread $\mathrm{sp}(B)$ of $B$ is the difference between the
largest and the smallest degree in $G$ of a vertex of $B$, and for an integer $k\geq 0$ the parameter
$\mathrm{sp}(G,k)$ is the largest cardinality of a set $B$ with $\mathrm{sp}(B)\leq k$. Caro, Lauri and Zarb
asked for the minimum of $\mathrm{sp}(G,k)$ over the maximal planar graphs of order $n$; we write
$\mathrm{MP}(n,\delta ,k)$ for this minimum over the maximal planar graphs of order $n$ and minimum degree
$\delta \in \{3,4,5\}$.

This manuscript is intended as an exposition of the subject of spread in the degree sequence of a graph $G$,
with emphasis on maximal planar graphs. We apply the general lower bound of the companion paper \cite{P1} to
this class, together with further ideas, some based on Caro--West and Caro--Lauri--Zarb and some new. In particular we asymptotically determine the values of $\mathrm{MP}(n,\delta ,k)$ for every pair $(\delta ,k)$ with
$\delta \in \{3,4,5\}$ and $k\geq 0$. This solution recovers the case $\mathrm{MP}(n,\delta ,0)$, which is the
repetition number $\mathrm{rep}(G)$ introduced by Caro and West.
\end{abstract}



\section{Introduction}

Every graph on at least two vertices has two vertices of the same degree. This classroom observation
has been generalised in several directions, among them a characterisation of the graphs with only one repeated pair of
degrees \cite{BehzadChartrand}, the characterisation of graphic sequences
\cite{ErdosGallai, Hakimi, Havel}, and the study of the repetition number $\mathrm{rep}(G)$, the
maximum multiplicity of a value in the degree sequence of $G$, introduced by Caro and West
\cite{CaroWest}; for the degrees of maximal outerplanar graphs in particular see Jao and West
\cite{JaoWest}.

Since everything below grows out of a handful of papers, we recall what they contain, beginning with
the repetition number. Caro and West \cite{CaroWest} introduced $\mathrm{rep}(G)$ and proved, by a
counting argument on the degree sequence, that
\[
\mathrm{rep}(G)\ \geq \ \left\lceil \frac{n}{2d-2\delta +1}\right\rceil
\]
for every graph of order $n$ with average degree $d$ and minimum degree $\delta $, next to the
elementary $\mathrm{rep}(G)\geq n/\Delta (G)$ which the pigeonhole principle gives when $G$ has no
isolated vertex.  Caro and West showed
their bound to be approximately sharp on several classes at once: the least repetition number of a
graph of order $n$ is asymptotic to $n/3$ among trees, to $n/5$ among maximal outerplanar graphs, to
$n/7$ among maximal planar graphs, to $n/5$ among maximal planar graphs of minimum degree four and to
$n/3$ among maximal planar graphs of minimum degree five \cite[Theorem 2.5]{CaroWest}. The last three of
these are the case $k=0$ of the problem of the present paper, and the constructions behind them are
tilings: a tile is designed to carry the intended degrees, it is repeated over the hexagonal grid, and
the pattern is stopped along a boundary of $O(\sqrt{n})$ vertices. That is also the shape of our
constructions in Section \ref{Sec_asym}, and the reason why the error terms there are $O(\sqrt{n})$ as
well.

The question regarding how far the degrees of a planar graph can be pushed apart has of course been asked in
other ways. One classical answer bounds the second moment of the degree sequence: Cook \cite{Cook}
proved that a planar graph of order $n$ with degrees $d_{1},\ldots ,d_{n}$ satisfies
$\sum_{i}d_{i}^{2}\leq 2n^{2}+O(n)$, with the main term best possible, and $\sum_{i}d_{i}^{2}\leq
n^{2}+O(n)$ if the graph is outerplanar. For the exact upper bound on the sum of the $k$-th powers of the degrees of a planar graph see Harant,
Jendro\v{l} and Madaras \cite{HarantJendrolMadaras}.


The generalisation we consider here is due to Erd\H{o}s, Chen, Rousseau and Schelp \cite{Erdos} and
was developed further by Caro, Lauri and Zarb \cite{CaroLauriZarb}. Let $G=(V,E)$ be a graph. For
$B\subseteq V$, the \emph{spread} of $B$ is
\[
\mathrm{sp}(B)=\max_{u\in B}\deg _{G}(u)-\min_{v\in B}\deg _{G}(v),
\]
and for an integer $k\geq 0$ we set
\[
\mathrm{sp}(G,k)=\max \{\left\vert B\right\vert :B\subseteq V(G)\text{ and }\mathrm{sp}
(B)\leq k\}.
\]
Thus $\mathrm{sp}(G,0)=\mathrm{rep}(G)$, and the first elementary result of graph theory reads
$\mathrm{sp}(G,0)\geq 2$ for every graph of order at least two. Erd\H{o}s, Chen, Rousseau and
Schelp proved that $\mathrm{sp}(G,k)\geq k+2$ whenever $G$ has at least $k+2$ vertices, and a short
proof avoiding the Erd\H{o}s--Gallai theorem is given in \cite{CaroLauriZarb}.

Notice that, writing $n_{j}$ for the number of vertices of degree $j$ in $G$, we have
\begin{equation}
\mathrm{sp}(G,k)=\max_{p\geq 0}\ (n_{p}+n_{p+1}+\cdots +n_{p+k}),  \label{For_window}%
\end{equation}

so that $\mathrm{sp}(G,k)$ depends on the degree sequence of $G$ only. We shall use
(\ref{For_window}) throughout, and we call the sets of $k+1$ consecutive degrees occurring in
(\ref{For_window}) \emph{windows}.

It is worth pausing over what (\ref{For_window}) measures. Reading the degree sequence of $G$ as a
distribution of $n$ units of mass over the integers, $\mathrm{sp}(G,k)$ is the largest mass carried by
$k+1$ consecutive values; a graph with a small value for $\mathrm{sp}(G,k)$ is one whose degrees are genuinely spread out,
while one with large $\mathrm{sp}(G,k)$ has a bulk of vertices of nearly equal degree. The
parameter increases with $k$, it equals $n$ as soon as $\Delta (G)-\delta (G)\leq k$, and it satisfies
$\mathrm{sp}(G,k)\leq (k+1)\,\mathrm{sp}(G,0)$ and
$\mathrm{sp}(G,k)\geq n/\left\lceil (\Delta -\delta +1)/(k+1)\right\rceil $, the latter because
$\left\lceil (\Delta -\delta +1)/(k+1)\right\rceil $ windows already cover all the degrees; the first of
these two inequalities and the invariance $\mathrm{sp}(G,k)=\mathrm{sp}(\overline{G},k)$ under
complementation are noted in \cite{CaroLauriZarb}. Every one of these estimates uses only the crudest
information about $G$, and all of them are weak: the interest of the parameter lies in what happens
inside a class of graphs whose structure constrains the degree sequence, and the constraint that matters
is on the degree sum.


A graph is \emph{maximal planar} if it is planar and the addition of any edge destroys planarity;
equivalently, for order $n\geq 3$, it is a triangulation of the sphere, with $3n-6$ edges and $2n-4$
triangular faces, so that
\[
\sum_{v\in V(G)}\deg (v)=2e(G)=n\,d(G)=6n-12 ,
\]
where $d(G)$ is the average degree; and for $n\geq 4$ the minimum degree is $3$, $4$ or $5$. It is this class that Caro, Lauri and Zarb proposed for the spread. Their paper
\cite{CaroLauriZarb} gives a lower bound on $\mathrm{sp}(G,k)$ in terms of the order, the average
degree, the minimum degree and the maximum degree of $G$ \cite[Theorem 2.1]{CaroLauriZarb}, together
with a sharper form, the inequality (4) of that paper, carrying a free integer parameter $t$; for
$k=0$ it is the bound of Caro and West. They then test the bound on two classes whose average degree
is determined by the order. For a tree they obtain $\mathrm{rep}(T)\geq \left\lceil n/3\right\rceil $,
which is the bound of \cite{CaroWest}, and $\mathrm{sp}(T,k)\geq (nk+2)/(k+1)$ for $k\geq 1$, both
sharp. For maximal outerplanar graphs they introduce
\[
\mathrm{MOP}(n,k)=\min \{\mathrm{sp}(G,k):G\text{ a maximal outerplanar graph of order }n\}
\]
and prove $\mathrm{MOP}(n,0)>n/5$, $\mathrm{MOP}(n,1)\geq n/3+1$, $\mathrm{MOP}(n,2)\geq 4n/9$ and
$\mathrm{MOP}(n,k)\geq (k-2)n/(k-1)$ for $k\geq 3$, the last by feeding into their count a bound of Jao
and West \cite{JaoWest} on the number of vertices of large degree; a family taken from
\cite{JaoWest} shows that last bound sharp up to an additive constant, a construction of their own does
the same for $k=1$, and for $k=0$ the construction of \cite{CaroWest} already does. The single
exception is $k=2$, where their best family gives only
$\mathrm{sp}(G,2)=(5n+19)/11$ for $n\equiv 5\ (\mathrm{mod}\ 11)$, and Problem 1 of
\cite{CaroLauriZarb} asks for the correct order of magnitude of $\mathrm{MOP}(n,2)$. That problem is
settled in the companion paper \cite{P1}, where we first prove a lower bound on $\mathrm{sp}(G,k)$ for
an arbitrary graph in terms of its order, its size and its minimum degree, containing the bounds of
\cite{CaroLauriZarb,CaroWest} and refining them by an explicit error term, and then obtain
$\mathrm{MOP}(n,2)=4n/9+O(1)$. Problem 2 of \cite{CaroLauriZarb} asks for the analogous quantity for
maximal planar graphs, and that problem is the subject of the present expository paper.

When preparing this paper we realised that, while some of the basic ideas are elaborations of seed ideas
already in \cite{CaroLauriZarb} and \cite{CaroWest}, several of the proofs of the lemmas and theorems are
rather technical and take some effort to follow. We therefore decided to incorporate many drawings, remarks,
examples and explanations along the way, in the hope that they will support a smoother reading. To the same
end we have put the companion paper \cite{P1} separately on the arXiv, where the window counting method is
developed in detail, its usefulness in settling the missing case $\mathrm{MOP}(n,2)$ of \cite{CaroLauriZarb}
is shown, and its value in bounding $\mathrm{sp}(G,k)$ from below in terms of $n$, $\delta $ and the average
degree $d=2m/n$ is set out. Our main objective here is to solve $\mathrm{MP}(n,\delta ,k)$ asymptotically,
as proposed in \cite{CaroLauriZarb}, mostly by that method together with a gallery of constructions which
elaborate and extend those of \cite{CaroLauriZarb} and \cite{CaroWest}. Not all pairs $(\delta ,k)$ can be
covered by one type of construction, and we have also included many examples and a good deal of information
about the computations at small $n$, to illustrate the techniques in use --- aware, in doing so, of what this
contributes to the length of the paper.


We therefore consider, for $\delta \in \{3,4,5\}$,
\[
\mathrm{MP}(n,\delta ,k)=\min \{\mathrm{sp}(G,k):G\text{ maximal planar of order }n\text{ with }\delta (G)=\delta \}.
\]

Since the minimum degree of a maximal planar graph on $n\geq 4$ vertices is a member of $\{3,4,5\}$, it follows
that the answer to the problem of determining $\mathrm{MP}(n,\delta ,k)$ is dominated strongly by the
parameter $\delta $, and we split the solution accordingly.

\begin{theorem}
\label{Tm_intro}For every $\delta \in \{3,4,5\}$ and every $k\geq 0$ the limit
\[
c_{\delta ,k}=\lim_{n\rightarrow \infty }\frac{\mathrm{MP}(n,\delta ,k)}{n}
\]
exists, and in fact $\mathrm{MP}(n,\delta ,k)=c_{\delta ,k}\,n+O(\sqrt{n})$ for every sufficiently large
$n$. Its value is
\[
\begin{array}
[c]{ll}%
c_{5,0}=\dfrac{1}{3}, & c_{5,k}=\dfrac{k}{k+1}\ \ (k\geq 1),\\[3mm]
c_{4,0}=\dfrac{1}{5},\ \ c_{4,1}=\dfrac{1}{3},\ \ c_{4,2}=\dfrac{4}{9}, & c_{4,k}=\dfrac{k-1}{k+1}\ \
(k\geq 3),\\[3mm]
c_{3,0}=\dfrac{1}{7},\ \ c_{3,1}=\dfrac{1}{4},\ \ c_{3,2}=\dfrac{1}{3},\ \ c_{3,3}=\dfrac{5}{12},\ \
c_{3,4}=\dfrac{7}{15}, & c_{3,k}=\dfrac{k-2}{k+1}\ \ (5\leq k\leq 8),\\[3mm]
& c_{3,k}=\dfrac{k-4}{k-2}\ \ (k\geq 9).
\end{array}
\]
For every pair except those of the last line, $c_{\delta ,k}$ is the coefficient of $n$ in the window
count of \cite{P1} maximised over its free parameter, and the lower bound is an inequality valid at
every order (Theorems \ref{Tm_planar} and \ref{Tm_planar34}). For $\delta =3$ and $k\geq 9$ that count
is not optimal; there the lower bound comes instead from the excess inequality
$\sum \{\deg (v)-6:\deg (v)\geq 7\}=\sum_{v}(\deg (v)-6)^{+}\leq 2n-14$, in which $x^{+}$ denotes
$\max \{x,0\}$, so that only the vertices of degree at least seven contribute; that bound too is valid at
every order (Lemma \ref{Lemma_excess} and Theorem \ref{Tm_deg3exact}). In every case the matching upper
bound is furnished by an explicit periodic triangulation of the plane, cut to a large disk and closed by a
cap and an apex vertex --- the cap by itself does not make the result maximal planar, since the outer face
has to be a triangle too.
\end{theorem}


Theorem \ref{Tm_intro} is Theorem \ref{Tm_asym} together with Theorem \ref{Tm_deg3exact} and
Proposition \ref{Prop_deg3}, and it is proved in Section \ref{Sec_asym}. Two ideas carry it, and it is
worth describing both before any formalism.

The first is a \emph{window count}. Fix $t\geq 1$ and lay the windows out from the minimum degree
upwards: $I_{0}=\{\delta ,\ldots ,\delta +k\}$, then $I_{1}$ the next $k+1$ integers, and so on. Give a
vertex the weight $t$ if its degree lies in $I_{0}$, the weight $t-1$ if it lies in $I_{1}$, and so on
down to the weight $1$ on $I_{t-1}$ and the weight $0$ above. The total weight can be estimated from two
sides. Each window holds at most $s=\mathrm{sp}(G,k)$ vertices, so the total weight is at most
$s\left( t+(t-1)+\cdots +1\right) =s\,t(t+1)/2$. On the other hand a vertex of degree $x$ has weight at
least $\left( \delta +t(k+1)-x\right) /(k+1)$: a vertex of small degree sits in a low window and carries
much weight, and a vertex of degree $\delta +t(k+1)$ or more carries none but is not needed. Summing that
over all vertices and using $\sum_{v}\deg (v)=2m$ bounds the total weight from below by
$\left( n(\delta +t(k+1))-2m\right) /(k+1)$. Comparing the two estimates and solving for $s$ gives the
bound (\ref{For_clz4}) contained in Theorem \ref{Tm_gen}, and keeping track of the slack in the second
estimate gives Theorem \ref{Tm_gen} in full. The bound has the form it has because the cheapest way to keep all windows thin is to use
one degree per window and to place it as low as possible, that is to use the degrees $\delta ,\delta
+(k+1),\ldots ,\delta +t(k+1)$ with $s$ vertices at each of the first $t$ of them; such a sequence has
degree sum roughly $n\delta +(k+1)\,n^{2}/(2s)$, so that a small $s$ demands a large degree sum. For a
maximal planar graph the degree sum is $6n-12$ and nothing more can be demanded, which turns the inequality
into a lower bound on $s$ linear in $n$. The parameter $t$ is free, and one takes the best of these
bounds; Theorem \ref{Tm_sequence} says that the result, the maximised bound, is exactly the minimum of
the window number over all integer sequences with the given order, size and minimum degree, so that the
only thing the count can lose is graphicality. For $\delta =5$ the picture is at its simplest: with
$t=1$ the count gives $\mathrm{sp}(G,k)\geq (kn+12)/(k+1)$, that is, at most $(n-12)/(k+1)$ vertices of
a maximal planar graph of minimum degree five can escape a single window of $k+1$ consecutive degrees.

The second idea is a \emph{periodic construction}, and it answers the question which the first leaves
open: whether the extremal degree sequence is really the degree sequence of a maximal planar graph. The
difficulty is not the interior but the boundary, and it is entirely a matter of Euler's formula. If
every face of a triangulation of a closed surface of Euler characteristic $\chi $ is a triangle, then
counting each face's three sides once from each of the two faces along it gives $3f=2e$, and $n-e+f=\chi $
then forces
\[
e=3n-3\chi ,\qquad f=2n-2\chi ,\qquad \sum_{v}\deg (v)=2e=6n-6\chi .
\]
Two values of $\chi $ matter here. On the \emph{sphere} $\chi =2$: a maximal planar graph on $n\geq 3$
vertices has $3n-6$ edges and $\sum_{v}\deg (v)=6n-12$, that is $\sum_{v}(6-\deg (v))=12$. On the
\emph{torus} $\chi =0$: a triangulation has $3n$ edges and $\sum_{v}\deg (v)=6n$, that is
$\sum_{v}(6-\deg (v))=0$. For the same number of vertices the sphere therefore has exactly
$3\chi =6$ edges fewer than the torus, and a degree sum smaller by twelve. Figure \ref{Fig_euler} shows the
two smallest examples of each.

Those twelve units are the whole difficulty. On the torus nothing has to be paid: the degrees may average
exactly six, so a pattern with prescribed degree densities can repeat forever, and it is a finite object
which one can design by hand or search for by computer. On the sphere the twelve must be paid somewhere. So
one designs the pattern on the torus --- equivalently, as a doubly periodic pattern in the plane --- cuts out
a large disk, and pays the whole twelve along the boundary of that disk: the disk is closed by a cap and an
apex vertex. (The six edges are lost in the cutting and the capping, not by deleting six edges from a torus
triangulation, which would leave faces that are not triangles; the operation itself is developed in Lemmas
\ref{Lemma_closure} and \ref{Lemma_periodic}, and it changes the number of vertices as well.) 

\begin{figure}[htbp]
\begin{center}
\includegraphics[width=0.34\textwidth]{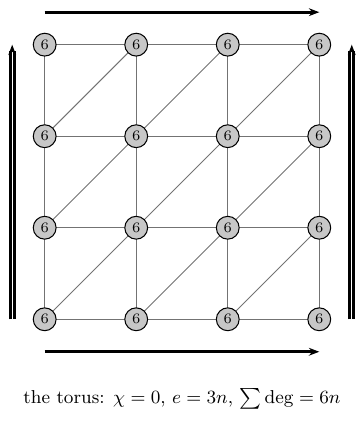}\hfill
\includegraphics[width=0.46\textwidth]{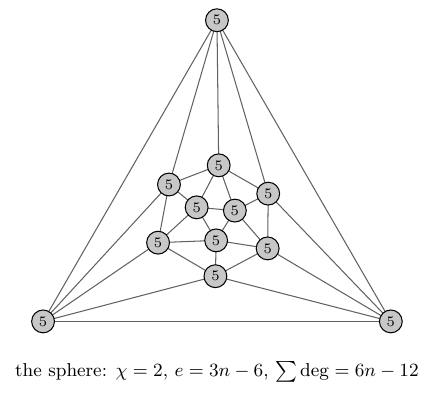}
\end{center}
\caption{The two counts. Left: the quotient of the triangular lattice by the translations $(3,0)$ and
$(0,3)$, a triangulation of the torus on $n=9$ vertices with $27=3n$ edges, drawn in a square whose opposite
sides are identified as the arrows show --- the sixteen discs of the picture are nine vertices, the four
corners being one of them --- and every degree is exactly six, so $\sum_{v}(6-\deg (v))=0$. Right: the
icosahedron, a maximal planar graph on $n=12$ vertices with $30=3n-6$ edges, every degree five, so
that the twelve units of $\sum_{v}(6-\deg (v))=12$ sit one on each vertex.}%
\label{Fig_euler}%
\end{figure}
The cap is a triangulated annulus of
nested rings glued along the boundary cycle, which absorbs the curvature and whose vertices are the only ones
whose degrees are not those of the pattern; the apex is one further vertex joined to the whole of the
innermost ring. The apex is not a decoration: without it the innermost ring would still bound a face that is
not a triangle, and the result would not be maximal planar.
A disk containing about $n$ vertices has a boundary of about $\sqrt{n}$ of them, and the cap can be
given any prescribed number of vertices above a threshold, so that the order comes out exactly $n$; the
degrees of all but $O(\sqrt{n})$ vertices are those of the pattern, and the window densities of the
pattern become the window densities of the sphere up to $O(\sqrt{n})$. That is the content of Lemmas
\ref{Lemma_closure} and \ref{Lemma_periodic}, and it is where the error term $O(\sqrt{n})$ of Theorem
\ref{Tm_intro} comes from. The asymptotic problem is thereby reduced to a question with no $n$ in it:
for which $(\delta ,k)$ does a periodic triangulation of the plane with minimum degree $\delta $ and the
extremal degree densities exist? For $\delta \in \{4,5\}$ and every $k$, and for $\delta =3$ with
$k\leq 8$, one does, and Section \ref{Sec_asym} exhibits it. For $\delta =3$ and $k\geq 9$ none does,
for a reason that has nothing to do with counting windows: too many vertices of degree three would force
too many vertices of large degree, and in a maximal planar graph the vertices of large degree are few, because
the faces they would need are few. Lemma \ref{Lemma_excess} makes that precise, and the coefficient
$\frac{k-4}{k-2}$ is what it yields.


The rest of the paper is organised as follows, and since it is longer than the result requires we say
also what a reader in a hurry may pass over.

Section \ref{Sec_prelim} recalls from the companion paper \cite{P1} the window count and the analysis of
its equality cases, in the notation of that paper and with no proof repeated. A reader willing to take
those results on trust needs only the statements of Theorem \ref{Tm_gen} and Proposition
\ref{Prop_equality}: the first is the bound, the second says that equality forces the degrees
$\delta ,\delta +h,\ldots ,\delta +th$ with prescribed multiplicities, and everything in this paper that
asks what an extremal maximal planar graph must look like appeals to it. Lemma \ref{Lemma_gen}, Lemma
\ref{Lemma_bestt} and Theorem \ref{Tm_sequence} may be skipped on a first pass; they say respectively
what the slack in the count is, which value of the free parameter is the best one, and that the
maximised count is the exact minimum of the window number over all integer sequences with the given
order, size and minimum degree.

Section \ref{Sec_planar} works the count out for maximal planar graphs. Section \ref{Sec_planar5} treats
$\delta =5$, which is the case to read first: the count with $t=1$ gives
$\mathrm{MP}(n,5,k)\geq \lceil (kn+12)/(k+1)\rceil $ for $k\geq 1$ and
$\mathrm{MP}(n,5,0)\geq \lceil (n+12)/3\rceil $, and it names the one degree sequence with which
equality can hold (Theorem \ref{Tm_planar}). Section \ref{Sec_hex} is an aside, and the one place in Section \ref{Sec_planar} where a construction,
and with it an upper bound, appears: it improves the remainder in the asymptotic of Caro and West
\cite{CaroWest} for the repetition number of maximal planar graphs of minimum degree five from $o(n)$ to
$O(\sqrt{n})$, by an explicit family of orders $n=3L^{2}+3L+7$, obtained by
flipping a periodic set of rhombi in a hexagonal patch of the triangular lattice and closing it by six
vertices, with $\mathrm{sp}(G,0)=n/3+2\sqrt{n/3}+O(1)$ and a completely described boundary, together
with the exact values for $n\equiv 0\ (\mathrm{mod}\ 3)$, $24\leq n\leq 48$. Nothing later depends on
that section, and a reader interested only in the asymptotics may skip it. Section \ref{Sec_planar34} treats $\delta =3$ and $\delta =4$, where the free parameter
must genuinely be optimised and the statement of Theorem \ref{Tm_planar34} is therefore a list of cases;
its proof is arithmetic, and the reader may take the list as given. The section ends with presenting the exact values of $\mathrm{MP}(n,\delta ,k)$ for $n\leq 14$ by
exhaustive enumeration of maximal planar graphs, and the orders $n\leq 30$ at which a maximal planar graph attaining
the maximised bound was found by a flip search.

Section \ref{Sec_asym} is the heart of the paper and determines $\lim \mathrm{MP}(n,\delta ,k)/n$ for
every pair. It opens with the passage from the plane to the sphere: Lemma \ref{Lemma_closure}, which
closes a triangulated disk by a cap of any prescribed size without disturbing the interior degrees, and Lemma \ref{Lemma_periodic}, which cuts a disk out of a periodic triangulation and applies the closure.
The closure of Lemma \ref{Lemma_closure} ends with an \emph{apex} vertex joined to the innermost ring, and it
is that last vertex which makes the result a maximal planar graph rather than a
disk with an $m$-gonal hole.
The proofs are careful but unsurprising bookkeeping about boundaries; what is used later is only the
statement of Lemma \ref{Lemma_periodic}, and a first reading may stop there. Section \ref{Sec_apex} then
says in detail what the closure does and where the apex vertex sits, with a worked example small enough
to check by hand; it proves nothing and may be skipped by a reader who finds the closure obvious. The
three subsections after it supply the patterns: six of them by hand (Proposition \ref{Prop_hand}), then one family
settling $\delta =4$ for every $k$ by inserting a chain of vertices into every cell of the square grid
(Proposition \ref{Prop_chain4}), then one settling $\delta =5$ for every $k$, whose vertices of large
degree are the hubs of wheels $W_{k+6}$ glued along their rims (Proposition \ref{Prop_wheel}). Section
\ref{Sec_computer} disposes of the seven remaining pairs by triangulations of the torus found by
computer; the point of that section is the format, which is designed so that a machine-found example can
be verified by hand from five printed conditions. Section \ref{Sec_deg3exact} is the one place where a
genuinely planar argument replaces a degree count, and a reader short of time should read it in
preference to the constructions: for $\delta =3$ and $k\geq 9$ the maximised bound is not the optimum. The
bound that is the optimum comes from the excess inequality $\sum_{v}(\deg (v)-6)^{+}\leq 2n-14$ of Lemma
\ref{Lemma_excess}, and matching
periodic triangulations with the three degrees $3$, $4$ and $k+4$ (Proposition \ref{Prop_deg3}) show
that its coefficient $\frac{k-4}{k-2}$ cannot be improved. Unlike the asymptotic statements around it,
the bound of Theorem \ref{Tm_deg3exact} is an inequality valid at every order $n\geq 7$, and below the
range of orders in which it can bite Proposition \ref{Prop_deg3small} determines
$\mathrm{MP}(n,3,k)$ completely. 
The final section collects the whole picture in Theorem \ref{Tm_asym} and reviews what has been obtained and what has not. Asymptotically no pair
$(\delta ,k)$ is left open; what remains open at every pair is the exact value of
$\mathrm{MP}(n,\delta ,k)$ and the best possible error term.


\section{The lower bound of the companion paper}
\label{Sec_prelim}

The bound that Section \ref{Sec_planar} applies is a counting bound, and it is
worth saying in plain words what it counts before the notation begins. Suppose
that no $k+1$ consecutive degrees are shared by more than $s$ vertices. Then at
most $s$ vertices have their degree in $\{\delta ,\ldots ,\delta +k\}$, at most
$s$ further vertices have it in the next $k+1$ values, and so on; a graph with
few vertices in every window is therefore forced to have many vertices of large
degree, and hence a large degree sum. Read backwards, this says that a graph
whose degree sum is small must have a large window count --- and a maximal planar graph has the smallest degree sum that Euler's formula permits, namely
$6n-12$. All the argument needs is a device for converting ``at most $s$ in each
window'' into a single inequality about $\sum_{v}\deg (v)$, and that device is
the weight $u_{t}$ defined below, which charges each vertex according to how far
its degree lies below the level $\delta +t(k+1)$, together with the correction
$\varepsilon _{t}$, which records what is lost by charging whole windows instead
of individual degrees.

Section \ref{Sec_planar} uses three things from what follows: the inequality
itself (Theorem \ref{Tm_gen}), the rule for choosing the parameter $t$ (Lemma
\ref{Lemma_bestt}), and the description of the degree sequences at which the
inequality is tight (Proposition \ref{Prop_equality}), since those are the
sequences one must then try to realise by a maximal planar graph.


We recall the results of \cite{P1} that we use, in the notation of that paper; the proofs are there. Throughout
this section $G$ is an arbitrary graph of order $n$ with $m$ edges, minimum degree $\delta =\delta (G)$, maximum
degree $\Delta =\Delta (G)$ and average degree $d=d(G)=2m/n$. We fix $k\geq 0$ and write $h=k+1$ for the width of a
window. The \emph{aligned windows} are
\[
I_{i}=\{\delta +ih,\ \delta +ih+1,\ \ldots ,\ \delta +ih+k\},\qquad W_{i}=\sum_{j\in I_{i}}n_{j}%
\qquad (i\geq 0),
\]
so that $\sum_{i\geq 0}W_{i}=n$ and, by (\ref{For_window}), $W_{i}\leq \mathrm{sp}(G,k)$ for every
$i$. For an integer $t\geq 1$ put
\[
u_{t}(x)=\sum_{i=0}^{t-1}(t-i)\cdot 1\{x\in I_{i}\},\qquad \varepsilon _{t}(x)=h\,u_{t}%
(x)-\left( \delta +th-x\right) .
\]

We unpack the two definitions. The aligned windows are obtained by starting at
the minimum degree and marching upwards in steps of $h$: $I_{0}$ consists of the
$h$ lowest degrees that can occur at all, $I_{1}$ of the next $h$ values, and so
on, so that the windows partition the possible degrees and every vertex is
counted by exactly one $W_{i}$. These are not the windows that compute
$\mathrm{sp}(G,k)$ in (\ref{For_window}); a window realising the maximum there
may straddle two aligned windows, and usually does. What matters is that each
$W_{i}$ is the count of \emph{some} window, so that $W_{i}\leq \mathrm{sp}(G,k)$,
while at the same time $\sum_{i}W_{i}=n$. A family of windows that is
simultaneously exhaustive and individually bounded is exactly what a counting
argument needs.

The weight $u_{t}(x)$ is easiest to read off the index of the window containing
$x$: if $x\in I_{i}$ with $i<t$ then $u_{t}(x)=t-i$, the number of the windows
$I_{i},I_{i+1},\ldots ,I_{t-1}$ lying at or above the window of $x$, while
$u_{t}(x)=0$ for $x\geq \delta +th$. Thus $h\,u_{t}(x)$ is the distance from the
\emph{bottom} of the window of $x$ up to the level $\delta +th$, whereas
$\delta +th-x$ is the distance from $x$ itself; the difference
$\varepsilon _{t}(x)$ is therefore the height of $x$ above the largest of the
levels $\delta ,\delta +h,\ldots ,\delta +th$ that does not exceed $x$. In
particular $\varepsilon _{t}$ vanishes precisely on those levels, which is the
content of Lemma \ref{Lemma_gen}, and the sum $\sum_{v}\varepsilon _{t}(\deg
(v))$ appearing in Theorem \ref{Tm_gen} measures exactly the extent to which the
degrees of $G$ miss them.


\begin{lemma}
\label{Lemma_gen}For every integer $x\geq \delta $,
\[
\varepsilon _{t}(x)=\left\{
\begin{array}
[c]{ll}%
(x-\delta )\ \mathrm{mod}\ h, & \text{if }\delta \leq x<\delta +th,\\
x-\delta -th, & \text{if }x\geq \delta +th.
\end{array}
\right.
\]
In particular $\varepsilon _{t}(x)\geq 0$, that is $h\,u_{t}(x)\geq \delta +th-x$, with equality if
and only if $x=\delta +ih$ for some $0\leq i\leq t$.
\end{lemma}

\begin{example}
\label{Ex_eps}Take $\delta =3$ and $k=2$, so that $h=3$, and let $t=2$. The two
windows carrying weight are $I_{0}=\{3,4,5\}$ and $I_{1}=\{6,7,8\}$, the levels
are $3$, $6$ and $9$, and the two functions are
\[
\begin{array}
[c]{c|cccccccccc}%
x & 3 & 4 & 5 & 6 & 7 & 8 & 9 & 10 & 11 & 12\\\hline
u_{2}(x) & 2 & 2 & 2 & 1 & 1 & 1 & 0 & 0 & 0 & 0\\
\varepsilon _{2}(x) & 0 & 1 & 2 & 0 & 1 & 2 & 0 & 1 & 2 & 3
\end{array}
\]
What the reader should notice is the sawtooth: below the level $\delta +th=9$
the function $\varepsilon _{2}$ repeats the pattern $0,1,2$ inside each window
and so never exceeds $h-1=2$, whereas from $9$ upwards it grows without bound.
A single vertex of very large degree therefore contributes as much to
$\sum_{v}\varepsilon _{t}(\deg (v))$ as a whole handful of vertices of middling
degree, which is why Theorem \ref{Tm_gen} is strong for graphs with a few
dominant vertices, while its error term disappears entirely for a graph all of
whose degrees sit on the levels.
\end{example}


\begin{theorem}
\label{Tm_gen}Let $G$ be a graph of order $n$ with $m$ edges and minimum degree $\delta $, and let
$k\geq 0$ and $t\geq 1$. Then
\[
\mathrm{sp}(G,k)\ \geq \ \frac{2\left( n\left( \delta +t(k+1)\right) -2m+\sum_{v\in V(G)}%
\varepsilon _{t}(\deg (v))\right) }{(k+1)\,t(t+1)} .
\]
\end{theorem}

Since $\varepsilon _{t}\geq 0$, Theorem \ref{Tm_gen} contains the bound
\begin{equation}
\mathrm{sp}(G,k)\ \geq \ \frac{2n\left( \delta +t(k+1)-d\right) }{(k+1)\,t(t+1)},
\label{For_clz4}%
\end{equation}
which is the bound (4) of \cite{CaroLauriZarb}. We write
\[
f_{t}=f_{t}(n,m,\delta ,k)=\frac{2\left( n(\delta +t(k+1))-2m\right) }{(k+1)\,t(t+1)}=\frac{2n(t-x)}%
{t(t+1)},\qquad x=\frac{d-\delta }{k+1}\geq 0,
\]
for the right-hand side of (\ref{For_clz4}). Theorem \ref{Tm_gen} holds for every $t\geq 1$, so that one may take
the maximum over $t$.

\begin{example}
\label{Ex_Tmgen}Let $G$ be the maximal planar graph on $v_{1},\ldots ,v_{10}$ drawn in
Figure \ref{Fig_Tmgen}; its sixteen faces are the sixteen triangles of that drawing, the outer one being
$v_{5}v_{6}v_{8}$.

\begin{center}
\includegraphics[width=0.58\textwidth]{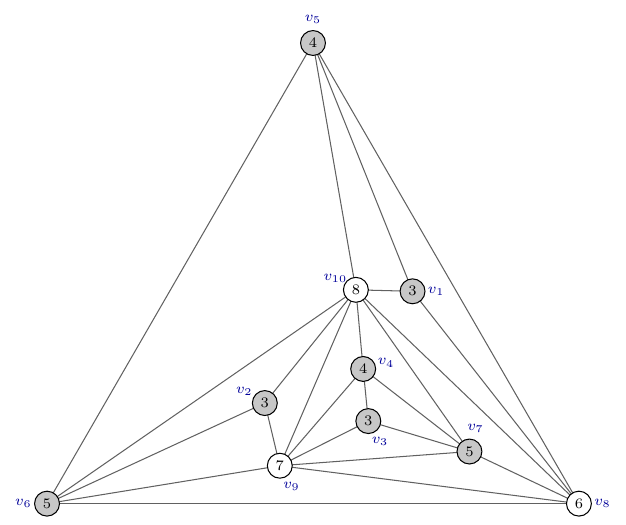}
\captionof{figure}{The triangulation of Example \ref{Ex_Tmgen}. Each vertex carries its degree, with its
name beside it; the seven grey vertices, of degrees $3$, $4$ and $5$, are those counted by the window
$\{3,4,5\}$, so that $W_{0}=\mathrm{sp}(G,2)=7$. The outer face is $v_{5}v_{6}v_{8}$.}%
\label{Fig_Tmgen}%
\end{center}

Its degree sequence is $(3,3,3,4,4,5,5,6,7,8)$, the vertices being listed in
order of increasing degree, so $n=10$, $\delta =3$ and
$\sum_{v}\deg (v)=48=6n-12$. Take $k=2$, hence $h=3$. The windows of
(\ref{For_window}) contain $7,5,4,3,2,1$ vertices for $p=3,4,\ldots ,8$, so that
$\mathrm{sp}(G,2)=7$, realised by the window $\{3,4,5\}$.

With $t=2$ the aligned windows are $I_{0}=\{3,4,5\}$, $I_{1}=\{6,7,8\}$ and
$I_{2}=\{9,10,11\}$, with $W_{0}=7$, $W_{1}=3$ and $W_{2}=0$, and the
contributions of the vertices are
\[
\begin{array}
[c]{c|cccccc}%
x & 3 & 4 & 5 & 6 & 7 & 8\\\hline
n_{x} & 3 & 2 & 2 & 1 & 1 & 1\\
u_{2}(x) & 2 & 2 & 2 & 1 & 1 & 1\\
\varepsilon _{2}(x) & 0 & 1 & 2 & 0 & 1 & 2\\
n_{x}\varepsilon _{2}(x) & 0 & 2 & 4 & 0 & 1 & 2
\end{array}
\]
so that $\sum_{v}\varepsilon _{2}(\deg (v))=9$ and Theorem \ref{Tm_gen} reads
\[
\mathrm{sp}(G,2)\ \geq \ \frac{2\left( 10\cdot 9-48+9\right) }{3\cdot 2\cdot 3}%
=\frac{17}{3}=5.66\ldots ,
\]
whereas (\ref{For_clz4}) by itself gives only $f_{2}=14/3=4.66\ldots $.

With $t=1$ the outcome is better. Now only $I_{0}$ carries weight, $u_{1}$
equals $1$ on $\{3,4,5\}$ and $0$ elsewhere, the values of $\varepsilon _{1}$ at
$3,\ldots ,8$ are again $0,1,2,0,1,2$, their weighted sum is again $9$, and
\[
\mathrm{sp}(G,2)\ \geq \ \frac{2\left( 10\cdot 6-48+9\right) }{3\cdot 1\cdot 2}%
=7,
\]
which is the exact value, although $f_{1}=4$ is smaller than $f_{2}=14/3$. Two
things are worth noticing. First, the error term is not a technicality: without
it the bound is $14/3$ and with it the bound is exact. Second, Lemma
\ref{Lemma_bestt} below governs the choice of $t$ in (\ref{For_clz4}) only; for the
sharper inequality of Theorem \ref{Tm_gen} the best $t$ may be smaller, so for a
concrete graph one should try all small values of $t$. The reason equality holds
here is visible in Proposition \ref{Prop_equality}(i) below: for $t=1$ the only
condition is $W_{0}=\mathrm{sp}(G,2)$, and indeed $W_{0}=7$.
\end{example}

Figure \ref{Fig_window} shows the same example as a window diagram, which is the picture to keep in
mind for the whole of this section.

\begin{figure}[htbp]
\begin{center}
\includegraphics[width=0.86\textwidth]{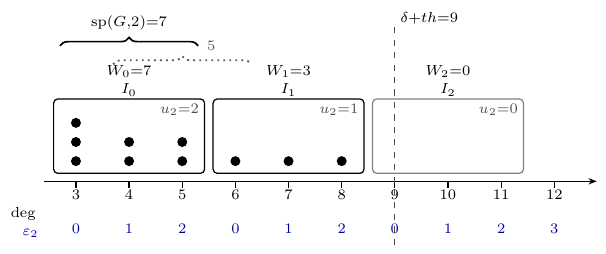}
\end{center}
\caption{The window diagram of Example \ref{Ex_Tmgen} for $k=2$ and $t=2$: the degree histogram, the aligned
windows $I_{0},I_{1},I_{2}$ with their weights $u_{2}$ and their counts $W_{i}$, and the sawtooth
$\varepsilon _{2}$ below the axis, which restarts at $\delta +th=9$ and grows without bound beyond it. The
dotted bracket shows the non-aligned window $\{4,5,6\}$, which catches only five vertices.}%
\label{Fig_window}%
\end{figure}

\begin{lemma}
\label{Lemma_bestt}For real $x\geq 0$ and integer $t\geq 1$ we have $f_{t}\geq f_{t+1}$ if and only
if $t\geq 2x$, with equality if and only if $t=2x$. Hence $\max_{t\geq 1}f_{t}$ is attained at
$t=\max \{1,\left\lceil 2x\right\rceil \}$, also at $t=2x+1$ when $2x$ is an integer, and at no
other $t$.
\end{lemma}

\begin{example}
\label{Ex_bestt}For a maximal planar graph with $\delta =3$ and $k=0$ we have
$h=1$ and $x=d-3=3-12/n$, so that $2x=6-24/n$. This is smaller than $6$ for
every $n$; it equals $4$ at $n=12$ and $5$ at $n=24$, and it is increasing. By
Lemma \ref{Lemma_bestt} the optimal $t$ is therefore $5$ for $13\leq n\leq 23$,
is both $5$ and $6$ at $n=24$, where $2x=5$ is an integer, and is $6$ for every
$n>24$; at $n=12$, where $2x=4$, both $t=4$ and $t=5$ are optimal, and for
$n<12$ the optimal $t$ is smaller still. The corresponding values are
$f_{5}=(2n+12)/15$ and $f_{6}=(n+4)/7$, and indeed
\[
f_{5}(20)=\tfrac{52}{15}>f_{6}(20)=\tfrac{24}{7},\qquad f_{5}(24)=f_{6}%
(24)=4,\qquad f_{5}(25)=\tfrac{62}{15}<f_{6}(25)=\tfrac{29}{7}.
\]
This is where the threshold $n\geq 24$ in Theorem \ref{Tm_planar34}(i) comes
from. Since $2x$ never reaches $6$, the value $t=7$ is never optimal; the best
$t$ is eventually constant, and this is the general phenomenon exploited in
Section \ref{Sec_planar34}: for a fixed pair $(\delta ,k)$ only finitely many
values of $t$ ever occur, and exactly one of them occurs for all large $n$.
\end{example}


The proof of Theorem \ref{Tm_gen} uses only the degree sequence of $G$. We therefore call a
nondecreasing sequence $D=(d_{1}\leq \cdots \leq d_{n})$ of integers with $d_{1}=\delta $ a
\emph{$\delta $-sequence}, and we call the largest number of its terms lying in an interval of $k+1$
consecutive integers its \emph{window number} $\mathrm{sp}(D,k)$; thus $\mathrm{sp}(G,k)=\mathrm{sp}%
(D,k)$ for the degree sequence $D$ of $G$, and Theorem \ref{Tm_gen} holds, with the same proof, for
every $\delta $-sequence with sum $2m$.

\begin{proposition}
\label{Prop_equality}Let $G$ be a graph of order $n$ with $m$ edges and minimum degree $\delta $, let
$k\geq 0$, $t\geq 1$, $h=k+1$ and $s=\mathrm{sp}(G,k)$.

\begin{itemize}
\item[\textrm{(i)}] Equality holds in Theorem \ref{Tm_gen} if and only if $W_{0}=W_{1}=\cdots
=W_{t-1}=s$.

\item[\textrm{(ii)}] Equality holds in (\ref{For_clz4}), that is $s=f_{t}$, if and only if every
degree of $G$ lies in $\{\delta ,\delta +h,\ldots ,\delta +th\}$ and
\[
n_{\delta }=n_{\delta +h}=\cdots =n_{\delta +(t-1)h}=s,\qquad n_{\delta +th}=n-ts\leq s.
\]
In particular equality forces $s=f_{t}$ to be an integer with $n/(t+1)\leq s\leq n/t$, and
$t-1\leq 2x\leq t$.
\end{itemize}
\end{proposition}

\begin{example}
\label{Ex_equality}The two parts of Proposition \ref{Prop_equality} really do
say different things, and one graph exhibits both. Let $G$ be the maximal planar graph
of order $24$ with degree sequence $5^{12}6^{12}$, the dual of the fullerene
with twelve hexagons (Figure \ref{Fig_MP0b}, top).

Take $k=0$ and $t=2$. Then $h=1$, $I_{0}=\{5\}$, $I_{1}=\{6\}$ and
$W_{0}=W_{1}=12=\mathrm{sp}(G,0)$, so (i) gives equality in Theorem
\ref{Tm_gen}. Here also every degree lies in
$\{\delta ,\delta +h,\delta +2h\}=\{5,6,7\}$ with $n_{5}=n_{6}=12=s$ and
$n_{7}=0=n-ts$, so (ii) applies too and $\mathrm{sp}(G,0)=f_{2}=12$: the plain
bound (\ref{For_clz4}) is already exact.

Take instead $k=1$ and $t=1$. Then $h=2$, $I_{0}=\{5,6\}$ and
$W_{0}=24=\mathrm{sp}(G,1)$, so (i) again gives equality in Theorem
\ref{Tm_gen}. But (ii) fails, because the degree $6$ is neither $\delta =5$ nor
$\delta +h=7$. Accordingly (\ref{For_clz4}) gives only $f_{1}=(24+12)/2=18$,
while the twelve vertices of degree $6$ each contribute
$\varepsilon _{1}(6)=1$, so that Theorem \ref{Tm_gen} gives $(24+12+12)/2=24$,
the exact value. The equality sequence of Proposition \ref{Prop_equality}(ii)
for $(\delta ,k)=(5,1)$ at $n=24$ is $5^{18}7^{6}$, of sum
$5\cdot 18+7\cdot 6=132=6\cdot 24-12$, a quite different graph: being extremal
for one $k$ says nothing about being extremal for another.
\end{example}


For $1\leq s\leq n$ put $t=\left\lfloor n/s\right\rfloor $ and $b=n-ts\in \lbrack 0,s)$, and let
the \emph{packed} $\delta $-sequence $P(n,s)$ consist of $s$ copies of each of $\delta ,\delta
+h,\ldots ,\delta +(t-1)h$ and $b$ copies of $\delta +th$; its sum is
\[
\Sigma (n,s)=n\delta +h\sum_{i=1}^{n}\left\lfloor \frac{i-1}{s}\right\rfloor =n\delta +h\left( \frac{%
st(t-1)}{2}+bt\right) .
\]
For $h=1$ these are exactly the \emph{packed lists} of Caro and West \cite{CaroWest}.

For $2m\geq n\delta $ let
\[
s^{\ast }(n,m,\delta ,k)=\min \{s\geq 1:\Sigma (n,s)\leq 2m\},
\]
which is well defined because $\Sigma (n,n)=n\delta $.

In words, $P(n,s)$ is the cheapest $\delta $-sequence of length $n$ whose window
number is at most $s$. One is not allowed more than $s$ terms in any window of
$k+1$ consecutive integers, so the greedy choice is to put $s$ terms at the
smallest admissible value $\delta $, then $s$ at $\delta +h$, then $s$ at
$\delta +2h$, and so on until the terms run out; any other placement can only
raise the sum. The quantity $\Sigma (n,s)$ is that cheapest sum, and it
decreases as $s$ grows, since a larger allowance at each level lets more terms
be placed low. Hence $s^{\ast }$, the least $s$ whose cheapest sequence still
fits inside the budget $2m$, is the least window number available at the given
order and size, which is what Theorem \ref{Tm_sequence} asserts.


\begin{theorem}
\label{Tm_sequence}Let $n\geq 2$, $\delta \geq 0$, $k\geq 0$ and $2m\geq n\delta $. Then
\[
\max_{t\geq 1}\left\lceil f_{t}(n,m,\delta ,k)\right\rceil \ =\ s^{\ast }(n,m,\delta ,k)\ =\ \min
\{\mathrm{sp}(D,k):D\in \mathcal{D}\},
\]
where $\mathcal{D}$ is the set of all $\delta $-sequences of length $n$ with sum $2m$. We call
$s^{\ast }(n,m,\delta ,k)$ the \emph{maximised bound}. Consequently the maximised bound is attained by a graph of order
$n$ with $m$ edges and minimum degree $\delta $ if and only if some $\delta $-sequence with sum
$2m$ and window number $s^{\ast }$ is the degree sequence of a graph.
\end{theorem}

\begin{example}
\label{Ex_sstar}Let $n=24$, $\delta =3$ and $k=0$, so that $h=1$ and
$2m=6n-12=132$. The packed sequences and their sums are
\[
\begin{array}
[c]{c|l|c}%
s & P(24,s) & \Sigma (24,s)\\\hline
3 & 3^{3}4^{3}5^{3}6^{3}7^{3}8^{3}9^{3}10^{3} & 156\\
4 & 3^{4}4^{4}5^{4}6^{4}7^{4}8^{4} & 132\\
5 & 3^{5}4^{5}5^{5}6^{5}7^{4} & 118
\end{array}
\]
so that $s^{\ast }(24,66,3,0)=4$: three vertices at each of eight consecutive
degrees would cost a degree sum of $156$, which a maximal planar graph of order $24$
cannot afford, whereas four at each of six consecutive degrees costs exactly
$132$. This agrees with $\max_{t}\left\lceil f_{t}\right\rceil =4$, attained at
$t=5$ and at $t=6$ as in Example \ref{Ex_bestt}. Since
$3^{4}4^{4}5^{4}6^{4}7^{4}8^{4}$ is in fact the degree sequence of a
maximal planar graph (Figure \ref{Fig_MP0a}, top), the last sentence of Theorem
\ref{Tm_sequence} gives $\mathrm{MP}(24,3,0)=4$.
\end{example}


Theorem \ref{Tm_sequence} says that the only loss in Theorem \ref{Tm_gen} is graphicality: the maximised bound is
the exact minimum of the window number over all integer sequences with the given $n$, $m$ and $\delta $.

\section{Maximal planar graphs}
\label{Sec_planar}

The bound of Section \ref{Sec_prelim} involves $n$, $m$ and $\delta $ only, and
for a maximal planar graph two of these three are determined: $m=3n-6$,
hence $d=6-12/n$, and $\delta \in \{3,4,5\}$. Applying Section \ref{Sec_prelim}
to this class therefore reduces, for each pair $(\delta ,k)$, to a finite
computation --- the choice of the best $t$ --- and this section is that
computation together with its consequences. It is therefore a section of lower bounds; the matching
constructions are in Section \ref{Sec_asym}, with one exception. The exception is Section \ref{Sec_hex},
which for $(\delta ,k)=(5,0)$ also supplies the family that matches the bound, and does so with an explicit
error term; it is placed here,
and not in Section \ref{Sec_asym}, because it belongs with the case $\delta =5$ that precedes it and
is an aside construction, easy to check and aimed at improving an error term rather than a leading term
already determined. The three minimum degrees behave quite differently. For $\delta =5$ Euler's formula in the form
$\sum_{v}(6-\deg (v))=12$ already forces twelve vertices of degree five, the
average degree is within $12/n$ of $\delta +1$, and a single aligned window does
all the work: the answer is a proportion $k/(k+1)$ of all vertices, and both the
bound and its equality case are settled in a few lines (Section
\ref{Sec_planar5}). The case $k=0$ with $\delta =5$ is the repetition number,
where the proportion drops to $1/3$, as shown in \cite{CaroWest}; Section \ref{Sec_hex} gives an explicit
family showing that $1/3$ is optimal, with a better error term. For $\delta =3$ and $\delta =4$ nothing
is forced beyond the degree sum, the number of windows one should use grows as
$\delta $ and $k$ decrease, and the answer is the list of bounds of Section
\ref{Sec_planar34}.


\subsection{Minimum degree five}
\label{Sec_planar5}

Problem 2 of \cite{CaroLauriZarb} asks for
\[
\mathrm{MP}(n,\delta ,k)=\min \{\mathrm{sp}(G,k):G\text{ a maximal planar graph of order }n%
\text{ with }\delta (G)=\delta \}
\]
for $\delta =3,4,5$ and $k\geq 0$. We treat first the case $\delta =5$, which Theorem
\ref{Tm_gen} of \cite{P1} settles up to the realisability of one degree sequence; the cases $\delta =3,4$ are
treated in Section \ref{Sec_planar34}. A maximal planar graph of
order $n\geq 3$ has $m=3n-6$ edges, so $\sum_{v}\deg (v)=6n-12$, and Euler's formula in the form
$\sum_{v}(6-\deg (v))=12$ shows that a maximal planar graph with $\delta =5$ has at least twelve
vertices of degree $5$.

The whole of the case $\delta =5$ rests on the following observation. Since
$d-\delta =1-12/n<1$, the quantity $x=(d-\delta )/(k+1)$ of Section
\ref{Sec_prelim} satisfies $2x<2/(k+1)\leq 1$ for every $k\geq 1$, so Lemma
\ref{Lemma_bestt} gives $t=1$: one window suffices, and the bound says that at
least a proportion $k/(k+1)$ of the vertices, plus a bounded surplus coming from
the twelve units of curvature, carry their degrees inside a single window. Part
(ii) treats $k=0$ separately, where $2x=2-24/n$ and $t=2$ becomes optimal at
$n=24$. In both parts the error term $\varepsilon $ of Theorem \ref{Tm_gen} is
written out explicitly, and it is what turns ``at least $kn/(k+1)$'' into the
sharper ``at least $(kn+12)/(k+1)$''.


\begin{theorem}
\label{Tm_planar}Let $n\geq 12$ and let $G$ be a maximal planar graph of order $n$ with
$\delta (G)=5$.

\begin{itemize}
\item[\textrm{(i)}] For every $k\geq 1$,
\[
\mathrm{sp}(G,k)\ \geq \ \frac{kn+12+\sum_{v}\varepsilon (\deg (v))}{k+1}\ \geq \ \frac{kn+12}{k+1}%
,
\]
the second inequality being an equality only if $n\equiv 12\ (\mathrm{mod}\ k+1)$, and
where $\varepsilon (x)=x-5$ for $5\leq x\leq k+5$ and $\varepsilon (x)=x-6-k$ for $x\geq k+6$.
Equality forces $\deg (v)\in \{5,k+6\}$ for every $v$, and then
\[
n_{5}=\frac{kn+12}{k+1},\qquad n_{k+6}=\frac{n-12}{k+1}.
\]

\item[\textrm{(ii)}] $\displaystyle \mathrm{sp}(G,0)\geq \frac{n+12+\sum_{v}\max \{0,\deg (v)-7\}}%
{3}\geq \frac{n+12}{3}$, and equality forces $3\mid n$, $n\geq 24$, $\deg (v)\in \{5,6,7\}$ for
every $v$, $n_{5}=n_{6}=\frac{n+12}{3}$ and $n_{7}=\frac{n-24}{3}$.
\end{itemize}

Consequently
\[
\mathrm{MP}(n,5,k)\geq \left\lceil \frac{kn+12}{k+1}\right\rceil \ (k\geq 1),\qquad \mathrm{MP}%
(n,5,0)\geq \left\lceil \frac{n+12}{3}\right\rceil .
\]
\end{theorem}

\begin{proof}
Part (i) is Theorem \ref{Tm_gen} with $\delta =5$, $t=1$ and $2m=6n-12$: there
$h=k+1$ and
\[
\frac{2\left( n(5+h)-(6n-12)+\sum_{v}\varepsilon _{1}(\deg (v))\right) }{h\cdot 1\cdot 2}=\frac{%
n(h-1)+12+\sum_{v}\varepsilon _{1}(\deg (v))}{h},
\]
and $h-1=k$; the displayed $\varepsilon $ is $\varepsilon _{1}$ of Lemma \ref{Lemma_gen} with
$\delta =5$. By that lemma $\varepsilon _{1}$ vanishes exactly at $5$ and at $5+h=k+6$, which gives
the equality condition, and then $n_{5}+n_{k+6}=n$ together with $5n_{5}+(k+6)n_{k+6}=6n-12$ yields
the two displayed counts. Part (ii) is Theorem \ref{Tm_gen} with $\delta =5$, $k=0$, $t=2$, where
$h=1$ and $\varepsilon _{2}(x)=\max \{0,x-7\}$; the weights of the two windows are $2$ and $1$, so
equality also forces $n_{5}=n_{6}=\mathrm{sp}(G,0)$, and $n_{5}+n_{6}+n_{7}=n$ together with
$5n_{5}+6n_{6}+7n_{7}=6n-12$ gives the counts.
\end{proof}

\begin{example}
\label{Ex_ico}The smallest instance is the icosahedron: $n=12$ and every degree
equals $5$, so $\mathrm{sp}(G,k)=12=n$ for every $k$, while part (i) gives
$(12k+12)/(k+1)=12$. Equality therefore holds simultaneously for every
$k\geq 1$, the equality conditions being satisfied vacuously with $n_{5}=12$ and
$n_{k+6}=0$.
\end{example}

\begin{example}
\label{Ex_MP5}At $n=24$ part (i) gives $\mathrm{MP}(24,5,1)\geq 18$ with
equality sequence $5^{18}7^{6}$, and part (ii) gives
$\mathrm{MP}(24,5,0)\geq 12$ with equality sequence $5^{12}6^{12}$; the latter
is realised (Figure \ref{Fig_MP0b}, top), so $\mathrm{MP}(24,5,0)=12$.
At $n=27$ the equality sequence for $k=0$ is $5^{13}6^{13}7^{1}$, also realised,
so $\mathrm{MP}(27,5,0)=13$. The numbers show how slowly the asymptotics take
over: the bounds $12$ and $18$ are a half and three quarters of all $24$
vertices, against the limiting proportions $1/3$ and $1/2$.
What should also be noticed is how rigid the equality case is: only two degrees
may occur, namely $5$ and $k+6$, so an extremal maximal planar graph for $k=1$ has no
vertex of degree $6$ at all.
\end{example}


The bound of Theorem \ref{Tm_planar}(i) is attained at $n=12$ by the icosahedron, and for $k=0$ Section
\ref{Sec_hex} gives an explicit family attaining $(n+12)/3$ up to $O(\sqrt{n})$. What matters is that the
equality sequence is realisable asymptotically: Section \ref{Sec_asym} gives the matching upper bound by
showing that $\mathrm{MP}(n,5,k)=\frac{k}{k+1}n+O(\sqrt{n})$ for every sufficiently large $n$ and every
$k\geq 1$.

\subsection{The repetition number of triangulations with minimum degree five: an explicit error
term}
\label{Sec_hex}

We now turn to $k=0$, where Theorem \ref{Tm_planar}(ii) gives
$\mathrm{MP}(n,5,0)\geq \left\lceil (n+12)/3\right\rceil $ with equality only for the degree sequence
$5^{s}6^{s}7^{s-12}$, $s=(n+12)/3$. Caro and West \cite[Theorem 2.5]{CaroWest} showed that the minimum
of $\mathrm{rep}(G)$ over maximal planar graphs with minimum degree five is asymptotic to $n/3$, using rows
of the hexagonal grid with degrees $5,6,7$ in proportion $1:1:1$. The following family makes the
error term explicit and has a completely described boundary.

The idea behind the family can be stated before any coordinates appear. In the
triangular lattice every vertex has degree six, so the repetition number of a
large piece of it is essentially $n$; to bring it down to $n/3$ one must spoil
two thirds of the degrees, and the cheapest way to spoil four degrees at once is
to flip a single edge. Flipping the edge shared by two triangles lowers by one
the two degrees at its ends and raises by one the two degrees opposite to it,
turning four sixes into two fives and two sevens. If the flipped edges are
spread out so that no vertex is touched twice --- this is what the sublattice
$\Lambda $ of index six achieves --- then flipping the right number of them
makes the degrees $5$, $6$ and $7$ occur in the ratio $1:1:1$, which is the
ratio minimising the largest of the three counts. What is then left is the
boundary: the piece has to be closed up into a sphere, and the six extra
vertices of the construction do this at the cost of $O(\sqrt{n})$ vertices of
exceptional degree.


The \emph{hexagonal patch} $H_{L}$ is the subgraph of the triangular lattice induced by
\[
V(H_{L})=\{(i,j)\in \mathbb{Z}^{2}:\left\vert i\right\vert \leq L,\ \left\vert j\right\vert \leq L,\
\left\vert i+j\right\vert \leq L\},
\]
where $(i,j)$ is adjacent to $(i,j)+d$ for every $d$ in
\[
D=\{(1,0),\ (0,1),\ (-1,1),\ (-1,0),\ (0,-1),\ (1,-1)\}.
\]
It has $3L^{2}+3L+1$ vertices and it is a triangulated hexagon with six corners of degree $3$,
$6(L-1)$ further boundary vertices of degree $4$ and $3L^{2}-3L+1$ interior vertices of degree $6$;
the interior vertices form $H_{L-1}$. A \emph{rhombus} is a set
$R(v)=\{v,\ v+(1,0),\ v+(0,1),\ v+(1,-1)\}$: its two triangles share the edge from $v$ to $v+(1,0)$,
and \emph{flipping} this edge, that is replacing it by the edge from $v+(0,1)$ to $v+(1,-1)$, keeps
the graph a triangulation, turns the two degrees at the old edge from $6$ into $5$ and the two at
the new edge from $6$ into $7$ when $R(v)\subseteq H_{L-1}$, and creates no multiple edge because the
two new neighbours are at lattice distance two. We flip rhombi $R(v)$ with $v$ in the lattice
\[
\Lambda =\{(2u,\ u+3w):u,w\in \mathbb{Z}\}
\]
of index $6$. The four vertices of $R(0,0)$ lie in four different cosets of $\Lambda $, since none of
$(1,0)$, $(0,1)$, $(1,-1)$, $(1,-1)-(0,1)=(1,-2)$, $(1,0)-(0,1)=(1,-1)$ and $(1,0)-(1,-1)=(0,1)$ has
an even first coordinate together with a second coordinate congruent to half of it modulo $3$;
hence the rhombi $R(v)$, $v\in \Lambda $, are pairwise disjoint, and they cover two thirds of the
lattice.

Two comments on these definitions may help. The patch $H_{L}$ is the regular
hexagon of the triangular lattice with $L$ steps along each side, described in
the oblique coordinates in which the six neighbours of a vertex are reached by
adding the six vectors of $D$; the columns $i=-L,\ldots ,L$ contain
$L+1,L+2,\ldots ,2L+1,\ldots ,L+2,L+1$ vertices. For $L=2$, for instance, the
columns contain $3,4,5,4,3$ vertices, nineteen in all, of which six are corners
of degree $3$, six more are boundary vertices of degree $4$, and the remaining
seven form $H_{1}$ and have degree $6$.

The lattice $\Lambda $ is generated by $(2,1)$ and $(0,3)$, whose determinant is
$6$, and that is its index; two vertices are congruent modulo $\Lambda $ when
their difference is of the form $(2u,u+3w)$. The verification above says that
the four vertices of a rhombus fall into four different classes, so the
translates $R(v)$, $v\in \Lambda $, are pairwise disjoint and together occupy
four of the six classes, that is two thirds of the lattice. That proportion is
exactly what Theorem \ref{Tm_hex} needs: each flip creates two vertices of
degree $5$ and two of degree $7$ out of four vertices, so producing $n/3$ fives
and $n/3$ sevens requires disjoint rhombi covering $2n/3$ vertices, and no
sparser sublattice would do.

\begin{lemma}
\label{Lemma_rhombi}For $M\geq 1$ the number $\Phi (M)$ of $v\in \Lambda $ with
$R(v)\subseteq H_{M}$ satisfies
\[
6\,\Phi (M)=\left\{
\begin{array}
[c]{ll}%
3M^{2}-M,\ \ 3M^{2}+3,\ \ 3M^{2}-2M-2 & \text{for }M\equiv 0,1,2\ (\mathrm{mod}\ 6),\\
3M^{2}-M,\ \ 3M^{2},\ \ 3M^{2}-2M+1 & \text{for }M\equiv 3,4,5\ (\mathrm{mod}\ 6),
\end{array}
\right.
\]
and in particular $\Phi (M)\geq (M^{2}-M-2)/2$.
\end{lemma}

\begin{proof}
For $v=(i,j)$ the four conditions $v,v+(1,0),v+(0,1),v+(1,-1)\in H_{M}$ amount to
$-M\leq i\leq M-1$, $-M+1\leq j\leq M-1$ and $-M\leq i+j\leq M-1$. Writing $i=2u$ and $j=u+3w$, for
fixed $u\geq 0$ the admissible $w$ form the interval from $-\left\lfloor (M-1+u)/3\right\rfloor $ to
$\left\lfloor (M-1)/3\right\rfloor -u$, and for fixed $u=-t<0$ the interval from
$t-\left\lfloor M/3\right\rfloor $ to $\left\lfloor (M-1+t)/3\right\rfloor $, where
$0\leq u\leq \left\lfloor (M-1)/2\right\rfloor $ and $1\leq t\leq \left\lfloor M/2\right\rfloor $.
Summing the lengths of these intervals and evaluating the floor sums in each residue class of $M$
modulo $6$ gives the six formulas, and $6\Phi (M)-3(M^{2}-M-2)$ equals $2M+6$, $3M+9$, $M+4$,
$2M+6$, $3M+6$, $M+7$ respectively, which is positive.
\end{proof}

\begin{example}
\label{Ex_rhombi}The first values of $\Phi $, obtained both by counting lattice
points directly and from the formula, are
\[
\begin{array}
[c]{c|ccccccccc}%
M & 1 & 2 & 3 & 4 & 5 & 6 & 7 & 8 & 9\\\hline
\Phi (M) & 1 & 1 & 4 & 8 & 11 & 17 & 25 & 29 & 39\\
(M^{2}-M-2)/2 & -1 & 0 & 2 & 5 & 9 & 14 & 20 & 27 & 35
\end{array}
\]
The quantity to keep in mind is the leading term. One has
$\Phi (M)=M^{2}/2+O(M)$ while $H_{M}$ has $3M^{2}+O(M)$ vertices, so the
disjoint rhombi $R(v)$, $v\in \Lambda $, cover $4\Phi (M)=2M^{2}+O(M)$ of them,
two thirds of the patch, as the coset count predicts. The convergence is slow:
the ratio $4\Phi (M)/\left\vert V(H_{M})\right\vert $ is $0.597$ at $M=12$ and
still only $0.652$ at $M=60$. The crude lower bound $(M^{2}-M-2)/2$ is all that
the proof of Theorem \ref{Tm_hex} uses; it is smaller than $\Phi (M)$ by
$O(M)$, which is absorbed in the error term of that theorem.
\end{example}

Figure \ref{Fig_hex3} shows the four rhombi of $H_{3}$ counted by $\Phi (3)=4$, and,
beside them, what a single flip does to the four degrees of one rhombus.

\begin{figure}[htbp]
\begin{center}
\includegraphics[width=0.5\textwidth]{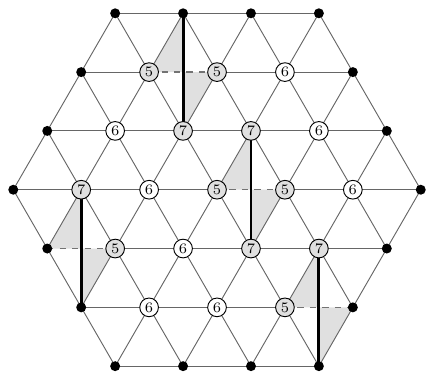}\hfill
\includegraphics[width=0.44\textwidth]{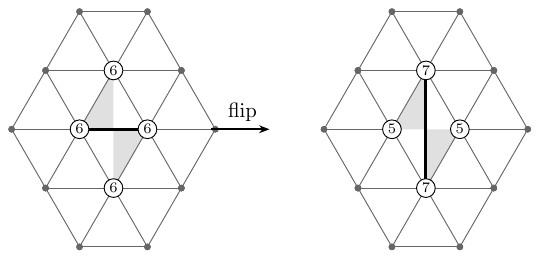}
\end{center}
\caption{Left: the patch $H_{3}$ and the four rhombi $R(v)$, $v\in \Lambda $, that it contains; in each
rhombus the dashed diagonal is flipped to the bold one, and the interior vertices carry their degree after
the flips. Right: a single flip in the lattice, which turns the degrees $6,6,6,6$ of the four vertices of a
rhombus into $5,7,5,7$ and changes no other degree.}%
\label{Fig_hex3}%
\end{figure}


\begin{theorem}
\label{Tm_hex}For every $L\geq 4$ there is a maximal planar graph $G_{L}$ of order $n=3L^{2}+3L+7$
with $\delta (G_{L})=5$ whose degree sequence consists of $5^{L^{2}+3L}\,6^{L^{2}+3L+1}\,7^{L^{2}-3L}$
and six further vertices of degrees $L+3$, $L+3$, $L+4$, $L+4$, $L+4$ and $L+6$. Consequently
\[
\left\lceil \frac{n+12}{3}\right\rceil \ \leq \ \mathrm{MP}(n,5,0)\ \leq \ L^{2}+3L+1\ =\ \frac{n+12}{3}%
+2L-\frac{16}{3}\qquad \text{for }n=3L^{2}+3L+7,
\]
so that $\mathrm{MP}(n,5,0)=n/3+O(\sqrt{n})$ along this family (Caro and West \cite[Theorem 2.5]{CaroWest}
obtain $n/3+o(n)$ for all $n$).
\end{theorem}

\begin{proof}
Start from $H_{L}$ and flip $F=(L^{2}-3L)/2$ of the rhombi $R(v)$, $v\in \Lambda $, contained in
$H_{L-1}$; there are enough of them by Lemma \ref{Lemma_rhombi} with $M=L-1$, since
$\Phi (L-1)\geq ((L-1)^{2}-(L-1)-2)/2=(L^{2}-3L)/2$. The flips are on pairwise disjoint rhombi, so
afterwards $2F=L^{2}-3L$ interior vertices have degree $5$, as many have degree $7$, and the remaining
$3L^{2}-3L+1-4F=L^{2}+3L+1$ interior vertices have degree $6$. Now add six new vertices $w_{0},\ldots
,w_{5}$, one outside each side of the hexagon, join $w_{r}$ to the $L+1$ vertices of its side, join
$w_{r}$ to $w_{r+1}$ for every $r$ (indices modulo $6$), and add the diagonals $w_{0}w_{2}$,
$w_{0}w_{3}$ and $w_{0}w_{4}$. Each side becomes a fan of $L$ triangles, each corner $c$ between the
sides of $w_{r}$ and $w_{r+1}$ lies in the triangle $cw_{r}w_{r+1}$, and the hexagon $w_{0}\cdots w_{5}$
is triangulated by the three diagonals; so the result $G_{L}$ is a maximal planar graph with
$n=3L^{2}+3L+7$ vertices. Every boundary vertex of $H_{L}$ now has degree $5$: a side vertex gains one
neighbour and a corner gains two. Hence $n_{5}=(L^{2}-3L)+6L=L^{2}+3L$, $n_{6}=L^{2}+3L+1$ and
$n_{7}=L^{2}-3L$ among the vertices of $H_{L}$, while $\deg (w_{0})=L+6$, $\deg (w_{1})=\deg
(w_{5})=L+3$ and $\deg (w_{2})=\deg (w_{3})=\deg (w_{4})=L+4$. For $L\geq 4$ these six degrees are at
least $7$ and $n_{7}+2<n_{6}$, so $\mathrm{sp}(G_{L},0)=n_{6}=L^{2}+3L+1$, and $L^{2}+3L+1-(n+12)/3=2L-16/3$.
\end{proof}

Figure \ref{Fig_G4} draws $G_{4}$, the smallest member of the family.

\begin{figure}[htbp]
\begin{center}
\includegraphics[width=0.62\textwidth]{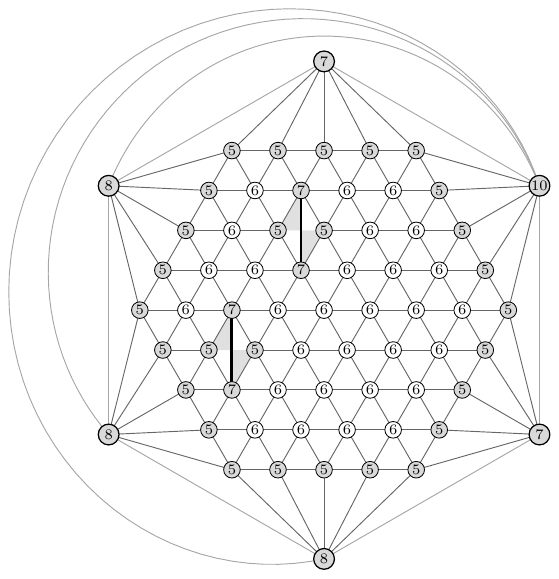}
\end{center}
\caption{The maximal planar graph $G_{4}$ of Theorem \ref{Tm_hex}, of order $n=67$: the patch $H_{4}$ with the two
rhombi $R(-2,-1)$ and $R(-2,2)$ flipped (bold edges) and the six cap vertices $w_{0},\ldots ,w_{5}$. The
three edges $w_{0}w_{2}$, $w_{0}w_{3}$, $w_{0}w_{4}$ are drawn around the outside, because the faces they
bound lie outside the hexagon $w_{0}\cdots w_{5}$. The labels are degrees: $5$ all along the boundary of the
patch, $6$ inside it except for the $5,5,7,7$ of each flipped rhombus, and $10,7,8,8,8,7$ on the caps.}%
\label{Fig_G4}%
\end{figure}

\begin{example}
\label{Ex_hexL4}Take $L=4$, the smallest value the theorem allows. Then $H_{4}$
has $61$ vertices, exactly $F=(16-12)/2=2$ of the $\Phi (3)=4$ available rhombi
inside $H_{3}$ are flipped, six vertices are added, and $n=67$. The degree
sequence of $G_{4}$ is
\[
5^{28}\,6^{29}\,7^{4}\qquad \text{together with}\qquad 7,\ 7,\ 8,\ 8,\ 8,\ 10,
\]
the six exceptional degrees being $L+3=7$ twice, $L+4=8$ three times and
$L+6=10$ once; the sum is $390=6\cdot 67-12$. Hence $n_{5}=28$, $n_{6}=29$,
$n_{7}=6$, $n_{8}=3$, $n_{10}=1$ and
$\mathrm{sp}(G_{4},0)=29=L^{2}+3L+1$, against the lower bound
$\left\lceil (67+12)/3\right\rceil =27$; the gap is $29-79/3=8/3=2L-16/3$.
At $L=4$ the two cap vertices of degree $L+3$ happen to have the same degree as
the interior sevens, which is why $n_{7}=6$ rather than $4$; from $L=5$ on all
six exceptional degrees exceed $7$. In every case the largest count is
$n_{6}=L^{2}+3L+1$, which exceeds $n/3$ by $2L-16/3=2\sqrt{n/3}+O(1)$: the
entire loss sits on the boundary of the patch, whose length is $6L$, and this is
the precise sense in which the family is optimal up to $O(\sqrt{n})$.
\end{example}

The constants are not optimised: the boundary contributes $6L$ vertices of degree $5$, and the flips
are chosen so that $n_{5}$ and $n_{6}$ differ by one. For small orders the bound of Theorem
\ref{Tm_planar}(ii) is exact: a search over the maximal planar graphs under diagonal flips found
maximal planar graphs with degree sequence $5^{s}6^{s}7^{s-12}$, $s=(n+12)/3$, for every
$n\in \{24,27,30,\ldots ,48\}$, so that
\[
\mathrm{MP}(n,5,0)=\frac{n+12}{3}\qquad \text{for every }n\equiv 0\ (\mathrm{mod}\ 3),\ 24\leq n\leq 48;
\]
for $n=24$ this is the dual of a fullerene with twelve hexagons, and the maximal planar graph for $n=27$
has a single vertex of degree $7$; both are drawn in Figure \ref{Fig_MP0b}.
Whether it is
attained for every $n\equiv 0\ (\mathrm{mod}\ 3)$ with $n>48$ remains open.

\subsection{Minimum degrees three and four}
\label{Sec_planar34}

For $\delta =3$ and $\delta =4$ the situation differs in one decisive respect:
Euler's formula gives no lower bound on the number of vertices of minimum
degree, since $\sum_{v}(6-\deg (v))=12$ can be met by a handful of vertices of
small degree together with one vertex of large degree. All that remains is the
degree sum $6n-12$ and the value of $\delta $ itself. The effect is that the
gap $d-\delta $ between the average and the minimum degree is now close to $3$
or to $2$ instead of being close to $1$, so that $x=(d-\delta )/(k+1)$ is large
when $k$ is small and several windows are needed: the optimal $t$ is $6$ for
$(\delta ,k)=(3,0)$ and decreases to $1$ as $k$ grows. What this subsection
does is to carry out $\max_{t}f_{t}$ in each case. The outcome is the list of
bounds of Theorem \ref{Tm_planar34}, each with a threshold on $n$ attached, and
the thresholds are precisely the orders at which the optimal $t$ changes.


For $\delta =3$ and $\delta =4$ the constraint $\sum_{v}(6-\deg (v))=12$ no longer bounds
$n_{\delta }$ from below, and Theorem \ref{Tm_gen} is the only tool we have. With $2m=6n-12$ the
bound (\ref{For_clz4}) reads
\[
\mathrm{sp}(G,k)\ \geq \ f_{t}(n)=\frac{2\left( n(\delta +t(k+1)-6)+12\right) }{(k+1)\,t(t+1)},
\]
and $2x=\left( 2(6-\delta )-24/n\right) /(k+1)$ increases with $n$ towards $2(6-\delta )/(k+1)$, so by
Lemma \ref{Lemma_bestt} the best $t$ is eventually constant. Working out the maximum over $t$ gives
the following bounds, which refine those of \cite{CaroLauriZarb,CaroWest} (see the discussion after the
proof); the equality sequences come from Proposition \ref{Prop_equality}. Each displayed inequality
holds for every $n$, since Theorem \ref{Tm_gen} holds for every $t$; the range in parentheses is where
it is the maximised bound.

\begin{theorem}
\label{Tm_planar34}Let $G$ be a maximal planar graph of order $n$.

\begin{itemize}
\item[\textrm{(i)}] If $\delta (G)=3$, then
\[
\begin{array}
[c]{ll}%
\mathrm{sp}(G,0)\geq \dfrac{n+4}{7}\ \ (n\geq 24), & \mathrm{sp}(G,1)\geq \dfrac{n+4}{4}\ \ (n\geq
12),\\[3mm]
\mathrm{sp}(G,2)\geq \dfrac{n+4}{3}\ \ (n\geq 8), & \mathrm{sp}(G,3)\geq \dfrac{5n+12}{12}\ \ (n\geq
12),\\[3mm]
\mathrm{sp}(G,4)\geq \dfrac{7n+12}{15}\ \ (n\geq 24), & \mathrm{sp}(G,k)\geq \dfrac{(k-2)n+12}{k+1}\
\ (k\geq 5),
\end{array}
\]
with equality only for the degree sequences $3^{s}4^{s}5^{s}6^{s}7^{s}8^{s}9^{n-6s}$,
$3^{s}5^{s}7^{s}9^{n-3s}$, $3^{s}6^{s}9^{n-2s}$, $3^{s}7^{s}11^{n-2s}$, $3^{s}8^{s}13^{n-2s}$ and
$3^{s}(k+4)^{n-s}$ respectively, where $s$ is the value of the bound. Below the stated orders the
bound is $\left\lceil (2n+12)/15\right\rceil $ for $13\leq n\leq 24$ and $k=0$,
$\left\lceil (n+12)/6\right\rceil $ for $7\leq n\leq 12$ and $k=1$, $4$ for $n\leq 8$ and $k=2$,
$\left\lceil (n+12)/4\right\rceil $ for $n\leq 12$ and $k=3$, and $\left\lceil (2n+12)/5\right\rceil $
for $n\leq 24$ and $k=4$; at the remaining small orders ($n\leq 12$ for $k=0$ and $n\leq 6$ for
$k=1$) the maximised bound $s^{\ast }$ is the entry of Table \ref{Tab_MP34small} when it carries no asterisk, and the
entry minus the excess indicated in the caption when it does.

\item[\textrm{(ii)}] If $\delta (G)=4$, then
\[
\begin{array}
[c]{ll}%
\mathrm{sp}(G,0)\geq \dfrac{n+6}{5}\ \ (n\geq 24), & \mathrm{sp}(G,1)\geq \dfrac{n+6}{3}\ \ (n\geq
12),\\[3mm]
\mathrm{sp}(G,2)\geq \dfrac{4n+12}{9}\ \ (n\geq 24), & \mathrm{sp}(G,k)\geq \dfrac{(k-1)n+12}{k+1}\ \
(k\geq 3),
\end{array}
\]
with equality only for $4^{s}5^{s}6^{s}7^{s}8^{n-4s}$, $4^{s}6^{s}8^{n-2s}$, $4^{s}7^{s}10^{n-2s}$
and $4^{s}(k+5)^{n-s}$ respectively. Below the stated orders the bound is
$\left\lceil (n+12)/6\right\rceil $ for $12\leq n\leq 24$ and $4$ for $8\leq n\leq 12$ when $k=0$,
$6$ for $6\leq n\leq 12$ when $k=1$, and $\left\lceil (n+12)/3\right\rceil $ for $n\leq 24$ when
$k=2$; for $n=6,7$ and $k=0$ it is $6$ and $5$.
\end{itemize}

In particular $\mathrm{MP}(n,3,0)\geq \left\lceil (n+4)/7\right\rceil $ and $\mathrm{MP}(n,4,0)\geq
\left\lceil (n+6)/5\right\rceil $ for $n\geq 24$.
\end{theorem}

\begin{proof}
For $\delta =3$ we have $f_{t}(n)=\left( 2n(th-3)+24\right) /(ht(t+1))$ with $h=k+1$. For $k=0$,
$f_{5}=(2n+12)/15$, $f_{6}=(n+4)/7$ and $f_{7}=(n+3)/7$; here $f_{6}\geq f_{5}$ is equivalent to
$n\geq 24$, $f_{6}>f_{7}$ always, and no other $t$ competes by Lemma \ref{Lemma_bestt}, since
$2x=6-24/n\in (5,6)$ for $n>24$. For $k=1$, $2x=3-12/n$, $f_{3}=(n+4)/4$ and $f_{2}=(n+12)/6$, with
$f_{3}\geq f_{2}$ exactly for $n\geq 12$. For $k=2$, $2x=2-8/n$, $f_{2}=(n+4)/3$ and $f_{1}=4$, with
$f_{2}\geq f_{1}$ exactly for $n\geq 8$. For $k=3$, $2x=\frac{3}{2}-\frac{6}{n}$, $f_{2}=(5n+12)/12$
and $f_{1}=(n+12)/4$, with $f_{2}\geq f_{1}$ exactly for $n\geq 12$. For $k=4$,
$2x=\frac{6}{5}-\frac{24}{5n}$, $f_{2}=(7n+12)/15$ and $f_{1}=(2n+12)/5$, with $f_{2}\geq f_{1}$
exactly for $n\geq 24$. For $k\geq 5$ we have $2x<1$, so $t=1$ and $f_{1}=((k-2)n+12)/(k+1)$.

For $\delta =4$, $f_{t}(n)=\left( 2n(th-2)+24\right) /(ht(t+1))$. For $k=0$, $2x=4-24/n$,
$f_{4}=(n+6)/5$, $f_{3}=(n+12)/6$, $f_{5}=(n+4)/5$ and $f_{2}=4$, with $f_{4}\geq f_{3}$ exactly for
$n\geq 24$ and $f_{3}\geq f_{2}$ exactly for $n\geq 12$. For $k=1$, $2x=2-12/n$, $f_{2}=(n+6)/3$,
$f_{1}=6$ and $f_{3}=(n+3)/3$. For $k=2$, $2x=(4-24/n)/3<4/3$, $f_{2}=(4n+12)/9$ and
$f_{1}=(n+12)/3$, with $f_{2}\geq f_{1}$ exactly for $n\geq 24$. For $k\geq 3$, $2x<1$, so $t=1$
and $f_{1}=((k-1)n+12)/(k+1)$.

The equality sequences are those of Proposition \ref{Prop_equality}(ii): $n_{\delta +ih}=s$ for
$i<t$ and $n_{\delta +th}=n-ts$, and the requirement $n-ts\geq 0$ is exactly the lower bound on $n$
in each case.
\end{proof}

\begin{example}
\label{Ex_planar34}The smallest instance of the first bound in each part is
drawn in Figure \ref{Fig_MP0a}. For $\delta =3$, $k=0$ and $n=24$ the bound is
$(24+4)/7=4$ and the equality sequence is $3^{4}4^{4}5^{4}6^{4}7^{4}8^{4}$, of
sum $132=6\cdot 24-12$: six consecutive degrees, each carried by exactly four
vertices, so that no degree is repeated more than four times. For $\delta =4$,
$k=0$ and $n=24$ the bound is $(24+6)/5=6$ with equality sequence
$4^{6}5^{6}6^{6}7^{6}$, again of sum $132$. Both are realised, so
$\mathrm{MP}(24,3,0)=4$ and $\mathrm{MP}(24,4,0)=6$. A case in which the bound
is much larger is $\delta =4$, $k=2$, $n=24$, where $(4\cdot 24+12)/9=12$ and
the equality sequence is $4^{12}7^{12}$: half the vertices of such a
maximal planar graph have degree exactly four, and it is realised as well (Table
\ref{Tab_MP34}).
\end{example}

\begin{remark}
\label{Rem_thresholds}The threshold attached to each bound admits two
readings, and it is instructive that they agree. On the one hand it is the
order at which the optimal $t$ changes, computed from Lemma \ref{Lemma_bestt} as
in Example \ref{Ex_bestt}. On the other hand it is the order at which the
equality sequence of Proposition \ref{Prop_equality}(ii) first exists, since
that sequence needs $n-ts\geq 0$ terms at the top level. For
$(\delta ,k)=(3,0)$, where $t=6$ and $s=(n+4)/7$, one has $n-6s=(n-24)/7$; for
$(3,1)$, $n-3s=(n-12)/4$; for $(3,2)$, $n-2s=(n-8)/3$; for $(3,3)$,
$n-2s=(n-12)/6$; for $(3,4)$, $n-2s=(n-24)/15$; for $(4,0)$, $n-4s=(n-24)/5$;
for $(4,1)$, $n-2s=(n-12)/3$; and for $(4,2)$, $n-2s=(n-24)/9$. In each case
the numerator vanishes exactly at the threshold stated in the theorem.
\end{remark}


Caro and West \cite[Theorem 2.5]{CaroWest} proved that the minimum of $\mathrm{rep}(G)$ over
maximal planar graphs with minimum degree $3$, respectively $4$, is asymptotic to $n/7$, respectively
$n/5$, by hexagonal tilings with $\mathrm{rep}(G)\leq n/7+O(\sqrt{n})$ and $n/5+O(\sqrt{n})$;
Theorem \ref{Tm_planar34} gives the exact linear terms $\left\lceil (n+4)/7\right\rceil $ and
$\left\lceil (n+6)/5\right\rceil $, and the computations below show that they are attained for every
$n\leq 30$. The first bound of \cite[Theorem 2.1(1)]{CaroLauriZarb} is $nh/(2d-2\delta +h)$ with
$d=6-12/n$; the bounds of Theorem \ref{Tm_planar34} exceed it by an amount tending to a positive
constant, and by an amount growing linearly in $n$ when the optimal $t$ is not the one implicit in
\cite{CaroLauriZarb}, for instance for $\delta =3$, $k=3$ and for $\delta =4$, $k=2$.

\begin{remark}
\label{Rem_clz}The two kinds of improvement are worth separating numerically.
When the optimal $t$ is the one implicit in \cite{CaroLauriZarb}, the gain is a
constant: for $\delta =5$ and $k=1$ that bound is $n^{2}/(2n-12)=n/2+3+O(1/n)$
whereas Theorem \ref{Tm_planar} gives $(n+12)/2=n/2+6$, a gain tending to $3$,
and for $\delta =3$ and $k=5$ the gain tends to $1$. When the optimal $t$ is
different, the gain is linear: for $\delta =3$ and $k=3$ the two coefficients of
$n$ are $5/12$ and $2/5$, so the gain is $n/60+O(1)$, and for $\delta =4$ and
$k=2$ they are $4/9$ and $3/7$, so the gain is $n/63+O(1)$. At the orders one
can actually search the linear gain is still small --- it is $0.46$ and $0.15$
respectively at $n=30$ --- which is why the difference between the two bounds
shows itself only in the asymptotics of Section \ref{Sec_asym}.
\end{remark}


Notice also that peeling, which for maximal outerplanar graphs improves the window count at once, gives nothing here
as long as one only counts degrees. The vertices of degree $3$ of a maximal planar graph $G$ of order $n\geq 5$
form an independent set $T$, each of them sits in a face of the maximal planar graph $G^{\prime }=G-T$, and
$\deg _{G}(u)=\deg _{G^{\prime }}(u)+a(u)$, where $a(u)$ is the number of faces at $u$ which received a
vertex of $T$; but Euler's formula for $G^{\prime }$ is equivalent to Euler's formula for $G$, so
peeling yields no inequality beyond the degree sum. What peeling does yield is a bound on the \emph{number} of
vertices of degree three, Lemma \ref{Lemma_deg3count} below; that inequality is useless for small $k$, where
the extremal density of the degree three is below $2/3$ anyway, but for $k\geq 9$ it is exactly what shows
that the window count is \emph{not optimal} (Section \ref{Sec_deg3exact}). Vertices of degree $4$ may even be
adjacent, as in the octahedron, so for $\delta =4$ there is no analogous peeling at all.

\begin{lemma}
\label{Lemma_deg3count}Let $G$ be a maximal planar graph with $n\geq 5$ vertices. Then the vertices of
degree three of $G$ are pairwise non-adjacent, no face of $G$ contains two of them, and
\[
n_{3}\ \leq \ \left\lfloor \frac{2n-4}{3}\right\rfloor .
\]
\end{lemma}

\begin{proof}
\emph{Independence.} Let $u$ have degree three. Since $G$ is a maximal planar graph, the neighbours of $u$ form a
cycle, here a triangle $xyv$, and the three faces at $u$ are $uxy$, $uyv$ and $uvx$. Suppose some neighbour
of $u$, say $v$, also has degree three. The neighbours of $v$ include $u$, $x$ and $y$, and there are only
three of them, so $N(v)=\{u,x,y\}$ and the three faces at $v$ are $vxy$, $vyu$ and $vux$. The six faces just
listed are the four distinct triangles $uxy$, $uyv$, $uvx$ and $vxy$, and every edge of $G[\{u,v,x,y\}]$ lies
in two of them; so these four triangles are all the faces of $G$, whence $|V(G)|=4$. This contradicts
$n\geq 5$, and the vertices of degree three are independent.

\emph{At most one to a face.} A face is a triangle, so any two of its vertices are adjacent; by the
foregoing it cannot contain two vertices of degree three.

\emph{Counting.} A vertex of degree three lies in exactly three faces. By the previous paragraph the triples
of faces belonging to distinct vertices of degree three are pairwise disjoint, so $3n_{3}$ is at most the
number $2n-4$ of faces of $G$, and $n_{3}\leq (2n-4)/3$. As $n_{3}$ is an integer the floor may be taken.
\end{proof}

For similar results see \cite{CaroRoditty}, where it is proved that for a maximal planar graph $G$ on $n\geq 4$
vertices with minimum degree $\delta \in \{3,4,5\}$ one has $\alpha (G)\leq (2n-4)/\delta $, which is sharp;
for a characterization of equality see \cite{Bickle}.

\begin{remark}
\label{Rem_peel}The inequality $n_{3}\leq (2n-4)/3$ of Lemma \ref{Lemma_deg3count} is a genuine
restriction on the density of the degree
three: no more than two thirds of the vertices of a maximal planar graph can have
degree three. Comparing it with the equality sequences of Theorem
\ref{Tm_planar34}(i) explains where the case distinction at $k=9$ comes from.
For $k\geq 5$ that sequence is $3^{s}(k+4)^{n-s}$ with $s=((k-2)n+12)/(k+1)$, so
the density of the degree three at equality is $(k-2)/(k+1)$, and
\[
\frac{k-2}{k+1}>\frac{2}{3}\quad \Longleftrightarrow \quad 3k-6>2k+2\quad
\Longleftrightarrow \quad k\geq 9.
\]
The first values are $1/2$, $4/7$, $5/8$, $2/3$ and $7/10$ for
$k=5,6,7,8,9$: at $k=8$ the two densities agree exactly, and from $k=9$ onwards
the equality sequence asks for more vertices of degree three than a
maximal planar graph can possess. That is the whole reason why $k=8$ is the last value
for which the bound is asymptotically tight, and Section \ref{Sec_deg3exact}
makes it quantitative. For $\delta =4$ there is no such obstruction, and the
octahedron shows why: it is a maximal planar graph all six of whose vertices have
degree four, so the density of the degree four can be $1$.
\end{remark}


How close are constructions? We enumerated all maximal planar graphs with $n\leq 13$ vertices
(their numbers $1,1,2,5,14,50,233,1249,7595,49566$ for $n=4,\ldots ,13$ agree with sequence
A000109 of \cite{OEIS}) and took the minimum of $\mathrm{sp}(G,k)$ over those with minimum degree
exactly $\delta $; Table \ref{Tab_MP34small} lists the exact values for $k\leq 3$ and $n\leq 14$. The row $n=14$ required extending the
enumeration to $339\,722$ maximal planar graphs; it contains two entries one above the maximised bound,
$\mathrm{MP}(14,3,3)=8$ and $\mathrm{MP}(14,4,2)=10$. For
$8\leq n\leq 30$ we searched over the maximal planar graphs under diagonal flips; Table \ref{Tab_MP34}
records the orders at which a maximal planar graph attaining the bound of Theorem \ref{Tm_planar34} was
found, each certified by its rotation system. The two extremal triangulations of order $24$ with
$k=0$, of degree sequences $3^{4}4^{4}5^{4}6^{4}7^{4}8^{4}$ and $4^{6}5^{6}6^{6}7^{6}$, are drawn in Figure
\ref{Fig_MP0a}, and two extremal maximal planar graphs with $\delta =5$ in Figure \ref{Fig_MP0b}; all certificates
(rotation systems) are
available from the authors.

\begin{figure}[p]
\begin{center}
\includegraphics[width=0.62\textwidth]{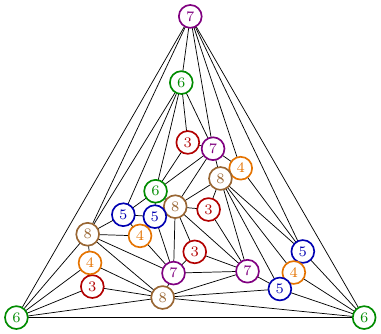}\\[14pt]
\includegraphics[width=0.62\textwidth]{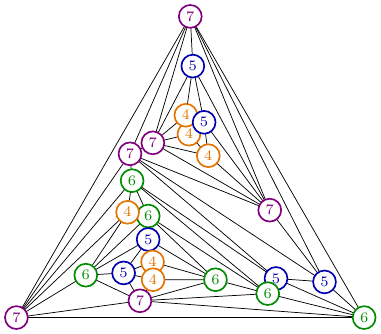}
\end{center}
\caption{Maximal planar graphs of order $24$ attaining the bound of Theorem \ref{Tm_planar34} for $k=0$, each vertex
labelled with its degree. Top: $\delta =3$, degree sequence $3^{4}4^{4}5^{4}6^{4}7^{4}8^{4}$,
$\mathrm{sp}(G,0)=4=\left\lceil (n+4)/7\right\rceil $. Bottom: $\delta =4$, $4^{6}5^{6}6^{6}7^{6}$,
$\mathrm{sp}(G,0)=6=\left\lceil (n+6)/5\right\rceil $. Both drawings are straight-line embeddings with the
outer face a triangle.}%
\label{Fig_MP0a}%
\end{figure}

\begin{figure}[p]
\begin{center}
\includegraphics[width=0.62\textwidth]{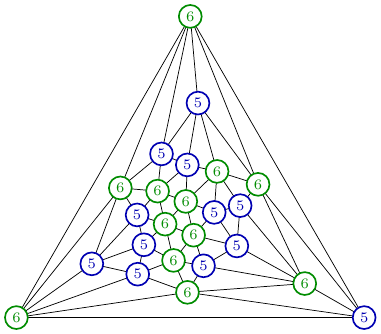}\\[14pt]
\includegraphics[width=0.62\textwidth]{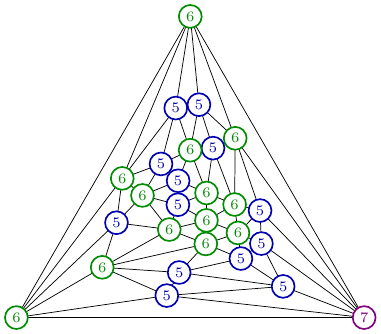}
\end{center}
\caption{Maximal planar graphs of minimum degree five attaining the bound of Theorem \ref{Tm_planar}(ii) for $k=0$,
each vertex labelled with its degree. Top: order $24$, $5^{12}6^{12}$, $\mathrm{sp}(G,0)=12=(n+12)/3$, the
dual of a fullerene with twelve hexagons. Bottom: order $27$, $5^{13}6^{13}7^{1}$,
$\mathrm{sp}(G,0)=13=(n+12)/3$. Both drawings are straight-line embeddings with the outer face a
triangle.}%
\label{Fig_MP0b}%
\end{figure}

\begin{table}[ht]
\small
\begin{center}
\begin{tabular}{|c|cccc|cccc|}
\hline
& \multicolumn{4}{c|}{$\mathrm{MP}(n,3,k)$} & \multicolumn{4}{c|}{$\mathrm{MP}(n,4,k)$}\\
$n$ & $k=0$ & $k=1$ & $k=2$ & $k=3$ & $k=0$ & $k=1$ & $k=2$ & $k=3$\\ \hline
$4$ & $4$ & $4$ & $4$ & $4$ & -- & -- & -- & --\\
$5$ & $3$ & $5^{\ast }$ & $5^{\ast }$ & $5$ & -- & -- & -- & --\\
$6$ & $2$ & $4^{\ast }$ & $6^{\ast }$ & $6^{\ast }$ & $6$ & $6$ & $6$ & $6$\\
$7$ & $2$ & $4$ & $5^{\ast }$ & $7^{\ast }$ & $5$ & $7^{\ast }$ & $7$ & $7$\\
$8$ & $2$ & $4$ & $4$ & $6^{\ast }$ & $4$ & $6$ & $8^{\ast }$ & $8^{\ast }$\\
$9$ & $3$ & $4$ & $5$ & $7^{\ast }$ & $4$ & $6$ & $7$ & $9^{\ast }$\\
$10$ & $3$ & $4$ & $5$ & $6$ & $4$ & $6$ & $8$ & $8$\\
$11$ & $3$ & $4$ & $6^{\ast }$ & $7^{\ast }$ & $4$ & $6$ & $9^{\ast }$ & $9$\\
$12$ & $3$ & $4$ & $6$ & $6$ & $4$ & $6$ & $8$ & $10^{\ast }$\\
$13$ & $3$ & $5$ & $6$ & $7$ & $5$ & $7$ & $9$ & $11^{\ast }$\\
$14$ & $3$ & $5$ & $6$ & $8^{\ast }$ & $5$ & $7$ & $10^{\ast }$ & $10$\\ \hline
\end{tabular}
\end{center}
\caption{The exact values of $\mathrm{MP}(n,\delta ,k)$ for $\delta =3,4$, $k\leq 3$ and $n\leq 14$,
by exhaustive enumeration; an asterisk marks the entries which exceed the maximised bound $s^{\ast }$
of Theorem \ref{Tm_sequence}, by exactly one except at $(\delta ,k,n)=(3,2,6)$ and $(3,3,7)$, where the excess
is two; the entries `--' mean that no maximal planar graph with $\delta =4$ exists at that order.}%
\label{Tab_MP34small}%
\end{table}

\begin{table}[ht]
\begin{center}
\begin{tabular}{|c|c|l|l|}
\hline
$\delta $ & $k$ & bound attained at $n$ & undecided orders $n\leq 30$\\ \hline
$3$ & $0$ & all $4\leq n\leq 30$ & none\\
$3$ & $1$ & $4$ and all $7\leq n\leq 30$ & none\\
$3$ & $2$ & $4,8,9,10,12,\ldots ,16,18,\ldots ,22,24,\ldots ,28,30$ & $17,23,29$\\
$3$ & $3$ & $4,5,10,12,13,15,\ldots ,18,20,\ldots ,23,25,\ldots ,30$ & $19,24$\\
$4$ & $0$ & all $6\leq n\leq 30$ & none\\
$4$ & $1$ & all $6\leq n\leq 30$, except $7$ & none\\
$4$ & $2$ & $6,7,9,10,12,13,15,\ldots ,30$ & none\\
$4$ & $3$ & $6,7,10,11,14,15,18,19,22,23,26,27$ & $16,17,20,21,24,25,28,29,30$\\ \hline
\end{tabular}
\end{center}
\caption{Orders $n\leq 30$ at which the maximised bound $s^{\ast }$ for $\mathrm{MP}(n,\delta ,k)$ is
attained by a certified maximal planar graph; at the orders $n\leq 14$ which are missing the value exceeds
the bound (Table \ref{Tab_MP34small} and the text), and at the undecided orders it is the bound or the
bound plus one.}%
\label{Tab_MP34}%
\end{table}

In particular
\[
\mathrm{MP}(n,3,0)=\left\lceil \frac{n+4}{7}\right\rceil \quad \text{and}\quad \mathrm{MP}(n,4,0)=%
\left\lceil \frac{n+6}{5}\right\rceil \qquad (24\leq n\leq 30),
\]
\[
\mathrm{MP}(n,3,1)=\left\lceil \frac{n+4}{4}\right\rceil \quad \text{and}\quad \mathrm{MP}(n,4,1)=%
\left\lceil \frac{n+6}{3}\right\rceil \qquad (12\leq n\leq 30),
\]
the extremal sequences $3^{s}5^{s}7^{s}9^{n-3s}$ being realised at $n=12,16,20,24,28$ and
$4^{s}6^{s}8^{n-2s}$ at every $n\equiv 0\ (\mathrm{mod}\ 3)$ in $[12,30]$. The misses are of the same
nature as in the outerplanar case of \cite{P1}: for $\delta =3$ and $k=2$ the equality
sequence $3^{s}6^{s}9^{b}$ (with $s=b+4$, $n=3b+8$) was realised for $b=0,2,4,6$, it is impossible
for $b=1$ by the enumeration, and it was not found for $b=3,5,7$.

\begin{remark}
\label{Rem_misses}It is worth spelling out the arithmetic of that family, since
it shows that nothing but graphicality is at stake. For $\delta =3$ and $k=2$
the bound is $(n+4)/3$ and the equality sequence is $3^{s}6^{s}9^{n-2s}$ with
$s=(n+4)/3$; writing $n=3b+8$ turns this into $s=b+4$ and $n-2s=b$, and the sum
of the sequence is $3(b+4)+6(b+4)+9b=18b+36=6n-12$, as it must be. Its window
number is $b+4$, because no window of three consecutive integers meets two of
the values $3$, $6$ and $9$. Thus the sequence passes every test that Section
\ref{Sec_prelim} can apply, and the only question left is whether a
maximal planar graph has it as its degree sequence.
\end{remark}



%

\section{Asymptotic sharpness of the bounds}
\label{Sec_asym}

In this section we determine for which pairs $(\delta ,k)$ the lower bounds of Theorems \ref{Tm_planar} and
\ref{Tm_planar34} have the right linear term. They do for $\delta =4$ and every $k$, for $\delta =5$ and
every $k$, and for $\delta =3$ and $0\leq k\leq 8$; for $\delta =3$ and $k\geq 9$ they do not
(Theorem \ref{Tm_deg3exact} and Proposition \ref{Prop_deg3}).
Write $c_{\delta ,k}$ for the coefficient of $n$ in the maximised bound, so that the bound reads
$\mathrm{MP}(n,\delta ,k)\geq c_{\delta ,k}\,n+O(1)$; for instance $c_{5,k}=k/(k+1)$ for $k\geq 1$, $c_{4,k}=(k-1)/(k+1)$ for
$k\geq 3$, $c_{3,k}=(k-2)/(k+1)$ for $k\geq 5$, and $c_{3,2}=1/3$, $c_{5,3}=3/4$. (For $\delta =3$ and $k\geq 9$ this $c_{3,k}$ is smaller than the true limit $(k-4)/(k-2)$ of Theorem \ref{Tm_deg3exact}, since there the bound is not tight.) Here $t$ is the value
maximising $f_{t}$ for large $n$ (Lemma \ref{Lemma_bestt}; when two values of $t$ are optimal we take the smaller $t$). By
Proposition \ref{Prop_equality}(ii), equality in $f_{t}$ forces the degrees $\delta ,\delta +h,\ldots ,\delta +th$ with
$n_{\delta +ih}=s$ for $i<t$ and $n_{\delta +th}=n-ts$; accordingly we call the densities $n_{\delta +ih}/n=c_{\delta ,k}$
for $i<t$ and $n_{\delta +th}/n=1-tc_{\delta ,k}$ the \emph{extremal densities} of $(\delta ,k)$ (a graph with
$\mathrm{sp}(G,k)=c_{\delta ,k}n+o(n)$ has all but $o(n)$ of its vertices at these degrees, with these
densities in the limit; the upper bounds of this paper give $O(\sqrt{n})$ in place of $o(n)$).

\subsection{From periodic triangulations of the torus to maximal planar graphs}

A \emph{periodic triangulation} of the plane is a locally finite triangulation $T$ of $\mathbb{R}^{2}$ by a simple
graph which is invariant under a lattice $\Lambda \subseteq \mathbb{R}^{2}$ of translations with compact fundamental
domain. The quotient $T/\Lambda $ is a triangulation of the torus (possibly with loops and multiple edges), and
conversely every triangulation of the torus lifts to a periodic triangulation of the plane, its universal cover; the
lift is simple if and only if no edge of the quotient is a loop with trivial translation label and no two edges join
the same pair of vertices with the same label (the labels are defined below). The \emph{density} of a degree $j$ in
$T$ is the proportion of vertices of degree $j$ in the fundamental domain.

It may help to say this less formally. A periodic triangulation is a wallpaper pattern made of
triangles: there are two independent directions in which the whole picture slides onto itself, and a
finite piece of it, together with the two sliding vectors $e_{1},e_{2}$ that generate $\Lambda $,
already determines everything. Such a finite piece is a \emph{fundamental domain}: a region meeting
every orbit of $\Lambda $ in exactly one vertex. All the information is therefore carried by the
finitely many orbits of vertices and by the way they are joined, and the one thing a plain list of
neighbours cannot record is \emph{which copy} of a neighbour is meant. That is the task of a
translation label. If the domain has $N$ vertices $0,\ldots ,N-1$, then writing $v_{(p,q)}$ in the
list of $u$ means that the copy of $u$ inside the domain is joined to the copy of $v$ translated by
$pe_{1}+qe_{2}$, the label $(0,0)$ being omitted; and the neighbours of $u$ are listed in the cyclic
order in which they occur around $u$ in the plane, so that the list records the triangulation and not
merely the graph. Since the vertices of the domain are in bijection with the orbits, the data are
exactly a triangulation of the torus $\mathbb{R}^{2}/\Lambda $ with its edges labelled by the
translations they cross, and this is the form in which the triangulations of Section
\ref{Sec_computer} are given; the five properties which guarantee that such a list really comes from
a periodic triangulation of the plane are listed there. The following example is small enough to be
checked in a couple of lines.

\begin{example}
\label{Ex_trilattice}Let $T$ be the triangular lattice, with vertex set $\mathbb{Z}^{2}$ and $(i,j)$
adjacent to $(i,j)+d$ for every $d$ in the set $D$ of Section \ref{Sec_hex}, drawn with $(i,j)$ at
$i\,e_{1}+j\,e_{2}$, where $e_{1}=(1,0)$ and $e_{2}=(\tfrac{1}{2},\tfrac{\sqrt{3}}{2})$. Its group of
translations contains $e_{1}$ and $e_{2}$, so all vertices form a single orbit and a fundamental
domain consists of one vertex $0$; the rotation system with translation labels is
\[
0:\quad 0_{(1,0)},\ 0_{(0,1)},\ 0_{(-1,1)},\ 0_{(-1,0)},\ 0_{(0,-1)},\ 0_{(1,-1)}.
\]
The quotient is not a simple graph at all: it is one vertex carrying three loops. The lift is
nevertheless simple, because no label is $(0,0)$ and no labelled neighbour is repeated. Tracing the
faces produces $2N=2$ triangles, the orbits of $\{v,v+(1,0),v+(0,1)\}$ and of
$\{v,v+(1,0),v+(1,-1)\}$, and the three labels around each of them sum to $(0,0)$; with $N=1$ vertex
and $E=3N=3$ edges this gives $V-E+F=1-3+2=0$, as it must for a torus. The only degree is $6$, its
density is $1$, and $\sum_{v}(6-\deg (v))=0$ over the domain.
\end{example}

Figure \ref{Fig_trilattice} draws the lattice with a fundamental domain and the
rotation system at one vertex.

\begin{figure}[htbp]
\begin{center}
\includegraphics[width=0.58\textwidth]{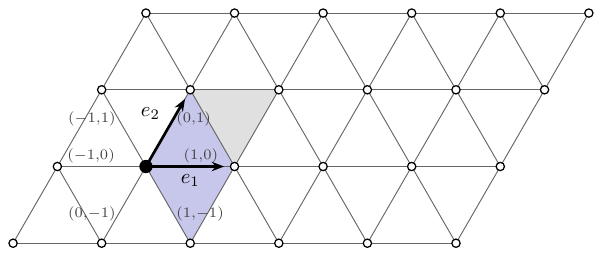}
\end{center}
\caption{Example \ref{Ex_trilattice}: the triangular lattice as a periodic triangulation with a single vertex
in a fundamental domain (shaded). Every vertex is a translate of the filled one by an integer combination of
$e_{1}$ and $e_{2}$, every face a translate of one of the two shaded faces, and the six labels at the filled
vertex are its whole rotation system.}%
\label{Fig_trilattice}%
\end{figure}

The identity at the end of the example is not an accident and will be used repeatedly: a
triangulation of the torus with $N$ vertices has $E=3N$ edges and $F=2N$ faces, so
$\sum_{v}\deg (v)=2E=6N$ and
\[
\sum_{v}\bigl( 6-\deg (v)\bigr) =0
\]
over a fundamental domain. Once it has been decided which degrees are to occur, this is one linear
condition on their densities; for two degrees it determines them outright.



\begin{lemma}
\label{Lemma_closure}Let $D$ be a triangulation of a closed disk with boundary cycle $b_{0}b_{1}\cdots b_{m-1}$,
$m\geq 16$, in which every $b_{i}$ has at least three neighbours. Then for every integer $p\geq m+1$ there is a
maximal planar graph $G$ such that

\begin{itemize}
\item[\textrm{(i)}] $D$ is a subgraph of $G$, every bounded face of $D$ is a face of $G$ and $|V(G)|=|V(D)|+p$;
\item[\textrm{(ii)}] every vertex of $V(G)\setminus V(D)$ has degree at least five in $G$;
\item[\textrm{(iii)}] every $b_{i}$ has degree $\deg_{D}(b_{i})+2\geq 5$ in $G$, and every interior vertex of $D$
has the same degree in $G$ as in $D$.
\end{itemize}
\end{lemma}

\begin{proof}
We build the \emph{cap}, a triangulation of the complementary disk, out of nested rings and one final
\emph{apex} vertex; without the apex the innermost ring would bound a face that is not a triangle. Let
$m=M_{1}\geq M_{2}\geq \cdots \geq M_{j}\geq 6$ be integers, let $R_{i}$ be a cycle $y^{i}_{0}\cdots y^{i}_{M_{i}-1}$ on
new vertices, and let $w$ be one further new vertex. Two consecutive cycles $x_{0}\cdots x_{M-1}$ and
$y_{0}\cdots y_{M'-1}$ with $M\geq M'$ are joined by a triangulated annulus as follows: choose
$0=q_{0}<q_{1}<\cdots <q_{M'-1}<M$ and join $y_{i}$ to $x_{q_{i}},x_{q_{i}+1},\ldots ,x_{q_{i+1}}$ (indices modulo $M$,
$q_{M'}=q_{0}+M$); the faces of the annulus are the triangles $y_{i}x_{t}x_{t+1}$ for $q_{i}\leq t<q_{i+1}$ and
$y_{i}y_{i+1}x_{q_{i+1}}$. Take $x=b$, $y=y^{1}$ with $M_{1}=m$ and $q_{i}=i$ for the first annulus, $x=y^{i}$,
$y=y^{i+1}$ for the following ones, and join $w$ to all of $R_{j}$. The result $G$ is a maximal planar graph:
the disk $D$, the $j$ annuli and the fan of $w$ cover it, and all faces are triangles. In the first annulus every
$b_{i}$ is joined to exactly two vertices of $R_{1}$, which is (iii); the interior vertices of $D$ are untouched. A
vertex $y^{i}_{t}$ is joined to its two neighbours on $R_{i}$, to at least two vertices of the cycle outside it and to
at least one vertex inside it (to $w$ if $i=j$), so its degree is at least five, and $\deg (w)=M_{j}\geq 6$: this is
(ii). Finally $|V(G)\setminus V(D)|=M_{1}+\cdots +M_{j}+1$.

It remains to realise every $p\geq m+1$. For $j=1$ we get $p=m+1$. For $p\geq m+7$ write $p-1-m=M_{2}+\cdots +M_{j}$
with $6\leq M_{i}\leq m$ non-increasing, which is possible for every $p-1-m\geq 6$: if $p-1-m\leq m$ take one summand;
otherwise write $p-1-m=am+r$ with $0\leq r<m$ and take $a$ summands $m$ when $r=0$, $a$ summands $m$ and one summand
$r$ when $r\geq 6$, and $a-1$ summands $m$ followed by $\lceil (m+r)/2\rceil ,\lfloor (m+r)/2\rfloor $ when
$1\leq r\leq 5$ (both lie in $[6,m]$ because $m\geq 16$). For the five remaining values $p=m+2,\ldots ,m+6$ we start
from $j=1$ and \emph{split} the vertex $w$ repeatedly. Splitting a vertex $v$ of degree $d\geq 6$ whose neighbours in
cyclic order are $z_{0},\ldots ,z_{d-1}$ means: delete $v$, add two adjacent vertices $v_{1},v_{2}$, join $v_{1}$ to
$z_{0},\ldots ,z_{a-1}$ and $v_{2}$ to $z_{a-1},z_{a},\ldots ,z_{0}$, where $4\leq a\leq d-2$. This keeps a
maximal planar graph, adds one vertex, gives $v_{1},v_{2}$ the degrees $a+1\geq 5$ and $d-a+3\geq 5$, raises the
degrees of $z_{0}$ and $z_{a-1}$ by one and changes nothing else; in particular (iii) is preserved. Splitting off
$a=4$ five times in succession is possible because $\deg (w)=m\geq 16$: the successive apex degrees are
$m,m-1,\ldots ,m-4$, all at least six, and the five vertices split off have degree five.
\end{proof}

\begin{figure}[htbp]
\begin{center}
\includegraphics[width=0.46\textwidth]{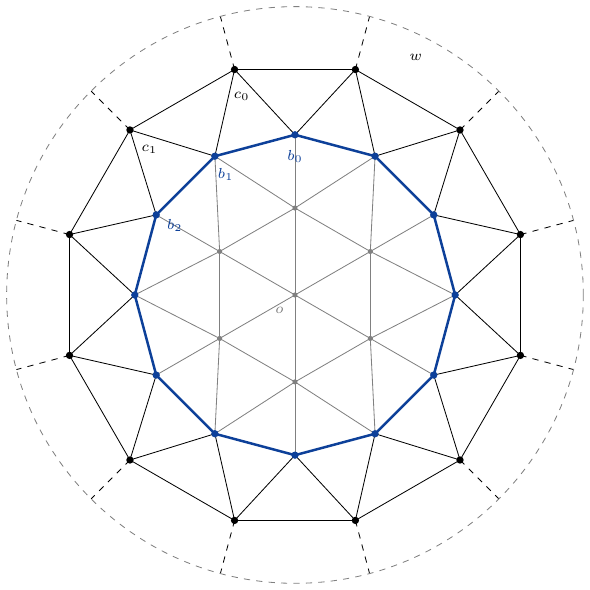}
\end{center}
\caption{The closure of Lemma \ref{Lemma_closure} with one ring: the triangulated disk $D$ with boundary cycle
$b_{0}b_{1}\cdots b_{m-1}$ (in bold), the ring $c_{0}c_{1}\cdots c_{m-1}$ with $c_{i}$ adjacent to $b_{i},b_{i+1}$ and
to $c_{i\pm 1}$, and the apex $w$, drawn as the point at infinity, adjacent to all $c_{i}$. Every $b_{i}$ gains
exactly two neighbours and every $c_{i}$ has degree five.}%
\label{Fig_closure}%
\end{figure}

\begin{lemma}
\label{Lemma_periodic}Let $\delta \in \{3,4,5\}$ and $k\geq 0$, and let $T$ be a periodic triangulation of the
plane with minimum degree $\delta $ in which the vertices of degree $j$ have density $d_{j}$. Write
\[
w_{k}(T)=\max_{p}\ \sum_{j=p}^{p+k}d_{j}
\]
for the largest density of a window of $k+1$ consecutive degrees; the maximum exists because a periodic
triangulation has only finitely many degrees. Then for every sufficiently large $n$ there is a
maximal planar graph $G$ of order $n$ with $\delta (G)=\delta $ and
\[
\mathrm{sp}(G,k)\ \leq \ w_{k}(T)\,n+O(\sqrt{n}).
\]
If the degrees of $T$ have the extremal densities of $(\delta ,k)$, then $w_{k}(T)=c_{\delta ,k}$ and
consequently
\[
\mathrm{MP}(n,\delta ,k)=c_{\delta ,k}\,n+O(\sqrt{n})
\]
for every sufficiently large $n$; in particular
$\lim_{n\to \infty }\mathrm{MP}(n,\delta ,k)/n=c_{\delta ,k}$.
\end{lemma}

\begin{proof}
Here $T$ is drawn in the plane with straight edges, as a periodic triangulation always may be, and the
\emph{length} of an edge is the Euclidean length of the segment drawing it; since $T$ has only finitely many
edges up to translation, the lengths take only finitely many values and the largest, $\ell $, exists. For the
triangular lattice spanned by $(1,0)$ and $(\frac{1}{2},\frac{\sqrt{3}}{2})$ every edge joins two points at
distance one, so $\ell =1$; for the square grid with the diagonal of each cell added, the edges have lengths
$1$ and $\sqrt{2}$, so $\ell =\sqrt{2}$.

Let $\Lambda $ be the lattice of $T$, let $\ell $ be the largest edge length of $T$, and fix a point $o$. For
$\rho >2\ell $ let $D_{\rho }$ be the union of the faces of $T$ contained in the closed disk $B(\rho )$ of radius
$\rho $ around $o$, and let $D_{\rho }^{\prime }$ be the component of $D_{\rho }$ (in the sense of adjacency of faces
across edges) containing $o$. Every face meeting $B(\rho -\ell )$ has all its vertices in $B(\rho )$, so it lies in
$D_{\rho }$ and, being connected to $o$ through faces meeting $B(\rho -\ell )$, in $D_{\rho }^{\prime }$. Hence every
vertex $x$ with $|x-o|<\rho -\ell $ lies in $D_{\rho }^{\prime }$ together with its whole link, i.e.\ is an interior
vertex of $D_{\rho }^{\prime }$. Since every face not in $D_{\rho }$ meets the exterior of $B(\rho )$, and every other component of $D_{\rho }$ is
adjacent across an edge only to faces not in $D_{\rho }$ (a face of $D_{\rho }$ adjacent to it would belong to it), the
complement of $D_{\rho }^{\prime }$ is connected; so $D_{\rho }^{\prime }$ is a union of disks glued at
\emph{pinch vertices}, a pinch vertex being a vertex $v$ at which the faces of the region containing $v$ form more than
one arc of the rotation at $v$, and its boundary is a single closed walk. We now clean the boundary: if a
boundary vertex lies in a single face (an ear), remove that face; at a pinch vertex, remove all faces of the arcs other
than the one containing the faces that meet $B(\rho -\ell )$. Two properties are preserved by every step: (a) every vertex $x$ with
$|x-o|<\rho -\ell $ lies in the region together with its whole star, and (b) the faces of the region that meet
$B(\rho -\ell )$ form an edge-connected set containing $o$ (any two of them are joined by a path inside
$B(\rho -\ell )$ avoiding vertices). Indeed, an ear at a vertex $v$ cannot contain such an $x$: the two faces at $x$
containing the edge $xv$ both lie in the region by (a), so $v$ would lie in two faces. An ear is attached to the rest
along a single edge, so removing it preserves (b); and at a pinch vertex the set in (b) lies on one side, which is the
side that is kept, so the removed side contains no face meeting $B(\rho -\ell )$ and hence, by (a), no vertex of
$B(\rho -\ell )$. Every removed face therefore has all its vertices in the annulus
$\rho -\ell \leq |x-o|\leq \rho $, and every vertex of $B(\rho -\ell )$ stays interior. Each step removes at least one face, so the process stops after finitely many steps, with a region
$\Delta _{\rho }$ having neither an ear nor a pinch vertex. At every boundary vertex of $\Delta _{\rho }$ the faces
therefore form a single arc of at least two faces, so the boundary walk visits every vertex exactly once: it is a simple
cycle $b_{0}b_{1}\cdots b_{m-1}$. A finite union of closed faces whose complement is connected and whose boundary is a
simple cycle is a closed disk, so $\Delta _{\rho }$ is a triangulated disk, and every $b_{i}$ lies in the annulus and
in at least two faces of $\Delta _{\rho }$, hence has at least three neighbours in $\Delta _{\rho }$. The annulus contains $O(\rho )$ vertices because $T$ has
bounded vertex density; so $m=O(\rho )$, and
$B(\rho -\ell )\cap V(T)\subseteq V(\Delta _{\rho })\subseteq B(\rho )\cap V(T)$ gives
$|V(\Delta _{\rho })|+m+1=\Theta (\rho ^{2})$.

The closed polygonal curve $b_{0}b_{1}\cdots b_{m-1}$ encircles $B(\rho -\ell )$, so its length is at least
$2\pi (\rho -\ell )$ while each of its $m$ edges has length at most $\ell $; hence $m\geq 2\pi (\rho -\ell )/\ell $
and, with $m=O(\rho )$ from above, $m=m_{\rho }=\Theta (\rho )$. Put
$n_{0}(\rho )=|V(\Delta _{\rho })|+m_{\rho }+1$; then
$n_{0}(\rho )=\Theta (\rho ^{2})$, and since $\Delta _{\rho }$ and $\Delta _{\rho +\ell }$ differ only inside the annulus
$\rho -\ell \leq |x-o|\leq \rho +\ell $, which contains $O(\rho )$ vertices, we have
$n_{0}(\rho +\ell )-n_{0}(\rho )=O(\rho )$. Fix $\rho _{0}>2\ell $ and let $\rho _{i}=\rho _{0}+i\ell $. Given
$n$ large, let $i$ be the largest index with $n_{0}(\rho _{i})\leq n$, which exists because
$n_{0}(\rho _{i})\to \infty $; then $\rho _{i}=\Theta (\sqrt{n})$ and
\[
r:=n-n_{0}(\rho _{i})\ <\ n_{0}(\rho _{i+1})-n_{0}(\rho _{i})\ =\ O(\rho _{i})\ =\ O(\sqrt{n}).
\]
Every boundary vertex of $D=\Delta _{\rho _{i}}$ has at least three neighbours in $D$, and $m\geq 16$ for large
$\rho _{i}$ because $m\geq 2\pi (\rho _{i}-\ell )/\ell $, so Lemma \ref{Lemma_closure} applies
with $p=m+1+r$ and produces a maximal planar graph $G$ of order exactly $|V(D)|+p=n$ in which the interior
vertices of $D$ keep their degrees in $T$, each $b_{i}$ has gained exactly two neighbours, and the $p$ vertices of the
cap have degree at least five. Hence $\delta (G)=\delta $, and only the $m+p=O(\sqrt{n})$ vertices of the boundary
and of the cap have degrees differing from those in $T$. Since $D$ contains all vertices of $B(\rho -\ell )$
and is contained in $B(\rho )$, it differs from a union of translates of the fundamental domain of $\Lambda $ by
$O(\rho )$ vertices, so for every degree $j$ the number of vertices of degree $j$ in $G$ is $d_{j}n+O(\sqrt{n})$,
where $d_{j}$ is the density of $j$ in $T$. A window of $k+1$ consecutive degrees therefore contains at most
$w_{k}(T)\,n+O(\sqrt{n})$ vertices of $G$, the error being the sum of the $k+1$ individual ones and $k$ being
fixed; as $n$ was an arbitrary large order, this is the first assertion. If the
degrees of $T$ have the extremal densities of $(\delta ,k)$, they are spaced by $h=k+1$, so such a window contains at
most one of them and $w_{k}(T)=\max (c_{\delta ,k},1-tc_{\delta ,k})=c_{\delta ,k}$, using
$1-tc_{\delta ,k}\leq c_{\delta ,k}$ (Proposition \ref{Prop_equality}(ii)); the lower bound of Theorems
\ref{Tm_planar} and \ref{Tm_planar34} then gives the equality.
\end{proof}

\begin{remark}
\label{Rem_patch}Instead of a Euclidean disk one may cut out an $a\times b$ block of fundamental domains and close the
resulting patch in the same way; this is what one would do to compute an example by hand, and we have carried it out
for all the triangulations of this section (with the labels normalised as in Section \ref{Sec_computer}, so that a block
of domains is a connected region). We prefer the disk because its boundary lies in a single annulus of width
$\ell $, which gives at once both $m=\Theta (\rho )$ and the estimate $d_{j}n+O(\sqrt{n})$ for the degree counts,
whereas a rectangular block needs a separate discussion of its four corners, where the boundary is not contained in
one annulus and the cleaning of the boundary can remove more than $O(\sqrt{n})$ vertices unless the aspect ratio is
controlled. For the same reason the disk makes the padding of Lemma \ref{Lemma_closure} unproblematic: the deficit is
$O(\sqrt{n})$ while the cap can be given any number of vertices.
\end{remark}

The lemma reduces the asymptotic problem to the existence of a periodic triangulation with the extremal densities
and minimum degree $\delta $ (the second condition is automatic when $\delta $ is the smallest extremal degree, as it
always is). For $k=0$ such triangulations were given by Caro and West \cite[Theorem 2.5]{CaroWest} for
$\delta =3,4,5$, and Section \ref{Sec_hex} gives an explicit one for $\delta =5$. We now give them for $k\geq 1$.

\subsection{What the closure does, and where the apex is}
\label{Sec_apex}

Lemma \ref{Lemma_closure} is the step at which an infinite pattern becomes a finite graph, and it is worth
looking at it slowly, because two things about it are easy to misread: what exactly is being closed, and where
the last vertex goes.

The input is a \emph{triangulated disk} $D$: a finite plane graph all of whose bounded faces are triangles,
whose unbounded face is bounded by a cycle $b_{0}b_{1}\cdots b_{m-1}$, the \emph{boundary cycle}. Such a $D$ is
not yet a maximal planar graph, because its unbounded face is an $m$-gon and not a triangle. To make it
one we must fill that $m$-gon with triangles, and we must do so without creating a vertex of small degree ---
otherwise the minimum degree, which is the whole point of $\mathrm{MP}(n,\delta ,k)$, would be destroyed at the
last moment. Filling the $m$-gon with a single vertex joined to all of $b_{0},\ldots ,b_{m-1}$ would be the
obvious thing to do, and it is what one does when the order is not prescribed; but it adds only one vertex, and
we need to be able to add any prescribed number $p$ of them.

The \emph{cap} of Lemma \ref{Lemma_closure}, drawn in Figure \ref{Fig_closure}, does the filling in
two stages. First come the \emph{rings}: cycles
$R_{1},R_{2},\ldots ,R_{j}$ of new vertices, drawn one inside the next, each joined to the one before it by a
triangulated annulus, with $R_{1}$ joined to the boundary cycle of $D$. Then comes the last vertex, the
\emph{apex} $w$, joined to every vertex of the innermost ring $R_{j}$. The rings carry the freedom: their
lengths $M_{1}\geq \cdots \geq M_{j}\geq 6$ can be chosen so that $M_{1}+\cdots +M_{j}+1$ is any prescribed
$p\geq m+1$, which is how the lemma reaches every order. The apex carries nothing but the closure itself: it is
the vertex that turns the last $m$-gon into a fan of triangles.

Two points about the apex answer the questions one naturally asks.

\emph{The apex is not on the boundary cycle.} It is a new vertex, adjacent to the innermost ring and to nothing
else. In the drawing of $D$ it lies in the unbounded face --- outside the boundary cycle --- and it is the only
vertex of $G$ that is at distance $j+1$ from $D$.

\emph{On the sphere, however, ``outside'' means nothing.} A maximal planar graph has no distinguished
face; a plane drawing has one only because we chose which face to send to infinity. Figure \ref{Fig_apex} shows
the same triangulation twice. On the left the disk $D$ is drawn in the usual way, the ring around it, and the
apex is the point at infinity: the short radial edges at the ring vertices all meet there. On the right the
same graph is drawn with a face of $D$ as the unbounded one, and now the apex sits in the middle of the picture,
an ordinary vertex of degree $M_{j}$ like any other. Nothing has changed but the choice of outer face.

\begin{figure}[htbp]
\begin{center}
\includegraphics[width=0.47\textwidth]{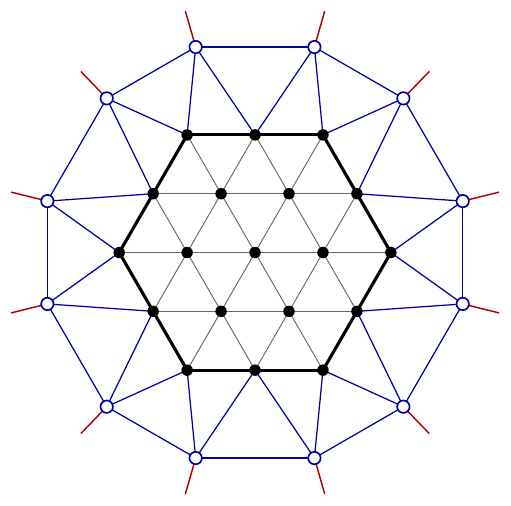}\hfill
\includegraphics[width=0.45\textwidth]{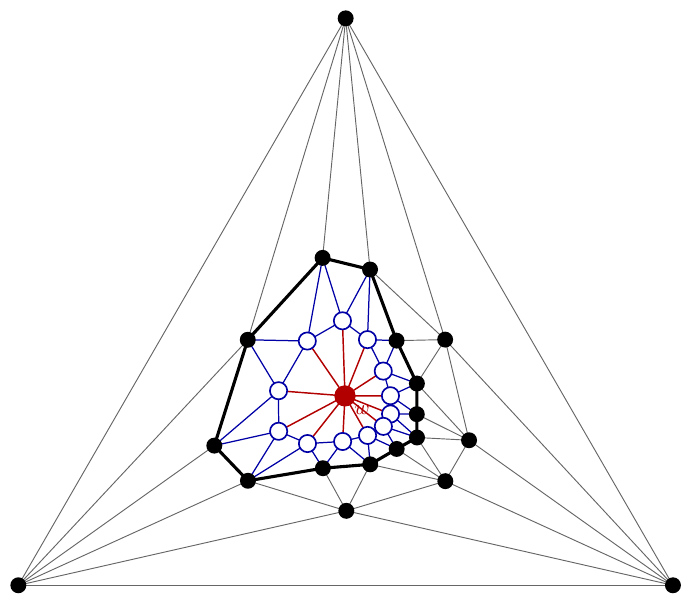}
\end{center}
\caption{The closure of Lemma \ref{Lemma_closure}, with one ring, drawn twice. Left: the triangulated disk $D$
(black vertices, its boundary cycle in bold), the ring (white vertices), and the apex $w$ as the point at
infinity --- the short radial edges at the ring all run to it. Right: the same maximal planar graph drawn
with a triangle of $D$ as the unbounded face; now $w$ (red) is in the middle, the ring surrounds it, and $D$ is
spread around the outside. The apex lies on no boundary cycle: whether it looks external depends only on which
face is drawn as the outer one.}%
\label{Fig_apex}%
\end{figure}

\begin{example}
\label{Ex_apex}Here is the whole passage from an infinite pattern to a finite maximal planar graph, at a
size one can check by hand. Take the triangular lattice, that is the $6$-regular periodic triangulation of the
plane, and let $D$ be the ball of radius two around a vertex: the vertex, its six neighbours and the twelve
vertices at distance two. It is a triangulated disk with
\[
19\ \text{vertices},\qquad 24\ \text{triangles},\qquad \text{a boundary cycle of length }m=12,
\]
its seven interior vertices having degree six and each of its twelve boundary vertices lying in at least two
triangles, as the closure requires.

Close it with one ring $c_{0}c_{1}\cdots c_{11}$, joining $c_{i}$ to $b_{i}$, $b_{i+1}$ and to $c_{i\pm 1}$, and
with the apex $w$ joined to all twelve $c_{i}$. The result $G$ has
\[
n=19+12+1=32\ \text{vertices},\qquad 90=3n-6\ \text{edges},\qquad 60=2n-4\ \text{triangular faces},
\]
so $|V|-|E|+|F|=32-90+60=2$ and $G$ is a maximal planar graph. Its degrees are
\[
5^{(18)}\ 6^{(13)}\ 12^{(1)} ,
\]
the seven interior vertices of $D$ keeping the degree six they had in the lattice, every boundary vertex gaining
exactly the two ring neighbours promised by Lemma \ref{Lemma_closure}(iii), every ring vertex having degree
five, and the apex having degree $m=12$. In particular $\delta (G)=5$: the minimum degree of the pattern
survives the closure, which is what the lemma is for. Figure \ref{Fig_apex} is this example.

Two hypotheses of the lemma are not needed for an illustration as small as this one. The bound $m\geq 16$ and
the freedom in the number $p$ of added vertices are there to realise \emph{every} order $p\geq m+1$; here we
close with a single ring and $p=m+1=13$, for which one ring and an apex suffice whatever $m$ is.
\end{example}

Two remarks on the figure are worth making, because they are the traps. First, the apex cannot be added
to the left panel at a finite point: if the disk and the ring are drawn as they are there, then wherever one
puts the apex and joins it to the twelve ring vertices by segments, the picture acquires crossings. We tested
$19876$ positions on a grid covering a square of side $28$ about the centre, and the fewest crossings any of
them produced was $38$. This says nothing against F\'{a}ry's theorem, which guarantees that the graph has
\emph{some} straight-line plane drawing: the right panel is one, and there the apex does sit at a finite
point. What fails is only the attempt to keep the left panel and add the apex to it, and that is why the left
panel sends the apex to infinity instead. Secondly, the whole construction is what one does to a \emph{torus}
triangulation as well: a periodic triangulation of the plane is the lift of a triangulation of the torus, one
cuts a large disk out of that lift, and the cap closes the disk into a sphere. The lattice of this example is
the lift of the $6$-regular triangulation of the torus with any $a\times b$ fundamental domain; Lemma
\ref{Lemma_periodic} performs the same passage for an arbitrary periodic triangulation and for every
sufficiently large order.

\begin{remark}
\label{Rem_fan}The apex is not the only way of closing the last $m$-gon, and the other way is in this paper
too. Instead of adding a new vertex joined to all of $b_{0},\ldots ,b_{m-1}$, one may \emph{fan} the
$m$-gon from one of its own vertices: the $m-3$ diagonals $b_{0}b_{2},b_{0}b_{3},\ldots ,b_{0}b_{m-2}$
triangulate it with no new vertex at all. The two are exchanged at a fixed price. An apex adds one vertex,
of degree $m$, and raises every $\deg (b_{i})$ by one; a fan adds no vertex, raises $\deg (b_{0})$ by
$m-3$, raises $\deg (b_{i})$ by one for $2\leq i\leq m-2$, and leaves $b_{1}$ and $b_{m-1}$ alone. In both
counts the two agree with $e=3n-6$: the fan adds $m-3$ edges and no vertex, the apex adds $m$ edges and
one vertex, and the difference $m-(m-3)=3$ is exactly what one extra vertex costs in a maximal planar graph.

The construction of Theorem \ref{Tm_hex} takes the second course: its hexagon $w_{0}\cdots w_{5}$ is closed
by the three diagonals $w_{0}w_{2}$, $w_{0}w_{3}$, $w_{0}w_{4}$ and not by a seventh vertex. That is
visible in the degrees. Before the diagonals every $w_{r}$ has degree $L+3$, namely the $L+1$ vertices of
its side and its two neighbours $w_{r\pm 1}$; the fan then gives $\deg (w_{0})=L+3+(m-3)=L+6$,
$\deg (w_{2})=\deg (w_{3})=\deg (w_{4})=L+4$, and leaves $\deg (w_{1})=\deg (w_{5})=L+3$. In Figure
\ref{Fig_G4}, where $L=4$, these are the six labels $10,7,8,8,8,7$, and the three diagonals are the curves
drawn around the outside of the hexagon --- outside, because the faces they bound are the ones outside it.

\end{remark}

\subsection{Six constructions by hand}

In the \emph{square grid} $\mathbb{Z}^{2}$ every unit square is triangulated by one of its two diagonals; a vertex then
has degree $4$ plus the number of the four surrounding diagonals incident with it. In the \emph{triangular lattice},
with vertex set $\mathbb{Z}^{2}$ and $(i,j)$ adjacent to $(i,j)+d$ for $d\in \{\pm (1,0),\pm (0,1),\pm (-1,1)\}$,
every vertex has degree $6$; \emph{flipping} the edge from $v$ to $v+(1,0)$ replaces it by the edge from $v+(0,1)$ to
$v+(1,-1)$, lowering the degrees of the two old endpoints and raising those of the two new ones by one, and
\emph{stacking} a vertex into a face adds a vertex of degree $3$ and raises the degrees of the three corners by one.
Figures \ref{Fig_hand1} and \ref{Fig_hand2} show the six patterns.

The choice of move in each of the six cases is dictated by the smallest degree that is wanted. In
the square grid a vertex starts with its four grid neighbours and gains one for every diagonal around
it that ends at it, so the only degrees available are $4,5,6,7,8$ and the pattern of diagonals decides
which vertex receives which; this is the right vehicle when the minimum degree is to be $4$, as in
(a) and (f), and also when two degrees at distance three are wanted, as in (c), where each vertex of
the class of density $1/3$ takes all four diagonals and each vertex of the other two classes takes
exactly one. In the triangular lattice every vertex starts at $6$, and a flip lowers two degrees and
raises two degrees by one, so flips of pairwise disjoint rhombi produce degrees at distance at most
one from $6$ and nothing else; this is exactly what the pair $(5,1)$ asks for, and it is the content
of (b). Stacking is the
only one of the three moves that creates a vertex of degree $3$, and it raises three degrees at once,
so it is the vehicle for $\delta =3$: in (e) every upward triangle is stacked, which lifts every
lattice vertex from $6$ to $9$ in one stroke and gives as many new vertices of degree $3$ as there are
old ones, and in (d) only the triangles of every fourth strip are stacked, which leaves two rows out
of four at degree $6$ and produces the third degree of the pair $(3,2)$. In each case the densities
are then forced rather than chosen, by the identity $\sum_{v}(6-\deg (v))=0$ per fundamental domain.

%

\begin{proposition}
\label{Prop_hand}The following are periodic triangulations of the plane with the extremal densities of the indicated
pair, so that by Lemma \ref{Lemma_periodic} the corresponding bound is asymptotically tight.

\begin{itemize}
\item[\textrm{(a)}] $(\delta ,k)=(4,1)$, degrees $4,6,8$ in equal thirds, $\mathrm{MP}(n,4,1)=n/3+O(\sqrt{n})$:
the square grid in which the diagonal of the unit square with lower left corner $(x,y)$ joins $(x,y)$ to $(x+1,y+1)$
when $x$ is even and $y\not \equiv 2\ (\mathrm{mod}\ 3)$, or when $x$ is odd and $y\equiv 2\ (\mathrm{mod}\ 3)$, and
joins $(x+1,y)$ to $(x,y+1)$ otherwise.

\item[\textrm{(b)}] $(\delta ,k)=(5,1)$, degrees $5,7$ in equal halves, $\mathrm{MP}(n,5,1)=n/2+O(\sqrt{n})$:
the triangular lattice with every edge from $(2u,2w)$ to $(2u+1,2w)$ flipped, $u,w\in \mathbb{Z}$.

\item[\textrm{(c)}] $(\delta ,k)=(5,2)$, degrees $5,8$ in the proportion $2{:}1$, $\mathrm{MP}(n,5,2)=2n/3+O(\sqrt{n})$:
the square grid in which the diagonal of the unit square with lower left corner $(x,y)$ joins $(x,y)$ to $(x+1,y+1)$ if
$x-y\equiv 0\ (\mathrm{mod}\ 3)$ and joins $(x+1,y)$ to $(x,y+1)$ otherwise.

\item[\textrm{(d)}] $(\delta ,k)=(3,2)$, degrees $3,6,9$ in equal thirds, $\mathrm{MP}(n,3,2)=n/3+O(\sqrt{n})$: the
triangular lattice with a vertex stacked into every face of the strip between the rows $j=4w$ and $j=4w+1$, $w\in
\mathbb{Z}$.

\item[\textrm{(e)}] $(\delta ,k)=(3,5)$, degrees $3,9$ in equal halves, $\mathrm{MP}(n,3,5)=n/2+O(\sqrt{n})$: the
triangular lattice with a vertex stacked into every face of the form $\{v,v+(1,0),v+(0,1)\}$.

\item[\textrm{(f)}] $(\delta ,k)=(4,3)$, degrees $4,8$ in equal halves, $\mathrm{MP}(n,4,3)=n/2+O(\sqrt{n})$: the
square grid in which every diagonal joins the two corners of even coordinate sum.
\end{itemize}
\end{proposition}

\begin{proof}
Each pattern is invariant under a lattice of translations and is a triangulation of the plane by construction (a
diagonal in every unit square; flips of pairwise disjoint rhombi; stackings into faces), so only the degrees have to
be computed.

(a) The pattern has period $(2,0),(0,3)$. For a vertex $(x,y)$ the four surrounding squares have lower left corners
$(x-1,y-1),(x,y-1),(x-1,y),(x,y)$, and the rule says which of their diagonals is incident with $(x,y)$. Going through
the six classes $(x\bmod 2,y\bmod 3)$ one finds the degrees $8,4$ for $(0,0),(1,0)$; $6,6$ for $(0,1),(1,1)$; $4,8$ for
$(0,2),(1,2)$. Thus in the rows $y\equiv 0,2\ (\mathrm{mod}\ 3)$ the vertices of degree $8$ (the crossings of two full
diagonals) alternate with vertices of degree $4$, staggered by one between the two rows, and the rows $y\equiv 1$
consist of vertices of degree $6$; the densities are $1/3,1/3,1/3$.

(b) The flipped edge from $(2u,2w)$ to $(2u+1,2w)$ lies in the rhombus with the further vertices $(2u,2w+1)$ and
$(2u+1,2w-1)$. The four vertices of this rhombus lie in the four cosets of the lattice $2\mathbb{Z}^{2}$, so the rhombi
are pairwise disjoint and partition $\mathbb{Z}^{2}$: a vertex in an even row is an endpoint of exactly one flipped
edge and a vertex in an odd row is a further vertex of exactly one rhombus. Hence even rows have degree $5$ and odd
rows degree $7$, each of density $1/2$.

(c) Let $c=x-y\bmod 3$ be the class of $(x,y)$; the two corners $(x,y)$, $(x+1,y+1)$ of a unit square have the same class
$c$ and the other two have the classes $c+1$ and $c-1$. By the rule, the square with lower left corner $(x,y)$ contributes
to the degree of $(x,y)$ if and only if $c=0$; the square with lower left corner $(x-1,y-1)$ (class $c$) likewise;
the square with lower left corner $(x-1,y)$ has class $c-1$ and its diagonal is incident with $(x,y)$ if and only if
$c-1\not \equiv 0$; the square with lower left corner $(x,y-1)$ has class $c+1$ and its diagonal is incident with $(x,y)$
if and only if $c+1\not \equiv 0$. Hence the vertices of class $0$ have degree $4+4=8$ and those of classes $1,2$ have
degree $4+1=5$: densities $1/3$ and $2/3$.

(d) A vertex of the triangular lattice lies in three faces of the strip above its row and three faces of the strip
below. A vertex in a row $\equiv 0\pmod 4$ lies in three stacked faces (the strip above it), a vertex in a row
$\equiv 1$ likewise (the strip below it), and vertices in rows $\equiv 2,3$ lie in no stacked face; the strip between
the rows $4w$ and $4w+1$ contains two faces per unit length, so the stacked vertices are as many as the vertices of two
rows. Per period of four rows we get two rows of degree $9$, two rows of degree $6$ and two rows' worth of vertices
of degree $3$: densities $1/3,1/3,1/3$.

(e) Every lattice vertex lies in exactly three faces of the form $\{v,v+(1,0),v+(0,1)\}$ (as $v$, as $v+(1,0)$ and as
$v+(0,1)$), so all lattice vertices get degree $9$; there is one such face per lattice vertex, so the stacked vertices
of degree $3$ are as many as the lattice vertices: densities $1/2,1/2$.

(f) Every unit square has exactly two corners of even coordinate sum, opposite to each other, so the rule is
consistent; a vertex of even coordinate sum receives the diagonal of each of its four squares and has degree $8$, a
vertex of odd coordinate sum receives none and has degree $4$: densities $1/2,1/2$.
\end{proof}

\begin{example}
\label{Ex_hand_a}The smallest $k$ occurring in Proposition \ref{Prop_hand} is $k=1$, in the patterns
(a) and (b). The pattern (a) is invariant under the translations $(2,0)$ and $(0,3)$, so a
fundamental domain is the set of the six vertices $(x,y)$ with $0\leq x<2$ and $0\leq y<3$, and their
degrees, read off from the rule, are
\[
\begin{array}{lll}
\deg (0,0)=8, & \deg (0,1)=6, & \deg (0,2)=4,\\
\deg (1,0)=4, & \deg (1,1)=6, & \deg (1,2)=8,
\end{array}
\]
so that the domain carries the degrees $4,6,8$ with multiplicities $2,2,2$. The check
$\sum_{v}(6-\deg (v))=2\cdot 2+2\cdot 0+2\cdot (-2)=0$ holds, the three densities are $1/3$ each, and
since consecutive degrees differ by $2$ while $k+1=2$, a window of two consecutive degrees contains at
most one of $4,6,8$; hence $w_{1}=1/3=c_{4,1}$. For the pattern (b) the periods are $(2,0)$ and
$(0,2)$, a fundamental domain is the four vertices $(x,y)$ with $0\leq x,y<2$, the two in the even row
have degree $5$ and the two in the odd row degree $7$, the check reads $2\cdot 1+2\cdot (-1)=0$, both
densities are $1/2$, and $7-5=2>1$ gives $w_{1}=1/2=c_{5,1}$.
\end{example}

Figure \ref{Fig_pat468} draws the pattern and one of its fundamental domains.

\begin{figure}[htbp]
\begin{center}
\includegraphics[width=0.36\textwidth]{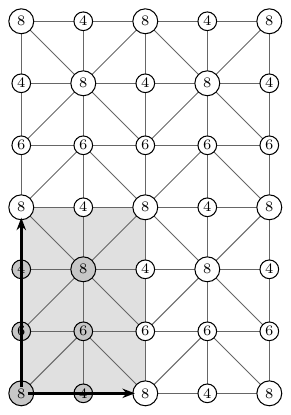}
\end{center}
\caption{The pattern (a) of Proposition \ref{Prop_hand}, with a fundamental domain shaded: the rows
$y\equiv 1\ (\mathrm{mod}\ 3)$ consist of vertices of degree $6$, the rows $y\equiv 0,2$ alternate the
degrees $8$ and $4$, staggered by one step, and the domain carries each of the three degrees exactly
twice.}%
\label{Fig_pat468}%
\end{figure}

\begin{figure}[htbp]
\begin{center}
\includegraphics[width=0.31\textwidth]{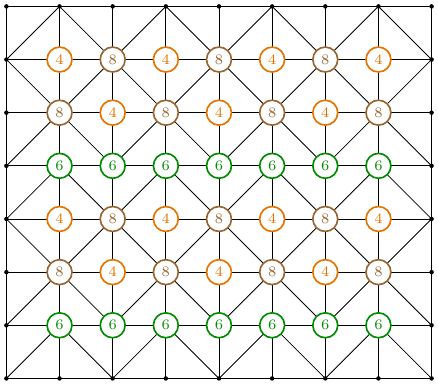}\hfill
\includegraphics[width=0.36\textwidth]{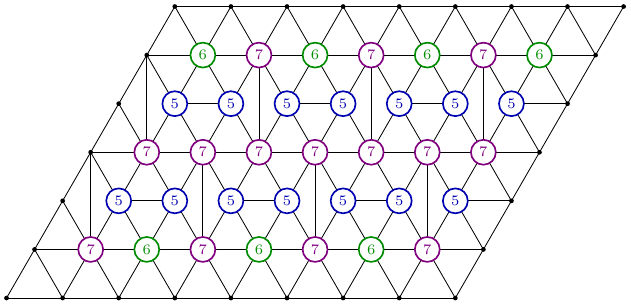}\hfill
\includegraphics[width=0.31\textwidth]{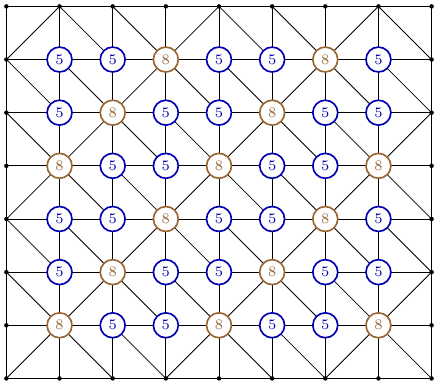}
\end{center}
\caption{The patterns (a), (b), (c) of Proposition \ref{Prop_hand}: degrees $4,6,8$ in the square grid; degrees $5,7$ in
the triangular lattice (the vertices shown with degree $6$ belong to rhombi cut by the window); degrees $5,8$ in the
square grid. Interior vertices are labelled with their degrees.}%
\label{Fig_hand1}%
\end{figure}

\begin{figure}[htbp]
\begin{center}
\includegraphics[width=0.36\textwidth]{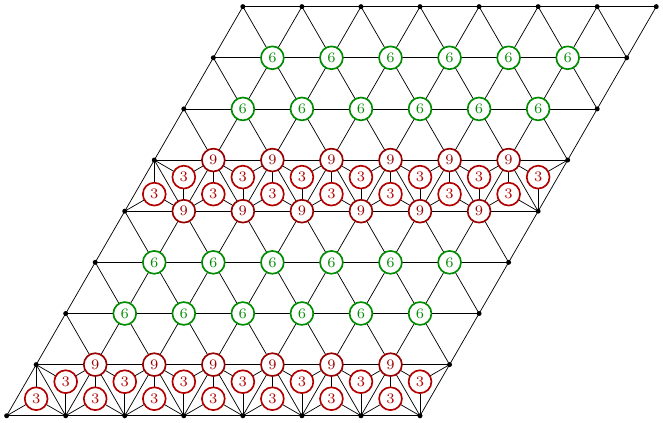}\hfill
\includegraphics[width=0.30\textwidth]{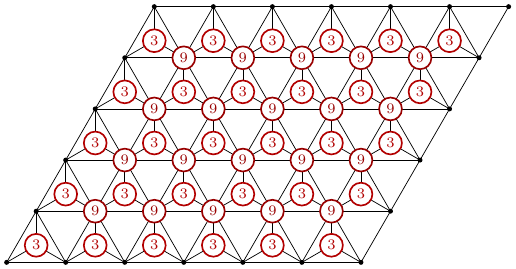}\hfill
\includegraphics[width=0.30\textwidth]{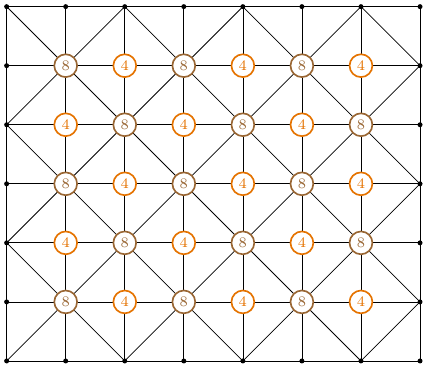}
\end{center}
\caption{The patterns (d), (e), (f) of Proposition \ref{Prop_hand}: degrees $3,6,9$ (stacked strips), $3,9$ (every
upward triangle stacked) and $4,8$ (the ``Union Jack'' grid).}%
\label{Fig_hand2}%
\end{figure}

\subsection{Minimum degree four: a construction for every $k\geq 3$}
\label{Sec_chain}

For $\delta =4$ one family settles all $k\geq 3$ at once, in the spirit of the unified construction of
\cite[Theorem 3.2(4)]{CaroLauriZarb} for maximal outerplanar graphs. The three smaller values of $k$ are
settled elsewhere: $(4,0)$ by the periodic triangulations of Caro and West \cite[Theorem 2.5]{CaroWest},
$(4,1)$ by Proposition \ref{Prop_hand}(a), and $(4,2)$ by the triangulation of the torus found by computer in
Section \ref{Sec_computer}. The extremal densities of $(4,k)$, $k\geq 3$, are
$(k-1)/(k+1)$ for the degree $4$ and $2/(k+1)$ for the degree $k+5$.

In words the construction is this. Only two degrees are to occur, namely $4$ and $k+5$, and the
vertices of degree $4$ are to outnumber the others in the ratio $(k-1):2$. The square grid supplies
the vertices of the large degree: a grid vertex keeps its four grid neighbours and has to collect
$k+1$ further ones from the four unit squares around it. The building block is one unit square with a
chain of $m$ new vertices strung across it, from the corner $Q$ to the opposite corner $S$, every
vertex of the chain being joined to the two ends $P$ and $R$ of the other diagonal. A vertex in the
middle of the chain then has two chain neighbours and the two vertices $P$ and $R$, which is already
four; the two vertices at the ends of the chain have only one chain neighbour each and are joined to
$Q$, respectively to $S$, instead. So every inserted vertex has degree exactly $4$, whatever $m$ is,
and the entire degree produced by the block is delivered to the grid.

A grid vertex $v=(x,y)$ lies in four unit squares. In two of them, the one above right of $v$ and the
one below left of it, $v$ is an end of the diagonal that the chain straddles and therefore sees every
one of the $m$ vertices of that chain; in the other two it sees only the single end of the chain that
reaches out to it. This gives $\deg (v)=4+m(y)+m(y-1)+2$, and the requirement $\deg (v)=k+5$ reads
$m(y)+m(y-1)=k-1$. For odd $k$ this is solved by the constant value $m=(k-1)/2$; for even $k$ the
number $k-1$ is odd and cannot be split into two equal parts, and the alternation of the rows in the
statement has no other purpose. A fundamental domain is accordingly one grid vertex together with the
chain of one square, or, for even $k$, two grid vertices together with the chains of two squares.



\begin{proposition}
\label{Prop_chain4}Let $k\geq 3$ and put $m(y)=(k-1)/2$ for every $y\in \mathbb{Z}$ if $k$ is odd, and $m(y)=k/2$ for
even $y$, $m(y)=k/2-1$ for odd $y$, if $k$ is even. In the square grid with vertex set $\mathbb{Z}^{2}$ insert into the
unit square with lower left corner $(x,y)$ a path $c_{1}c_{2}\cdots c_{m(y)}$ of new vertices, joining every $c_{i}$
to $P=(x,y)$ and to $R=(x+1,y+1)$, and in addition $c_{1}$ to $Q=(x+1,y)$ and $c_{m(y)}$ to $S=(x,y+1)$. The result is
a periodic triangulation of the plane in which every inserted vertex has degree $4$ and every grid vertex has degree
$k+5$, with the extremal densities of $(4,k)$. Consequently $\mathrm{MP}(n,4,k)=\frac{k-1}{k+1}n+O(\sqrt{n})$ for
every sufficiently large $n$ and every $k\geq 3$.
\end{proposition}

\begin{proof}
Draw the $c_{i}$ of the square with lower left corner $(x,y)$ on the segment from $Q$ to $S$, in this order, so that
$P$ and $R$ lie on opposite sides of that segment. All the edges listed are then drawn without crossings, and the
square is divided into the $2m(y)+2$ triangles
\[
PQc_{1},\quad QRc_{1},\quad Pc_{i}c_{i+1},\quad Rc_{i}c_{i+1}\ (1\leq i<m(y)),\quad RSc_{m(y)},\quad SPc_{m(y)}
\]
(for $m(y)=1$ the four triangles $PQc_{1},QRc_{1},RSc_{1},SPc_{1}$), so the plane is triangulated and the graph is
simple. An inserted vertex $c_{i}$ with $1<i<m(y)$ is adjacent to $P,R,c_{i-1},c_{i+1}$, and $c_{1}$ is adjacent to
$P,R,Q$ and to $c_{2}$ (to $S$ if $m(y)=1$), likewise $c_{m(y)}$; so every inserted vertex has degree $4$. A grid
vertex $v=(x,y)$ lies in four unit squares: it is the $P$-corner of the square $(x,y)$ and the $R$-corner of the square
$(x-1,y-1)$, hence adjacent to all $m(y)$ resp.\ all $m(y-1)$ vertices inserted there, and it is the $Q$-corner of the
square $(x-1,y)$ and the $S$-corner of the square $(x,y-1)$, hence adjacent to one inserted vertex in each. With its
four grid neighbours,
\[
\deg (v)=4+m(y)+m(y-1)+2=6+m(y)+m(y-1)=k+5 .
\]
The construction is invariant under the translations $(1,0)$ and $(0,1)$ for odd $k$, and under $(1,0)$ and $(0,2)$ for
even $k$. A fundamental domain therefore consists of one grid vertex and $m=(k-1)/2$ inserted vertices for odd $k$, and
of two grid vertices and $m(0)+m(1)=k-1$ inserted vertices for even $k$; in both cases the densities are
$\frac{k-1}{k+1}$ for the degree $4$ and $\frac{2}{k+1}$ for the degree $k+5$, which are the extremal densities of
$(4,k)$. Lemma \ref{Lemma_periodic} gives the asymptotics, and Theorem \ref{Tm_planar34} the matching lower bound.
\end{proof}

Figure \ref{Fig_chainblock} draws the block with $m=3$.

\begin{figure}[htbp]
\begin{center}
\includegraphics[width=0.5\textwidth]{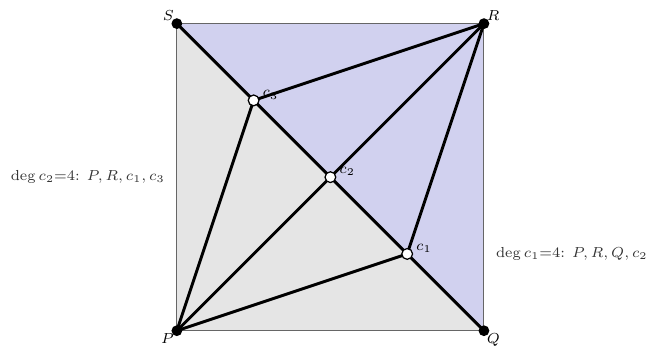}
\end{center}
\caption{The building block of Proposition \ref{Prop_chain4} with $m=3$: the square is fanned out from $P$
and from $R$, the chain $c_{1}c_{2}c_{3}$ is the seam between the two fans, and every inserted vertex has
exactly four neighbours whatever the length of the chain.}%
\label{Fig_chainblock}%
\end{figure}

Figure \ref{Fig_chain4} shows the cases $k=5$ and $k=6$. For $k=1$ and $k=2$ the pattern degenerates and the two
constructions of Proposition \ref{Prop_hand}(a) and of Section \ref{Sec_computer} are used instead.

\begin{example}
\label{Ex_chain3}Take $k=3$, the smallest value for which the pattern does not degenerate; here $m=1$
in every row, so a single vertex $c_{1}$ is inserted into each unit square and is joined to all four
of its corners. The periods are $(1,0)$ and $(0,1)$, and a fundamental domain consists of the two
vertices $(0,0)$ and $c_{1}$ of the square with lower left corner $(0,0)$, of degrees $k+5=8$ and $4$
respectively. The degrees of the domain are thus $4$ and $8$ with multiplicity one each,
$\sum_{v}(6-\deg (v))=2+(-2)=0$, and both densities are $1/2$, in agreement with
$\frac{k-1}{k+1}=\frac{2}{k+1}=\frac{1}{2}$. As $8-4=4>3=k$, no window of $k+1=4$ consecutive degrees
contains both, so $w_{3}=1/2=c_{4,3}$. Written as in Section \ref{Sec_computer}, with $0$ the
inserted vertex and $1$ the grid vertex, the domain is
\[
\begin{array}{l}
0:\ 1_{(1,1)}\ 1_{(0,1)}\ 1\ 1_{(1,0)};\\
1:\ 1_{(1,0)}\ 0\ 1_{(0,1)}\ 0_{(-1,0)}\ 1_{(-1,0)}\ 0_{(-1,-1)}\ 1_{(0,-1)}\ 0_{(0,-1)},
\end{array}
\]
with $N=2$ vertices, $E=6$ edges and $F=4$ faces. This instance is not new: mapping
$(x,y)\mapsto \bigl( \tfrac{x+y}{2},\tfrac{y-x}{2}\bigr) $ on the vertices of even coordinate sum and
$(x,y)\mapsto c_{1}$ of the square with lower left corner
$\bigl( \tfrac{x+y-1}{2},\tfrac{y-x-1}{2}\bigr) $ on the vertices of odd coordinate sum carries the
grid of Proposition \ref{Prop_hand}(f) onto the present triangulation, the four neighbours of a
vertex of odd coordinate sum being precisely the four corners of a unit square in the new coordinates.

The smallest even case is $k=4$, where $m(y)=2$ for even $y$ and $m(y)=1$ for odd $y$. The periods
are $(1,0)$ and $(0,2)$, and a fundamental domain consists of two grid vertices, of degree
$k+5=9$, together with the two inserted vertices of one square of an even row and the one inserted
vertex of one square of an odd row, all of degree $4$: the degrees are $4$ with multiplicity three and
$9$ with multiplicity two, $\sum_{v}(6-\deg (v))=3\cdot 2+2\cdot (-3)=0$, the densities are $3/5$ and
$2/5$, as $\frac{k-1}{k+1}$ and $\frac{2}{k+1}$ demand, and $9-4=5>4=k$ gives $w_{4}=3/5=c_{4,4}$.
\end{example}

Figure \ref{Fig_chain34} draws the two patterns with their fundamental domains; the
two panels show why the domain grows when the chain lengths of the rows differ.

\begin{figure}[htbp]
\begin{center}
\includegraphics[width=0.42\textwidth]{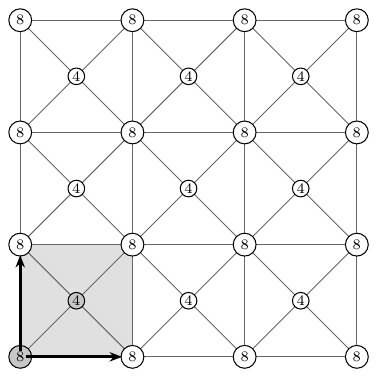}\hfill
\includegraphics[width=0.42\textwidth]{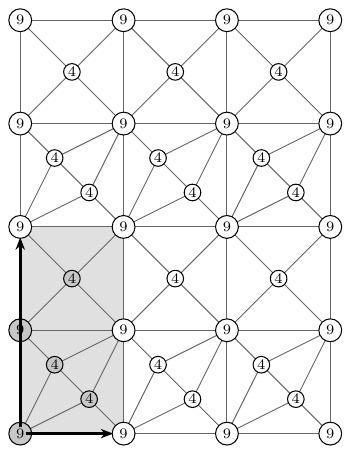}
\end{center}
\caption{The diagonal-chain triangulations of Example \ref{Ex_chain3}, with their fundamental domains shaded.
Left, $k=3$: one grid vertex of degree $8$ and one inserted vertex of degree $4$. Right, $k=4$: the domain is
two squares tall, because the two rows carry chains of different lengths, and it holds three vertices of
degree $4$ against two of degree $9$.}%
\label{Fig_chain34}%
\end{figure}

\begin{figure}[htbp]
\begin{center}
\includegraphics[width=0.42\textwidth]{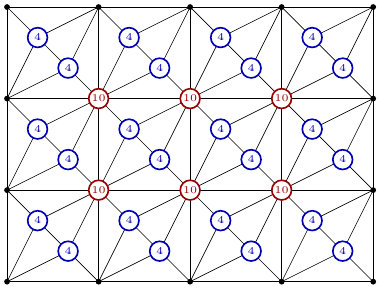}\hfill
\includegraphics[width=0.42\textwidth]{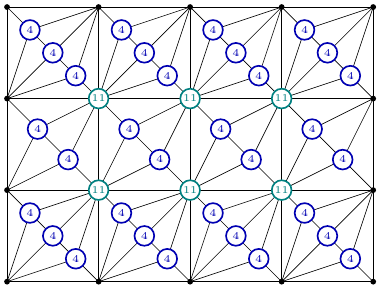}
\end{center}
\caption{The diagonal-chain triangulations of Proposition \ref{Prop_chain4} for $k=5$ (two inserted vertices in every
unit square; degrees $4$ and $10$) and for $k=6$ (three inserted vertices in the even rows, two in the odd ones;
degrees $4$ and $11$). Interior vertices are labelled with their degrees.}%
\label{Fig_chain4}%
\end{figure}

\subsection{Minimum degree five: wheels glued along their rims}
\label{Sec_wheel}

For $\delta =5$ there is a family of the same kind for every $k\geq 2$, and it has a suggestive
description: the vertices of degree $k+6$
are the hubs of wheels $W_{k+6}$, and the wheels are glued along their rims, every hub lying on the rims of two
further wheels. The extremal densities of $(5,k)$ are $\frac{k}{k+1}$ for the degree $5$ and $\frac{1}{k+1}$ for the
degree $k+6$. The two smaller values of $k$ are settled elsewhere: $(5,0)$ by the family of Section \ref{Sec_hex},
which also makes the error term explicit, and again by \cite[Theorem 2.5]{CaroWest}; and $(5,1)$ by
Proposition \ref{Prop_hand}(b). Note that a periodic triangulation using no degrees other than $5$ and $k+6$ has these densities
automatically, since $\sum_{v}(6-\deg (v))=0$ per fundamental domain.

Spelling this out: if a fraction $d$ of the vertices has degree $k+6$ and the remaining fraction
$1-d$ has degree $5$, then $(1-d)\cdot 1+d\cdot (-k)=0$, so $d=\frac{1}{k+1}$. The task is therefore
to build a triangulation in which one vertex in every $k+1$ has degree $k+6$ and every other vertex
has degree $5$. A hub together with its rim accounts for $k+7$ vertices, which is more than the $k+1$
vertices per hub that the density allows, so the wheels cannot be disjoint: some hubs must lie on the
rims of others. In the construction below this happens in the simplest possible way, every hub lying
on the rims of exactly two others, the ones directly above and below it, so that the wheels of one
column of the grid form a ladder.

The building block is again a unit square, now carrying two parallel chains
$a_{1}\cdots a_{m}$ and $b_{1}\cdots b_{m}$ that run alongside the diagonal $PR$, the first on the
side of $Q$ and the second on the side of $S$, joined to each other by rungs: a ladder of $2m$
vertices. An inserted vertex in the middle of a chain has two chain neighbours, two rungs and one
edge of the fan at $Q$ or at $S$, so it reaches degree $5$ at once. The four vertices at the ends of
the two chains are one edge short, and the last step of the construction is what repairs them: a
horizontal edge of the grid lies in exactly two triangles, one in the square above it and one in the
square below it, and flipping it inside that quadrilateral replaces it by the edge from $a_{1}$ of the
square above to $b_{m}$ of the square below. Each flip hands $a_{1}$ and $b_{m}$ their fifth
neighbour and takes one neighbour from each of the two grid vertices at its ends; without the flips a
grid vertex would have degree $k+8$ and the four chain ends would have degree $4$. The count at a
grid vertex is then as in Section \ref{Sec_chain}: two vertical grid neighbours, the two chain ends
$a_{1},b_{1}$ of the square of which $v$ is the corner $P$, the two chain ends $a_{m},b_{m}$ of the
square of which $v$ is the corner $R$, the whole $a$-chain of the square of which $v$ is the corner
$Q$, and the whole $b$-chain of the square of which $v$ is the corner $S$, that is
$6+m(y)+m(y-1)$. The requirement that this be $k+6$ reads $m(y)+m(y-1)=k$, which for even $k$ is met
by the constant $m=k/2$ and for odd $k$ forces the two rows of the period to differ, as in Section
\ref{Sec_chain}.

\begin{proposition}
\label{Prop_wheel}Let $k\geq 2$ and put $m(y)=k/2$ for every $y\in \mathbb{Z}$ if $k$ is even, and $m(y)=(k-1)/2$ for
even $y$, $m(y)=(k+1)/2$ for odd $y$, if $k$ is odd. In the square grid with vertex set $\mathbb{Z}^{2}$ insert into
the unit square with lower left corner $P=(x,y)$, writing $Q=(x+1,y)$, $R=(x+1,y+1)$, $S=(x,y+1)$ and $m=m(y)$, two
paths $a_{1}a_{2}\cdots a_{m}$ and $b_{1}b_{2}\cdots b_{m}$ of new vertices, and join

\begin{itemize}
\item[\textrm{(i)}] $Qa_{i}$ and $Sb_{i}$ for $1\leq i\leq m$ \emph{(the two fans)};
\item[\textrm{(ii)}] $Pa_{1}$, $Pb_{1}$, $Ra_{m}$, $Rb_{m}$ \emph{(the four ends)};
\item[\textrm{(iii)}] $a_{1}b_{1}$, and $a_{i}b_{i-1}$, $a_{i}b_{i}$ for $2\leq i\leq m$ \emph{(the rungs)}.
\end{itemize}

Keep the vertical edges of the grid, and replace every horizontal edge $(x,y)(x+1,y)$ by the edge joining the vertex
$a_{1}$ of the square with lower left corner $(x,y)$ to the vertex $b_{m(y-1)}$ of the square with lower left corner
$(x,y-1)$. The result $T_{k}$ is a simple periodic triangulation of the plane in which every inserted vertex has degree
$5$ and every grid vertex $v$ has degree
\[
\deg (v)=6+m(y)+m(y-1)=k+6 ,
\]
the closed neighbourhood of $v$ being a wheel $W_{k+6}$ two of whose rim vertices, the grid vertices above and below
$v$, are again hubs. The densities are $\frac{k}{k+1}$ for the degree $5$ and $\frac{1}{k+1}$ for the degree $k+6$,
so that $\mathrm{MP}(n,5,k)=\frac{k}{k+1}n+O(\sqrt{n})$ for every sufficiently large $n$ and every $k\geq 2$; the
same holds for $k=1$ by Proposition \ref{Prop_hand}(b).
\end{proposition}

\begin{proof}
Draw $a_{i}$ and $b_{i}$ at $P+t_{i}(1,1)+\varepsilon (1,-1)$ and $P+t_{i}(1,1)+\varepsilon (-1,1)$, where
$t_{i}=i/(m+1)$ and $\varepsilon >0$ is small: the two paths run alongside the diagonal $PR$, the $a$-path on the side
of $Q$ and the $b$-path on the side of $S$, and all the edges (i)--(iii) together with the four sides of the square are
drawn without crossings. Before the horizontal edges are treated, the square is divided into the $4m+2$ triangles
\[
Pa_{1}b_{1},\quad a_{i}a_{i+1}b_{i},\quad a_{i+1}b_{i}b_{i+1}\ (1\leq i<m),\quad Ra_{m}b_{m}
\]
between the two paths, the $m+1$ triangles $PQa_{1}$, $Qa_{i}a_{i+1}$ $(1\leq i<m)$, $Qa_{m}R$ on the side of $Q$, and
the $m+1$ triangles $PSb_{1}$, $Sb_{i}b_{i+1}$ $(1\leq i<m)$, $Sb_{m}R$ on the side of $S$; so the plane is
triangulated. A horizontal edge $(x,y)(x+1,y)$ lies in exactly two of these triangles, namely $PQa_{1}$ of the square
above it and $Sb_{m}R$ of the square below it, whose union is the quadrilateral with vertices $(x,y)$, $a_{1}$,
$(x+1,y)$, $b_{m(y-1)}$; replacing the diagonal $(x,y)(x+1,y)$ of this quadrilateral by the other one,
$a_{1}b_{m(y-1)}$, leaves a triangulation. Distinct horizontal edges give disjoint quadrilaterals, so all these flips
may be performed simultaneously, and the result is a triangulation of the plane. It is simple because the vertices
inserted in different squares are distinct and no two of the listed edges join the same pair.

Now the degrees. Let $1<i<m$. Then $a_{i}$ is adjacent to $Q$, to $a_{i-1}$ and $a_{i+1}$, and to $b_{i-1}$ and
$b_{i}$; and $b_{i}$ is adjacent to $S$, to $b_{i-1}$ and $b_{i+1}$, and to $a_{i}$ and $a_{i+1}$: degree $5$. The
vertex $a_{1}$ is adjacent to $Q$, $P$, $a_{2}$, $b_{1}$ and, after the flip, to the vertex $b_{m(y-1)}$ of the square
below; $a_{m}$ to $Q$, $R$, $a_{m-1}$, $b_{m-1}$, $b_{m}$; $b_{1}$ to $S$, $P$, $b_{2}$, $a_{1}$, $a_{2}$; and $b_{m}$
to $S$, $R$, $b_{m-1}$, $a_{m}$ and, after the flip, to the vertex $a_{1}$ of the square above: again degree $5$. (For
$m=1$ the two inserted vertices are $a_{1}$, adjacent to $Q,P,R,b_{1}$ and its flip partner below, and $b_{1}$,
adjacent to $S,P,R,a_{1}$ and its flip partner above.) A grid vertex $v=(x,y)$ is the $P$-corner of the square
$(x,y)$, the $Q$-corner of $(x-1,y)$, the $R$-corner of $(x-1,y-1)$ and the $S$-corner of $(x,y-1)$; it has lost its two
horizontal grid neighbours and kept the two vertical ones, and it is adjacent to $a_{1},b_{1}$ of the square $(x,y)$,
to all $m(y)$ vertices $a_{i}$ of the square $(x-1,y)$, to $a_{m},b_{m}$ of the square $(x-1,y-1)$ and to all $m(y-1)$
vertices $b_{i}$ of the square $(x,y-1)$. This gives $2+2+2+m(y)+m(y-1)=k+6$, and reading these neighbours in the
cyclic order in which they occur around $v$ exhibits the wheel: counterclockwise from the east,
\begin{align*}
&a_{1},b_{1}\ \text{of }(x,y);\quad (x,y+1);\quad a_{m},\ldots ,a_{1}\ \text{of }(x-1,y);\\
&b_{m},a_{m}\ \text{of }(x-1,y-1);\quad (x,y-1);\quad b_{1},\ldots ,b_{m}\ \text{of }(x,y-1),
\end{align*}
consecutive entries being adjacent (a rung, a fan edge, an end, a path edge, a flipped edge, and so on). Two of its rim
vertices, $(x,y\pm 1)$, are hubs, so the wheels of a column are glued along their rims.

Finally $T_{k}$ is invariant under the translations $(1,0)$ and $(0,1)$ for even $k$, and under $(1,0)$ and $(0,2)$ for
odd $k$. A fundamental domain consists of one grid vertex and the $2m=k$ inserted vertices of one square when $k$ is
even, and of two grid vertices and the $(k-1)+(k+1)=2k$ inserted vertices of two squares when $k$ is odd; in both cases
the densities are $\frac{1}{k+1}$ for the degree $k+6$ and $\frac{k}{k+1}$ for the degree $5$, the extremal densities
of $(5,k)$. Lemma \ref{Lemma_periodic} now gives the asymptotics for $k\geq 2$, and Theorem \ref{Tm_planar} the
matching lower bound; the case $k=1$ is Proposition \ref{Prop_hand}(b).
\end{proof}

Figure \ref{Fig_wheelblock} draws the block and the way consecutive wheels are
glued along their rims.

\begin{figure}[htbp]
\begin{center}
\includegraphics[width=0.46\textwidth]{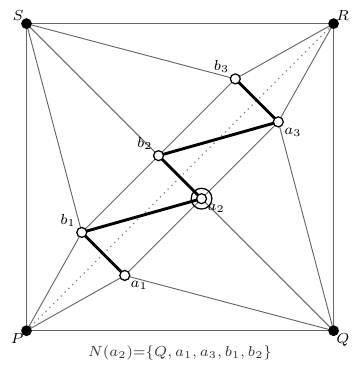}\hfill
\includegraphics[width=0.3\textwidth]{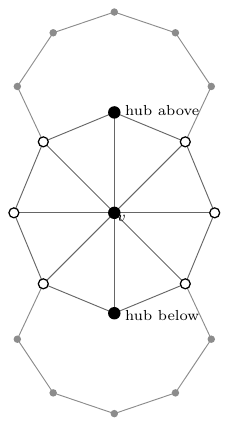}
\end{center}
\caption{Left: the building block of Proposition \ref{Prop_wheel} with $m=3$; the ladder of rungs is what
gives every inserted vertex degree five, and the dotted diagonal is a guide line, not an edge. Right: the
wheels of one column for $k=2$; each hub lies antipodally on the rim of the hub above it and of the hub below
it, and consecutive wheels share a spoke and two rim edges.}%
\label{Fig_wheelblock}%
\end{figure}

Figure \ref{Fig_wheel} shows the cases $k=6$ and $k=7$.

\begin{example}
\label{Ex_wheel2}Take $k=2$, the smallest value the proposition allows; here $m=1$ in every row, so
each unit square receives just the two vertices $a_{1}$ and $b_{1}$, joined to each other, to $P$ and
to $R$, and to $Q$ respectively $S$, and each of them is given its fifth neighbour by one of the
flipped horizontal edges. The periods are $(1,0)$ and $(0,1)$, and a fundamental domain consists of
the grid vertex $(0,0)$ and the two inserted vertices of the square with lower left corner $(0,0)$.
The degrees of the domain are $5$ with multiplicity two and $k+6=8$ with multiplicity one, so
$\sum_{v}(6-\deg (v))=2\cdot 1+(-2)=0$, the densities are $2/3$ and $1/3$, as $\frac{k}{k+1}$ and
$\frac{1}{k+1}$ demand, and since $8-5=3>2=k$ no window of $k+1=3$ consecutive degrees contains both:
$w_{2}=2/3=c_{5,2}$. Written as in Section \ref{Sec_computer}, with $0=a_{1}$, $1=b_{1}$ and $2$ the
hub, the domain is
\[
\begin{array}{l}
0:\ 2_{(1,1)}\ 1\ 2\ 1_{(0,-1)}\ 2_{(1,0)};\\
1:\ 2_{(1,1)}\ 0_{(0,1)}\ 2_{(0,1)}\ 2\ 0;\\
2:\ 0\ 1\ 2_{(0,1)}\ 0_{(-1,0)}\ 1_{(-1,-1)}\ 0_{(-1,-1)}\ 2_{(0,-1)}\ 1_{(0,-1)},
\end{array}
\]
with $N=3$ vertices, $E=9$ edges and $F=6$ faces. The list of the hub is the general cyclic order of
the proof with $m=1$: its own $a_{1}$ and $b_{1}$, the hub above, the $a$-chain of the square to the
left, the $b$- and $a$-ends of the square below left, the hub below, and the $b$-chain of the square
below. The two hubs on the rim stand at cyclic distance four on a rim of eight, that is antipodally;
the same is true of the vertices of degree $8$ in the pattern of Proposition \ref{Prop_hand}(c), which
realises the same pair $(5,2)$.

The smallest odd case is $k=3$, where $m(y)=1$ for even $y$ and $m(y)=2$ for odd $y$. The periods are
$(1,0)$ and $(0,2)$, a fundamental domain consists of two hubs of degree $k+6=9$ and of the
$2+4=6$ inserted vertices of one square in each of the two rows, all of degree $5$; then
$\sum_{v}(6-\deg (v))=6\cdot 1+2\cdot (-3)=0$, the densities are $3/4$ and $1/4$, and $9-5=4>3=k$
gives $w_{3}=3/4=c_{5,3}$.
\end{example}

Figure \ref{Fig_wheel2} draws this triangulation, with the wheel of one grid vertex
picked out in bold.

\begin{figure}[htbp]
\begin{center}
\includegraphics[width=0.5\textwidth]{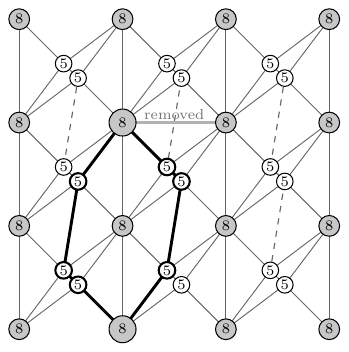}
\end{center}
\caption{The wheel-ladder triangulation of Example \ref{Ex_wheel2}, for $k=2$. Around every grid vertex there
is a wheel on eight rim vertices, drawn in bold for the vertex $(1,1)$; the grid vertices of one column lie on
each other's rims; and the dashed edges are the flips that replace the horizontal grid edges and hand two
units of degree to the ends of the chains.}%
\label{Fig_wheel2}%
\end{figure}

\begin{figure}[htbp]
\begin{center}
\includegraphics[width=0.44\textwidth]{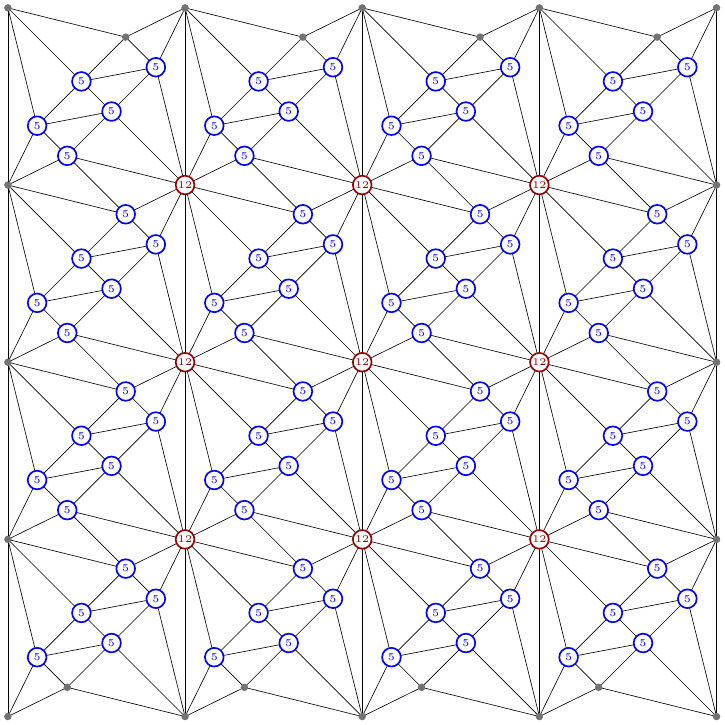}\hfill
\includegraphics[width=0.44\textwidth]{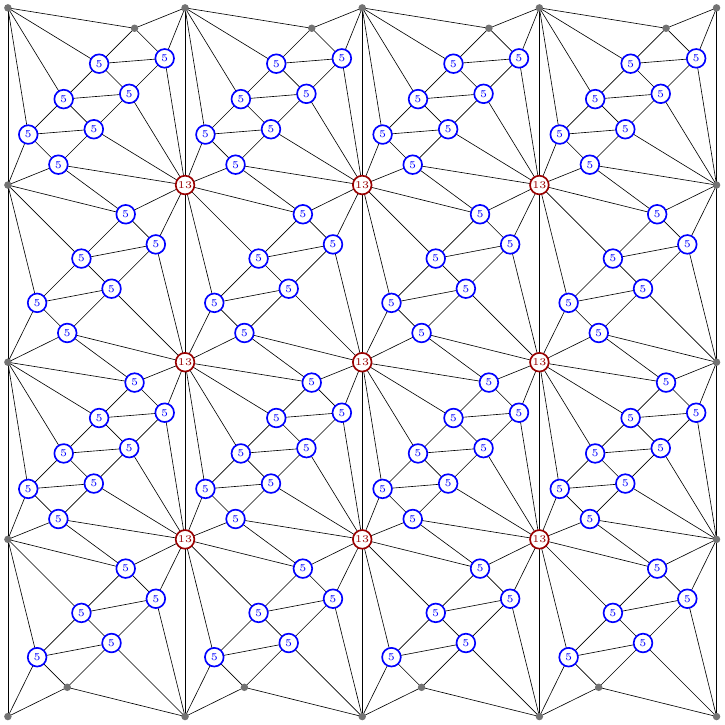}
\end{center}
\caption{The wheel-ladder triangulations of Proposition \ref{Prop_wheel} for $k=6$ (three inserted vertices on each
side of the diagonal of every unit square; degrees $5$ and $12$) and for $k=7$ (three on each side in the even rows,
four in the odd ones; degrees $5$ and $13$). The hubs are the grid vertices; interior vertices are labelled with their
degrees.}%
\label{Fig_wheel}%
\end{figure}

\subsection{Periodic triangulations found by computer}
\label{Sec_computer}

For the remaining pairs, namely $(3,k)$ with $k\in \{1,3,4,6,7,8\}$ and the pair $(4,2)$, we found periodic
triangulations by a search over triangulations of the torus. A
triangulation of the torus with $N$ vertices is given by its rotation system with \emph{translation labels}: the
vertices are $0,\ldots ,N-1$, and for every vertex $u$ we list its neighbours in cyclic order, each neighbour $v$
carrying a label $(p,q)\in \mathbb{Z}^{2}$ which says that the copy of $u$ in the fundamental domain is joined to the
copy of $v$ translated by $p$ times the first and $q$ times the second period vector (the label $(0,0)$ is omitted).
The conditions that make such data legitimate are the following, and all of them can be checked by
hand: (i) the label of an edge seen from its two ends is
opposite ($v$ carries $(p,q)$ in the list of $u$ if and only if $u$ carries $(-p,-q)$ in the list of $v$); (ii) tracing
the faces (from the edge $uv$ go to the neighbour of $v$ following $u$ in the rotation of $v$, and so on) produces $2N$
triangles, around each of which the labels sum to $(0,0)$, so that with $3N$ edges the surface is closed, orientable
(the rotations are consistent) and of Euler characteristic $0$, i.e.\ a torus, and the labels are the coordinates of
the deck transformations; (iii) the labels generate $\mathbb{Z}^{2}$, so that the lift is connected; (iv) no list
contains its own vertex with the label $(0,0)$ and no list contains the same labelled neighbour twice, so that the lift
is a simple graph; (v) the degrees are the lengths of the lists. Then the data define a periodic triangulation of the
plane with the given vertices as fundamental domain, and Lemma \ref{Lemma_periodic} applies. The search started
from the $a\times b$ triangular-lattice torus and performed random diagonal flips until the degree multiset was
the extremal one; the labels were then normalised so that a spanning tree of the quotient carries the label
$(0,0)$ and the basis of the period lattice was changed so that all labels lie in $\{-1,0,1\}^{2}$. Table
\ref{Tab_computer} records what was found and also checked by computer; a drawing of a disk of each
periodic triangulation, and the programs that check the systems, are in the supplementary material.

That the five conditions suffice is the content of the next lemma, and the seven pairs are then settled by
exhibiting seven systems that satisfy them.

\begin{lemma}
\label{Lemma_torus}Let $N\geq 1$ and let each of the vertices $0,\ldots ,N-1$ carry a cyclic list of labelled
neighbours $v_{(p,q)}$ with $v\in \{0,\ldots ,N-1\}$ and $(p,q)\in \mathbb{Z}^{2}$, and suppose that
\textup{(i)}--\textup{(v)} above hold. Then these data define a simple periodic triangulation $T$ of the
plane, invariant under a lattice $\Lambda $ of rank two with $N$ vertex orbits, in which the vertex of the
orbit of $u$ has degree the length of the list of $u$.
\end{lemma}

\begin{proof}
Let $T$ be the graph on $\{0,\ldots ,N-1\}\times \mathbb{Z}^{2}$ in which $u_{(a,b)}$ is joined to
$v_{(a+p,b+q)}$ for every entry $v_{(p,q)}$ of the list of $u$, and give $u_{(a,b)}$ the cyclic order its
list prescribes. Condition (i) makes the adjacency symmetric, so $T$ is a graph and the cyclic orders make it
a map; (iv) says that no vertex is joined to itself and no pair is joined twice, so $T$ is simple. The group
$\mathbb{Z}^{2}$ acts on $T$ by $u_{(a,b)}\mapsto u_{(a+a',b+b')}$, freely and by automorphisms of the map,
with exactly $N$ orbits of vertices, and the quotient map has $N$ vertices and, by (v), the prescribed
degrees, hence $3N$ edges.

Trace the faces of the quotient. By (ii) this produces $2N$ triangles, so the quotient is a closed surface
with $N-3N+2N=0$ as Euler characteristic, orientable because the rotations are consistent: a torus. Its
universal cover is the plane, and $T$ is the lift of the quotient map to that cover, hence a triangulation of
the plane; the deck group is the group generated by the labels, which by (iii) is all of $\mathbb{Z}^{2}$, so
$T$ is invariant under a lattice $\Lambda $ of rank two. The condition in (ii) that the labels sum to $(0,0)$
around every face is what makes each traced face close up in the lift rather than run off to infinity, so the
faces of $T$ are triangles. Finally a fundamental domain of $\Lambda $ meets each of the $N$ orbits once, so
the degrees of $T$ take the values of (v) with the stated multiplicities.
\end{proof}

\begin{proposition}
\label{Prop_computer}For each of the seven pairs $(\delta ,k)$ of Table \ref{Tab_computer} there is a periodic
triangulation $T_{\delta ,k}$ of the plane with minimum degree $\delta $ and $w_{k}(T_{\delta ,k})=c_{\delta
,k}$. Consequently, by Lemma \ref{Lemma_periodic},
\[
\mathrm{MP}(n,\delta ,k)=c_{\delta ,k}\,n+O(\sqrt{n})
\]
for every sufficiently large $n$, and in particular $\lim_{n\to \infty }\mathrm{MP}(n,\delta ,k)/n=c_{\delta
,k}$.
\end{proposition}

\begin{proof}
The seven rotation systems are printed below, one for each pair, in the notation fixed above. Each of them
satisfies (i)--(v), which is a finite check: (i) and (iv) are read off the lists, (ii) is the face tracing,
(iii) asks that the labels generate $\mathbb{Z}^{2}$, and (v) is the list lengths. By Lemma
\ref{Lemma_torus} each therefore defines a simple periodic triangulation of the plane with the stated
fundamental domain. In each case the degree multiset of the domain is as recorded, its smallest entry is
$\delta $, and the degrees occurring are so far apart that no window of $k+1$ consecutive integers meets two
of them; hence $w_{k}$ is the largest density of a single degree, which is $c_{\delta ,k}$. Lemma
\ref{Lemma_periodic} then gives the displayed equality.
\end{proof}


\noindent{\boldmath\textbf{$(\delta ,k)=(3,1)$, $N=4$.}}\quad {\small\raggedright $0$:~$1$ $1_{(1,1)}$ $3$; $1$:~$1_{(-1,-1)}$ $2$ $2_{(1,1)}$ $1_{(1,1)}$ $0$ $3$ $2_{(0,-1)}$ $3_{(-1,-1)}$ $0_{(-1,-1)}$; $2$:~$1_{(-1,-1)}$ $2_{(-1,-1)}$ $3_{(-1,0)}$ $1_{(0,1)}$ $3_{(0,1)}$ $2_{(1,1)}$ $1$; $3$:~$1$ $0$ $1_{(1,1)}$ $2_{(1,0)}$ $2_{(0,-1)}$.\par}
\noindent The lengths of the lists are the degrees, giving the multiset $3^{1}5^{1}7^{1}9^{1}$; the smallest is $3$. Two of the
degrees occurring differ by at least $2>k$, so a window of $k+1$ consecutive integers meets only one of
them, and $w_{k}=\tfrac{1}{4}=c_{3,1}$.

\noindent{\boldmath\textbf{$(\delta ,k)=(3,3)$, $N=12$.}}\quad {\small\raggedright $0$:~$6$ $6_{(0,1)}$ $11$; $1$:~$4$ $2$ $9_{(0,-1)}$ $5_{(0,-1)}$ $3$ $7$ $9$; $2$:~$1$ $4$ $9_{(0,-1)}$; $3$:~$9$ $7$ $1$ $5_{(0,-1)}$ $9_{(0,-1)}$ $10_{(1,0)}$ $10_{(1,1)}$; $4$:~$9_{(0,-1)}$ $2$ $1$ $9$ $6$ $8$ $6_{(0,-1)}$; $5$:~$3_{(0,1)}$ $1_{(0,1)}$ $9$; $6$:~$0_{(0,-1)}$ $6_{(0,-1)}$ $8$ $4$ $9$ $4_{(0,1)}$ $8_{(0,1)}$ $6_{(0,1)}$ $0$ $11$ $11_{(0,-1)}$; $7$:~$9$ $1$ $3$; $8$:~$6$ $6_{(0,-1)}$ $4$; $9$:~$4_{(0,1)}$ $6$ $4$ $1$ $7$ $3$ $10_{(1,1)}$ $3_{(0,1)}$ $5$ $1_{(0,1)}$ $2_{(0,1)}$; $10$:~$10_{(0,-1)}$ $11$ $11_{(0,1)}$ $10_{(0,1)}$ $3_{(-1,0)}$ $9_{(-1,-1)}$ $3_{(-1,-1)}$; $11$:~$11_{(0,1)}$ $10$ $10_{(0,-1)}$ $11_{(0,-1)}$ $6$ $0$ $6_{(0,1)}$.\par}
\noindent The lengths of the lists are the degrees, giving the multiset $3^{5}7^{5}11^{2}$; the smallest is $3$. Two of the
degrees occurring differ by at least $4>k$, so a window of $k+1$ consecutive integers meets only one of
them, and $w_{k}=\tfrac{5}{12}=c_{3,3}$.

\noindent{\boldmath\textbf{$(\delta ,k)=(3,4)$, $N=15$.}}\quad {\small\raggedright $0$:~$3$ $7$ $4$ $13$ $2$ $9$ $12$ $11$; $1$:~$14$ $4$ $7$ $7_{(0,-1)}$ $11$ $12$ $5$ $6$; $2$:~$12$ $9$ $0$ $13$ $4$ $8$ $6$ $10$; $3$:~$7$ $0$ $11$; $4$:~$7$ $1$ $14$ $6$ $8$ $2$ $13$ $0$; $5$:~$12$ $6$ $1$; $6$:~$8$ $4$ $14$ $1$ $5$ $12$ $10$ $2$; $7$:~$7_{(1,0)}$ $11_{(1,0)}$ $7_{(0,-1)}$ $1$ $4$ $0$ $3$ $11$ $7_{(-1,0)}$ $11_{(0,1)}$ $1_{(0,1)}$ $7_{(0,1)}$ $11_{(1,1)}$; $8$:~$2$ $4$ $6$; $9$:~$2$ $12$ $0$; $10$:~$12$ $2$ $6$; $11$:~$7_{(0,-1)}$ $7_{(-1,-1)}$ $7_{(-1,0)}$ $7$ $3$ $0$ $12$ $1$; $12$:~$2$ $10$ $6$ $5$ $1$ $11$ $0$ $9$; $13$:~$0$ $4$ $2$; $14$:~$1$ $6$ $4$.\par}
\noindent The lengths of the lists are the degrees, giving the multiset $3^{7}8^{7}13^{1}$; the smallest is $3$. Two of the
degrees occurring differ by at least $5>k$, so a window of $k+1$ consecutive integers meets only one of
them, and $w_{k}=\tfrac{7}{15}=c_{3,4}$.

\noindent{\boldmath\textbf{$(\delta ,k)=(3,6)$, $N=7$.}}\quad {\small\raggedright $0$:~$5$ $1$ $2$; $1$:~$2$ $0$ $5$ $3$ $2_{(-1,-1)}$ $4_{(-1,0)}$ $5_{(-1,0)}$ $6_{(-1,0)}$ $2_{(-1,0)}$ $5_{(0,1)}$; $2$:~$3_{(1,1)}$ $5_{(1,1)}$ $1_{(1,0)}$ $6$ $5$ $0$ $1$ $5_{(0,1)}$ $4_{(0,1)}$ $1_{(1,1)}$; $3$:~$2_{(-1,-1)}$ $1$ $5$; $4$:~$1_{(1,0)}$ $2_{(0,-1)}$ $5$; $5$:~$2_{(0,-1)}$ $1_{(0,-1)}$ $2_{(-1,-1)}$ $3$ $1$ $0$ $2$ $6$ $1_{(1,0)}$ $4$; $6$:~$2$ $1_{(1,0)}$ $5$.\par}
\noindent The lengths of the lists are the degrees, giving the multiset $3^{4}10^{3}$; the smallest is $3$. Two of the
degrees occurring differ by at least $7>k$, so a window of $k+1$ consecutive integers meets only one of
them, and $w_{k}=\tfrac{4}{7}=c_{3,6}$.

\noindent{\boldmath\textbf{$(\delta ,k)=(3,7)$, $N=8$.}}\quad {\small\raggedright $0$:~$4$ $1$ $6$; $1$:~$6_{(0,1)}$ $2_{(1,1)}$ $4_{(1,1)}$ $3_{(1,1)}$ $6_{(1,1)}$ $5_{(1,0)}$ $4_{(1,0)}$ $7_{(1,0)}$ $6$ $0$ $4$; $2$:~$6_{(-1,0)}$ $4$ $1_{(-1,-1)}$; $3$:~$4$ $6$ $1_{(-1,-1)}$; $4$:~$3$ $1_{(-1,-1)}$ $2$ $6_{(-1,0)}$ $7$ $1_{(-1,0)}$ $5$ $6_{(0,1)}$ $1$ $0$ $6$; $5$:~$1_{(-1,0)}$ $6_{(0,1)}$ $4$; $6$:~$1$ $7_{(1,0)}$ $4_{(1,0)}$ $2_{(1,0)}$ $1_{(0,-1)}$ $4_{(0,-1)}$ $5_{(0,-1)}$ $1_{(-1,-1)}$ $3$ $4$ $0$; $7$:~$6_{(-1,0)}$ $1_{(-1,0)}$ $4$.\par}
\noindent The lengths of the lists are the degrees, giving the multiset $3^{5}11^{3}$; the smallest is $3$. Two of the
degrees occurring differ by at least $8>k$, so a window of $k+1$ consecutive integers meets only one of
them, and $w_{k}=\tfrac{5}{8}=c_{3,7}$.

\noindent{\boldmath\textbf{$(\delta ,k)=(3,8)$, $N=9$.}}\quad {\small\raggedright $0$:~$6$ $5$ $8$; $1$:~$5_{(1,0)}$ $8_{(1,1)}$ $6$; $2$:~$8$ $5_{(1,0)}$ $6$; $3$:~$5_{(1,1)}$ $6$ $8_{(1,1)}$; $4$:~$8_{(0,1)}$ $6$ $5_{(1,1)}$; $5$:~$1_{(-1,0)}$ $6_{(-1,0)}$ $2_{(-1,0)}$ $8_{(-1,0)}$ $4_{(-1,-1)}$ $6_{(-1,-1)}$ $3_{(-1,-1)}$ $8$ $0$ $6$ $7$ $8_{(0,1)}$; $6$:~$8_{(1,1)}$ $3$ $5_{(1,1)}$ $4$ $8_{(0,1)}$ $7$ $5$ $0$ $8$ $2$ $5_{(1,0)}$ $1$; $7$:~$6$ $8_{(0,1)}$ $5$; $8$:~$2$ $6$ $0$ $5$ $3_{(-1,-1)}$ $6_{(-1,-1)}$ $1_{(-1,-1)}$ $5_{(0,-1)}$ $7_{(0,-1)}$ $6_{(0,-1)}$ $4_{(0,-1)}$ $5_{(1,0)}$.\par}
\noindent The lengths of the lists are the degrees, giving the multiset $3^{6}12^{3}$; the smallest is $3$. Two of the
degrees occurring differ by at least $9>k$, so a window of $k+1$ consecutive integers meets only one of
them, and $w_{k}=\tfrac{2}{3}=c_{3,8}$.

\noindent{\boldmath\textbf{$(\delta ,k)=(4,2)$, $N=9$.}}\quad {\small\raggedright $0$:~$6$ $7$ $2$ $8$ $5$ $4$ $1$ $6_{(-1,-1)}$ $3$ $2_{(-1,0)}$; $1$:~$7_{(-1,-1)}$ $6_{(-1,-1)}$ $0$ $4$; $2$:~$3_{(1,0)}$ $8$ $0$ $7$ $8_{(1,1)}$ $6_{(1,0)}$ $0_{(1,0)}$; $3$:~$2_{(-1,0)}$ $0$ $6_{(-1,-1)}$ $8_{(-1,0)}$; $4$:~$1$ $0$ $5$ $7_{(-1,-1)}$; $5$:~$0$ $8$ $7_{(-1,-1)}$ $4$; $6$:~$7$ $0$ $2_{(-1,0)}$ $8_{(0,1)}$ $3_{(1,1)}$ $0_{(1,1)}$ $1_{(1,1)}$; $7$:~$0$ $6$ $1_{(1,1)}$ $4_{(1,1)}$ $5_{(1,1)}$ $8_{(1,1)}$ $2$; $8$:~$0$ $2$ $3_{(1,0)}$ $6_{(0,-1)}$ $2_{(-1,-1)}$ $7_{(-1,-1)}$ $5$.\par}
\noindent The lengths of the lists are the degrees, giving the multiset $4^{4}7^{4}10^{1}$; the smallest is $4$. Two of the
degrees occurring differ by at least $3>k$, so a window of $k+1$ consecutive integers meets only one of
them, and $w_{k}=\tfrac{4}{9}=c_{4,2}$.

Two things make these pairs harder than those of Sections \ref{Sec_chain} and \ref{Sec_wheel}. The
first is that for four of the seven the equality sequence has more than two degrees, four of them for
$(3,1)$ and three for $(3,3)$, $(3,4)$ and $(4,2)$, so that the densities are not settled by the
identity $\sum_{v}(6-\deg (v))=0$ alone and a pattern has to hit all of them at once. The second,
for $\delta =3$, is that a vertex of degree $3$ is stacked into a face and raises three other degrees
as it is put in, which couples the densities tightly. The patterns that result have no visible
regularity of the kind exploited above, and we found the domains by search rather than by design. A
row of Table
\ref{Tab_computer} is read as follows: $N$ is the number of vertices of the fundamental domain, so
that a domain carries $N$ vertices, $3N$ edges and $2N$ faces; the listed degrees are those of the
equality sequence of the pair, and their multiplicities are determined by $N$ and the extremal
densities, which for $c=c_{\delta ,k}$ are $c$ for each degree but the last and $1-tc$ for the last.

\begin{example}
\label{Ex_T31}The smallest of the seven domains is the one for $(\delta ,k)=(3,1)$, with $N=4$
vertices and one vertex of each of the degrees $3,5,7,9$. In the notation of this section it is
\[
\begin{array}{l}
0:\ 1\ 1_{(1,1)}\ 3;\\
1:\ 1_{(-1,-1)}\ 2\ 2_{(1,1)}\ 1_{(1,1)}\ 0\ 3\ 2_{(0,-1)}\ 3_{(-1,-1)}\ 0_{(-1,-1)};\\
2:\ 1_{(-1,-1)}\ 2_{(-1,-1)}\ 3_{(-1,0)}\ 1_{(0,1)}\ 3_{(0,1)}\ 2_{(1,1)}\ 1;\\
3:\ 1\ 0\ 1_{(1,1)}\ 2_{(1,0)}\ 2_{(0,-1)}.
\end{array}
\]
The degrees are the lengths of the four lists, namely $3,9,7,5$, so the multiset of the domain is
$3^{1}5^{1}7^{1}9^{1}$ and $\sum_{v}(6-\deg (v))=3+(-3)+(-1)+1=0$. Each density is $1/4$, and since
the four degrees are spaced by $2=k+1$ a window of two consecutive degrees contains exactly one of
them, so $w_{1}=1/4=c_{3,1}$. Tracing the faces gives $2N=8$ triangles, against $3N=12$ edges, so
$V-E+F=4-12+8=0$. Two features of the notation are visible here. The vertex $1$ occurs in its own
list, with the labels $(-1,-1)$ and $(1,1)$: the quotient has a loop at $1$, and only the labels
record that its two ends are different vertices of the plane triangulation. And the vertex $0$, of
degree $3$, has the neighbours $1$, the copy of $1$ translated by $e_{1}+e_{2}$, and $3$; in the
plane it is a stacked vertex, sitting inside the triangle formed by those three. The domain for
$(3,6)$ is read in the same way: $N=7$ vertices, four of degree $3$ and three of degree $10$, with
$4\cdot 3+3\cdot (-4)=0$, densities $4/7$ and $3/7$, and, since $10-3=7>6=k$, the window density
$w_{6}=4/7=c_{3,6}$.
\end{example}

Figure \ref{Fig_T31} draws a disk of $T_{3,1}$ with one fundamental domain shaded and the
two periods marked, which is what the five printed conditions amount to.

\begin{figure}[htbp]
\begin{center}
\includegraphics[width=0.55\textwidth]{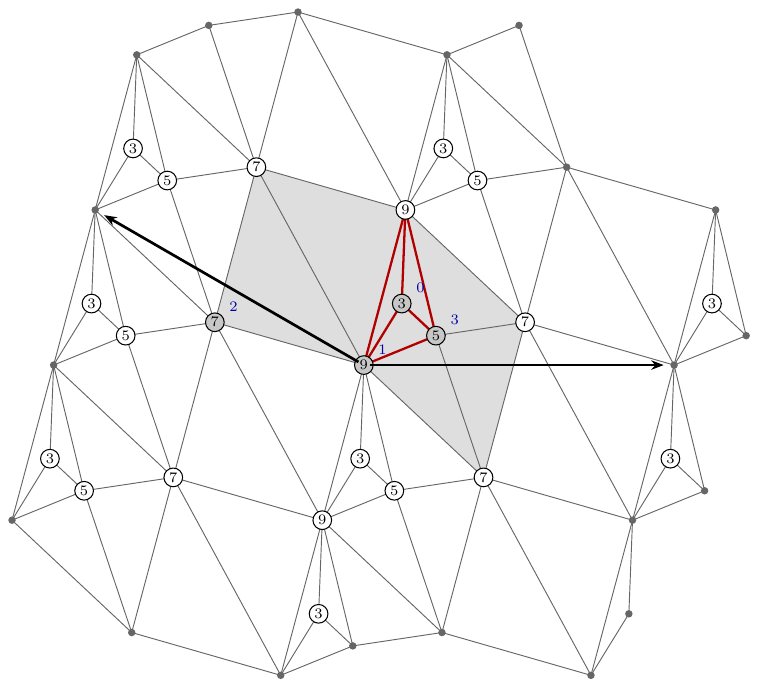}
\end{center}
\caption{A disk of the periodic triangulation $T_{3,1}$ of Example \ref{Ex_T31}, with one fundamental domain
shaded and the two periods drawn as arrows from the vertex $1$. The four vertices $0,1,2,3$ of the domain have
degrees $3,9,7,5$; the vertex $0$ of degree three is stacked into the triangle drawn in red, and the two edges
from $1$ to copies of $1$ --- the loop of the quotient --- leave it in opposite directions.}%
\label{Fig_T31}%
\end{figure}

\begin{table}[htbp]
\begin{center}
\renewcommand{\arraystretch}{1.35}%
\begin{tabular}{|c|c|c|l|}
\hline
$(\delta ,k)$ & $c_{\delta ,k}$ & $N$ & degrees of the equality sequence\\[2pt] \hline
$(3,1)$ & $\tfrac{1}{4}$ & $4$ & $3,\ 5,\ 7,\ 9$\\
$(3,3)$ & $\tfrac{5}{12}$ & $12$ & $3,\ 7,\ 11$\\
$(3,4)$ & $\tfrac{7}{15}$ & $15$ & $3,\ 8,\ 13$\\
$(3,6)$ & $\tfrac{4}{7}$ & $7$ & $3,\ 10$\\
$(3,7)$ & $\tfrac{5}{8}$ & $8$ & $3,\ 11$\\
$(3,8)$ & $\tfrac{2}{3}$ & $9$ & $3,\ 12$\\
$(4,2)$ & $\tfrac{4}{9}$ & $9$ & $4,\ 7,\ 10$\\
\hline
\end{tabular}
\end{center}
\caption{The seven pairs for which a periodic triangulation with the extremal densities was found by computer:
the coefficient $c_{\delta ,k}$ of the maximised bound, the number $N$ of vertices of a fundamental domain, and
the degrees occurring. By Lemma \ref{Lemma_periodic} each of them gives
$\mathrm{MP}(n,\delta ,k)=c_{\delta ,k}\,n+O(\sqrt{n})$.}%
\label{Tab_computer}%
\end{table}

The row $\delta =3$ of our table stops at $k=8$, and not because the search failed: for $k\geq 9$ the extremal
densities cannot be realised at all, as the next section shows.

\subsection{Minimum degree three: the coefficient for $k\geq 9$}
\label{Sec_deg3exact}

For a maximal planar graph $G$ write
\[
X(G)=\sum \{\deg (v)-6:\deg (v)\geq 7\}=\sum_{v\in V(G)}\bigl(\deg (v)-6\bigr)^{+}
\]
for its total excess above the average degree six. Euler's formula says $\sum_{v}(6-\deg (v))=12$, so
\begin{equation}
X(G)=3n_{3}+2n_{4}+n_{5}-12,\label{eq_excess}
\end{equation}
and $X$ measures at the same time how much degree is piled up on the large vertices and how many small vertices
there are. It cannot be large.

Before the proof, a word on why the total excess is the quantity to look at. Suppose
$\delta (G)=3$ and let $k$ be large. The window at $p=3$ is $\{3,4,\ldots ,k+3\}$, and since no
degree is smaller than three it misses only the vertices of degree at least $k+4$, so that
\[
\mathrm{sp}(G,k)\ \geq \ n-\left\vert \{v:\deg (v)\geq k+4\}\right\vert ;
\]
for large $k$ the whole question is how many vertices can escape this single window. Escaping is
expensive: a vertex of degree at least $k+4$ carries an excess of at least $(k+4)-6=k-2$, and it
must pay it out of $X(G)$, which all the vertices of the maximal planar graph share. That $X(G)$ is a
bounded resource at all is a fact about planarity and not about the degree sum: the equality
sequence $3^{s}(k+4)^{n-s}$ of Theorem \ref{Tm_planar34}(i), with $s=((k-2)n+12)/(k+1)$, has the
degree sum $6n-12$ of a maximal planar graph, and its excess $3s-12$ equals $(21n-84)/10$ for $k=9$ and so
exceeds $2n-14$ at every order.
What does bound the excess is that the faces at a vertex are pairwise distinct and that there are
only $2n-4$ faces altogether: a vertex of degree $d$ occupies $d$ of them, two vertices of large
degree can occupy one face jointly only across an edge joining them, and three only across a face
all of whose corners are large. The proof below is this accounting and nothing else, and it also
explains why three large vertices are worse than two: for $|S|\geq 3$ planarity limits the edges
inside $S$ to $3|S|-6$, so that the term $2e-6|S|$ in the count below is at most $-12$, whereas
two adjacent large vertices give $2e-6|S|=-10$.

Read through \eqref{eq_excess}, the same inequality says that a maximal planar graph cannot have too many
small vertices, $3n_{3}+2n_{4}+n_{5}\leq 2n-2$. For $n_{3}$ alone this gives $n_{3}\leq (2n-2)/3$, which
is weaker than the bound $n_{3}\leq (2n-4)/3$ of Lemma \ref{Lemma_deg3count}; the sharper bound does not
improve the inequality above, however, because that inequality constrains the combination
$3n_{3}+2n_{4}+n_{5}$ and not $n_{3}$ by itself --- the graph that makes it sharp has only two vertices of
degree three. What the inequality is, and what Lemma \ref{Lemma_deg3count} is not, is the
constraint which the window count of Section \ref{Sec_planar34} cannot see, since it uses the degree
sum only.


\begin{lemma}
\label{Lemma_excess}Every maximal planar graph $G$ with
$n\geq 7$ vertices satisfies
\[
X(G)\ \leq \ 2n-14,\qquad \text{equivalently}\qquad 3n_{3}+2n_{4}+n_{5}\ \leq \ 2n-2 ,
\]
and for every $n\geq 7$ some maximal planar graph of order $n$ attains it. If at least three vertices of $G$ have degree
at least seven, then $X(G)\leq 2n-16$.
\end{lemma}

\begin{proof}
Let $S=\{v:\deg (v)\geq 7\}$, let $e$ be the number of edges of $G$ with both ends in $S$, and for $i=1,2,3$ let
$\varphi _{i}$ be the number of faces of $G$ with exactly $i$ corners in $S$; write
$F(S)=\varphi _{1}+\varphi _{2}+\varphi _{3}$ for the number of faces meeting $S$. A maximal planar graph of order
at least four is $3$-connected, so the faces at a vertex $v$ are pairwise distinct and $\deg (v)$ in number;
counting the pairs (vertex of $S$, face containing it) in two ways therefore gives
\[
\sum_{v\in S}\deg (v)=\varphi _{1}+2\varphi _{2}+3\varphi _{3}.
\]
A face is a triangle, so its corners are pairwise adjacent; a face with exactly two corners in $S$ therefore contains
exactly one edge of $G[S]$, a face with three corners contains three, and a face with at most one corner contains
none. Every edge lies in exactly two faces, whence
\[
2e=\varphi _{2}+3\varphi _{3}.
\]
Subtracting the second identity from the first gives
\begin{equation}
F(S)=\sum_{v\in S}\deg (v)-2e+\varphi _{3},\label{eq_faces}
\end{equation}
and therefore, since $F(S)$ is at most the number $2n-4$ of faces of $G$ and $\varphi _{3}\geq 0$,
(\ref{eq_faces}) gives
\begin{equation}
X(G)=\sum_{v\in S}\bigl(\deg (v)-6\bigr)=F(S)+2e-\varphi _{3}-6|S|\ \leq \ (2n-4)+2e-6|S| .\label{eq_Xbound}
\end{equation}
If $|S|\geq 3$ then $G[S]$ is a simple planar graph on $|S|$ vertices, so $e\leq 3|S|-6$ and $X(G)\leq 2n-16$. If
$|S|=2$ then $e\leq 1$ and $X(G)\leq 2n-14$. If $|S|=1$, say $S=\{v\}$, then $e=0$ and $F(S)=\deg (v)\leq n-1$, so
$X(G)=\deg (v)-6\leq n-7\leq 2n-14$ because $n\geq 7$. If $|S|=0$ then $X(G)=0\leq 2n-14$ because $n\geq 7$. In every
case $X(G)\leq 2n-14$, and \eqref{eq_excess} turns this into the second form.

For equality take two adjacent vertices $u,v$ and join both of them to every vertex of a path
$w_{1}w_{2}\cdots w_{n-2}$. This is a maximal planar graph: its faces are $uw_{i}w_{i+1}$ and
$vw_{i}w_{i+1}$ for $1\leq i\leq n-3$ together with $uvw_{1}$ and $uvw_{n-2}$, which are $2n-4$ in number, and it
has $(n-2)+(n-2)+(n-3)+1=3n-6$ edges. Its degrees are $\deg (u)=\deg (v)=n-1$, $\deg (w_{1})=\deg (w_{n-2})=3$
and $\deg (w_{i})=4$ otherwise, so $X=2(n-7)=2n-14$.
\end{proof}

The second bound of the lemma is attained as well. Stacking a vertex into every face of a maximal planar graph $H$ of
order $n_{H}$ with $\delta (H)\geq 4$ produces a maximal planar graph of order $n=3n_{H}-4$ in which the vertices of $H$
carry all the excess, $X=6n_{H}-24=2n-16$, while the $2n_{H}-4$ new vertices have degree three; the smallest
maximal planar graph obtained in this way is the stacked octahedron, of order $14$ with degrees $3^{8}8^{6}$. Smaller
examples of $X=2n-16$ with three vertices of degree at least seven exist, the smallest being of order $10$.

\begin{example}
\label{Ex_excess_sharp}
Let $n=9$ and let $G$ be the maximal planar graph used for equality: two adjacent vertices $u,v$ joined to
every vertex of the path $w_{1}w_{2}\cdots w_{7}$. Its degrees are $\deg (u)=\deg (v)=8$,
$\deg (w_{1})=\deg (w_{7})=3$ and $\deg (w_{i})=4$ for $2\leq i\leq 6$, so that
\[
X(G)=(8-6)+(8-6)=4=2n-14,\qquad 3n_{3}+2n_{4}+n_{5}=3\cdot 2+2\cdot 5+0=16=2n-2 .
\]
In the notation of the proof $S=\{u,v\}$, $e=1$, $\varphi _{1}=12$, $\varphi _{2}=2$,
$\varphi _{3}=0$ and $F(S)=14=2n-4$, so that \eqref{eq_Xbound} reads $X(G)=14+2-0-12=4$ with
equality at every step.
Three things make it extremal, and all three are visible in the picture: the set $S$ has two
elements, so that no face can have three large corners; the two large vertices are adjacent, so that
$e$ is as large as $|S|=2$ permits; and every one of the $2n-4$ faces meets $S$, because $u$ and $v$
lie on opposite sides of the path and the edge $uv$ closes the two remaining faces $uvw_{1}$ and
$uvw_{7}$ at its ends (Figure \ref{Fig_excess9}).
\end{example}

\begin{figure}[htbp]
\begin{center}
\includegraphics[width=0.9\textwidth]{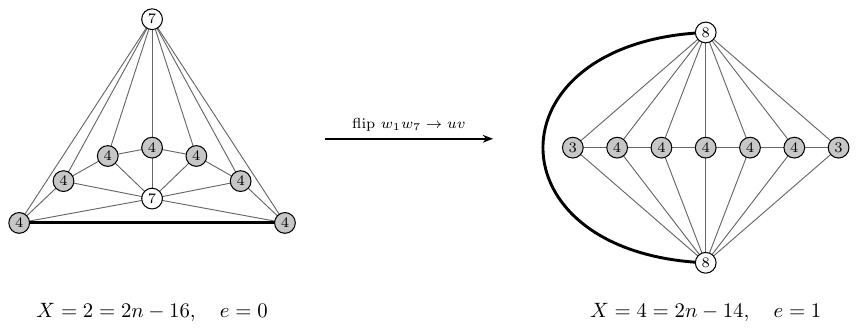}
\end{center}
\caption{Order nine. Left: the bipyramid, with $X=2=2n-16$ and no edge between two vertices of degree at least
seven. Right: the extremal maximal planar graph of Lemma \ref{Lemma_excess}, with $X=4=2n-14$ and one such edge. The
two differ by the flip of the bold edge, which is where the two units of excess come from.}%
\label{Fig_excess9}%
\end{figure}

The example also shows how little room there is above the bound. Joining $u$ and $v$ to a cycle
$w_{1}w_{2}\cdots w_{n-2}$ instead of a path gives the bipyramid, in which $\deg (u)=\deg (v)=n-2$
and every other degree is four; there $e=0$ while every face still meets $S$, so that $X=2n-16$, two
units less.
A single diagonal flip turns one into the other: flipping a rim edge $w_{i}w_{i+1}$ replaces it by
the edge $uv$, raises $\deg (u)$ and $\deg (v)$ by one each, and lowers $\deg (w_{i})$ and
$\deg (w_{i+1})$ from four to three, which is $+2$ for $X$.

\begin{example}
\label{Ex_excess_second}
The second bound is attained by the stacked octahedron, of order $n=14$: its six octahedron vertices
have degree eight and its eight stacked vertices degree three, so that $X=6\cdot 2=12=2n-16$. Here
$|S|=6$ and $S$ induces the octahedron, so $e=12=3|S|-6$, while $\varphi _{2}=24=2n-4$ and
$\varphi _{1}=\varphi _{3}=0$: the two inequalities used in the proof for $|S|\geq 3$, namely
$e\leq 3|S|-6$ and $F(S)\leq 2n-4$, are both equalities.
The smallest maximal planar graphs with $X=2n-16$ and three vertices of degree at least seven have order
ten, and there are seven of them up to isomorphism. One is obtained from $K_{4}$ by inserting into
its four faces chains of lengths $2,2,1,1$ in the sense of Proposition \ref{Prop_deg3}, the ends and
tips chosen so that every vertex of the $K_{4}$ is an end of exactly one of the two chains of length
two and gains a single neighbour in each of its two remaining faces; all four then have degree
$3+4=7$, and the six inserted vertices have degrees $3,3,3,3,4,4$ (Figure \ref{Fig_excess10}).
\end{example}

\begin{figure}[htbp]
\begin{center}
\includegraphics[width=0.52\textwidth]{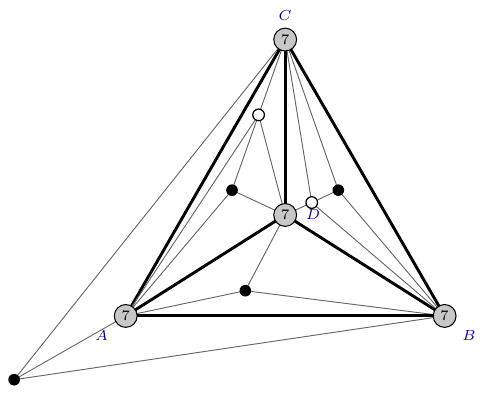}
\end{center}
\caption{The order-ten example of Example \ref{Ex_excess_second}: the graph $K_{4}$ (bold) with a chain
inserted into each of its four faces, so that $A$, $B$, $C$, $D$ all have degree seven and $X=4=2n-16$. The
six inserted vertices have degree three (black) or four (white).}%
\label{Fig_excess10}%
\end{figure}

The same enumeration confirms the first bound at the small orders: the maximum of $X(G)$ over all
maximal planar graphs of order $n$ is $2n-14$ for every $7\leq n\leq 11$.

Applied to a maximal planar graph of minimum degree three, Lemma \ref{Lemma_excess} bounds the number of vertices that can
escape the first window, and that is all one needs.

\begin{theorem}
\label{Tm_deg3exact}Let $k\geq 8$ and let $G$ be a maximal planar graph with $n\geq 7$ vertices and
$\delta (G)=3$. Then
\[
\mathrm{sp}(G,k)\ \geq \ \frac{(k-4)\,n+14}{k-2},
\]
and if $n\geq k+6$ then even
\[
\mathrm{sp}(G,k)\ \geq \ \frac{(k-4)\,n+16}{k-2},\qquad \text{so that}\qquad
\mathrm{MP}(n,3,k)\ \geq \ \left\lceil \frac{(k-4)\,n+16}{k-2}\right\rceil .
\]
\end{theorem}

\begin{proof}
Put $B=\{v:\deg (v)\geq k+4\}$ and $b=|B|$. As $\delta (G)=3$, every vertex outside $B$ has its degree in
$\{3,\ldots ,k+3\}$, which is a set of $k+1$ consecutive integers; the window at $p=3$ therefore already contains
$n-b$ vertices and $\mathrm{sp}(G,k)\geq n-b$. Every vertex of $B$ contributes at least $(k+4)-6=k-2$ to $X(G)$, so
$b(k-2)\leq X(G)\leq 2n-14$ by Lemma \ref{Lemma_excess}, and
$\mathrm{sp}(G,k)\geq n-(2n-14)/(k-2)=((k-4)n+14)/(k-2)$.

For the second bound distinguish two cases. As $k\geq 8$, the vertices of $B$ have degree at least $k+4\geq 12$, so
$B\subseteq S=\{v:\deg (v)\geq 7\}$. If $b\geq 3$ then $|S|\geq 3$, the second bound of Lemma \ref{Lemma_excess}
gives $X(G)\leq 2n-16$, and $\mathrm{sp}(G,k)\geq n-(2n-16)/(k-2)=((k-4)n+16)/(k-2)$. If $b\leq 2$ then
$\mathrm{sp}(G,k)\geq n-2$, and $n-2\geq ((k-4)n+16)/(k-2)$ because that inequality is equivalent to
$2n\geq 2k+12$, which is the hypothesis $n\geq k+6$.
\end{proof}

\begin{example}
\label{Ex_deg3count}
The counting of the proof at $k=9$ and $n=35$. A vertex of $B$ has degree at least $k+4=13$ and
therefore costs at least $k-2=7$ units of excess. If $b\geq 3$ then three vertices have degree at
least seven, Lemma \ref{Lemma_excess} grants only $X(G)\leq 2n-16=54$, and
$b\leq \left\lfloor 54/7\right\rfloor =7$, whence $\mathrm{sp}(G,9)\geq 35-7=28$; and if $b\leq 2$
then $\mathrm{sp}(G,9)\geq 33$ outright. So the bound is
$28=\left\lceil (5\cdot 35+16)/7\right\rceil =\left\lceil 191/7\right\rceil $, the value of the
theorem. The maximised bound of Theorem
\ref{Tm_planar34}(i) gives $\left\lceil (7\cdot 35+12)/10\right\rceil =26$ at the same order. At
$n=100$ the two bounds read $74$ and $72$, at $n=1000$ they read $717$ and $702$, and the gap grows
linearly because the coefficients are $5/7$ and $7/10$.
\end{example}


Compare the coefficients. For $k=8$ that of Theorem \ref{Tm_deg3exact} equals $c_{3,8}=(k-2)/(k+1)=\frac23$, the
coefficient of Theorem \ref{Tm_planar34}; for $k\leq 7$ the latter is the larger; and for $k\geq 9$ the former is
larger, since $\frac{k-4}{k-2}-\frac{k-2}{k+1}$ has the sign of $(k-4)(k+1)-(k-2)^{2}=k-8$. In particular the
maximised bound of Theorem \ref{Tm_planar34} is \emph{not} asymptotically tight for $\delta =3$ and $k\geq 9$,
and the range $k\leq 8$ is exactly the range in which it is. This is a comparison of coefficients only: for
small $n$ the additive term $16/(k-2)$ can make Theorem \ref{Tm_deg3exact} the stronger bound even when its
coefficient is not. And for $k\geq 9$ its coefficient is the optimal one.

\begin{remark}
\label{Rem_crossover}
The crossover is easiest to read as a budget of escapes from the window at $p=3$. That window
catches every vertex of degree at most $k+3$, so both bounds have the form $n$ minus a bound on the
number of vertices which avoid it, and they differ only in how that number is bounded. The degree
sum $6n-12$ by itself allows $(3n-12)/(k+1)$ escapes, which is what Theorem \ref{Tm_planar34}(i)
says in the form $n-((k-2)n+12)/(k+1)$; the excess inequality allows $(2n-16)/(k-2)$, which is what
Theorem \ref{Tm_deg3exact} says. The two escape densities are $3/(k+1)$ and $2/(k-2)$, and
\[
\frac{3}{k+1}-\frac{2}{k-2}=\frac{k-8}{(k+1)(k-2)},
\]
so the arithmetical constraint is the weaker of the two exactly for $k\geq 9$. Subtracting from one
returns the coefficients $\frac{k-2}{k+1}$ and $\frac{k-4}{k-2}$ and the difference computed above.

Behind the arithmetic stands a density. Equality in Theorem \ref{Tm_planar34}(i) requires the degree
sequence $3^{s}(k+4)^{n-s}$, in which a fraction $\frac{k-2}{k+1}$ of all vertices has degree three;
but $3n_{3}\leq 3n_{3}+2n_{4}+n_{5}\leq 2n-2$ in every maximal planar graph, so at most two thirds of the
vertices can have degree three. Since $\frac{k-2}{k+1}>\frac{2}{3}$ precisely for $k\geq 9$, the
extremal density is unattainable there even in the limit, while for $k\leq 8$ it is at most $2/3$
and Section \ref{Sec_computer} realises it. 
\end{remark}


\begin{proposition}
\label{Prop_deg3}For every $k\geq 8$ there is a periodic triangulation $S_{k}$ of the plane all of whose degrees
lie in $\{3,4,k+4\}$, the three occurring with densities
\[
\frac{4}{k-2},\qquad \frac{k-8}{k-2},\qquad \frac{2}{k-2}
\]
respectively. Its minimum degree is three and $w_{k}(S_{k})=\frac{k-4}{k-2}$, so that by Lemma
\ref{Lemma_periodic} and Theorem \ref{Tm_deg3exact}
\[
\mathrm{MP}(n,3,k)=\frac{k-4}{k-2}\,n+O(\sqrt{n})
\]
for every sufficiently large $n$; in particular $\lim_{n\to \infty }\mathrm{MP}(n,3,k)/n=\frac{k-4}{k-2}$.
\end{proposition}

\begin{proof}
Draw the triangular lattice on the vertex set $\mathbb{Z}^{2}$ as the square grid together with the diagonal
joining $(1,0)$ to $(0,1)$ in every cell, so that its faces are the upward triangles
$U(p)=\{p,p+e_{1},p+e_{2}\}$ and the downward triangles $D(p)=\{p+e_{1},p+e_{2},p+e_{1}+e_{2}\}$, every vertex
has degree six, and the six faces at a vertex $v$ are
\begin{equation}
U(v),\quad D(v-e_{1}),\quad U(v-e_{1}),\quad D(v-e_{1}-e_{2}),\quad U(v-e_{2}),\quad D(v-e_{2}).\label{eq_sixfaces}
\end{equation}

Into a triangle $abc$ we insert a \emph{chain of length} $m\geq 1$: put a vertex $x_{1}$ into $abc$, joined to
$a,b,c$; then a vertex $x_{2}$ into the face $abx_{1}$, joined to $a,b,x_{1}$; then $x_{3}$ into $abx_{2}$; and
so on, $m$ times. The neighbours of $x_{1}$ are then $a,b,c,x_{2}$, those of $x_{i}$ for $1<i<m$ are
$a,b,x_{i-1},x_{i+1}$, and those of $x_{m}$ are $a,b,x_{m-1}$, so a chain of length $m$ contributes $m-1$
vertices of degree four and one vertex of degree three, while the two \emph{ends} $a$ and $b$ each gain $m$
neighbours and the \emph{tip} $c$ gains one. (For $m=1$ this is the insertion of a single vertex of degree three
into $abc$.)

We draw the chain inside its triangle, the vertex $x_{i}$ at the point dividing the segment from $c$ to the
midpoint of $ab$ in the ratio $i:(m+1-i)$, and all the edges as straight segments; the drawing is then plane, and
in the periodic pattern below all edges have length at most the diameter of a lattice triangle, as Lemma
\ref{Lemma_periodic} requires.

Each face of the lattice is a triangle with one distinguished corner, which we call its \emph{tip}: for
an upward face $U(p)=\{p,p+e_{1},p+e_{2}\}$ the tip is $p$, and for a downward face
$D(p)=\{p+e_{1},p+e_{2},p+e_{1}+e_{2}\}$ it is $p+e_{1}+e_{2}$; the other two corners are its \emph{ends}.
The tip is the corner the chain inserted into that face will be joined to by all of its vertices, while only
the first and the last vertex of the chain reach the two ends, which is why a lattice vertex gains a great
deal of degree from the faces of which it is the tip and little from those of which it is an end. Insert now
a chain into every face of the lattice, choosing
\[
U(p):\ \text{tip }p,\ \text{ends }p+e_{1},p+e_{2};\qquad
D(p):\ \text{tip }p+e_{1}+e_{2},\ \text{ends }p+e_{1},p+e_{2},
\]
and letting the length depend only on the class of the face: it is $m_{U,i}$ for $U(p)$ and $m_{D,i}$ for
$D(p)$, where $i\in \{0,1\}$ is the parity of the first coordinate of $p$. By \eqref{eq_sixfaces} a vertex
$v=(x,y)$ is the tip of exactly two of its six faces, namely $U(v)$ and $D(v-e_{1}-e_{2})$, and an end of the
remaining four, which are $U(v-e_{1}),U(v-e_{2}),D(v-e_{1}),D(v-e_{2})$ and lie in the classes
$(U,x-1),(U,x),(D,x-1),(D,x)$, that is one class of each parity for each of the two shapes. Every lattice vertex
therefore has degree
\[
6+2+m_{U,0}+m_{U,1}+m_{D,0}+m_{D,1}=8+M,\qquad M=m_{U,0}+m_{U,1}+m_{D,0}+m_{D,1},
\]
whatever $x$ and $y$ are. We take $M=k-4$, with the four lengths as equal as possible; since $k\geq 8$ they are
all at least one, so every face really receives a chain. For $k=8$ all four lengths are one and $S_{8}$ is the
stacked triangular lattice, with degrees $3$ and $12$.

The construction is invariant under the translations by $(2,0)$ and $(0,1)$, and a fundamental domain of that
group contains two lattice vertices and one face of each of the four classes. Hence per fundamental domain
\[
N=2+M=k-2,\qquad N_{3}=4,\qquad N_{4}=M-4=k-8,\qquad N_{k+4}=2,
\]
which are the asserted densities; the minimum degree is three. (As a check,
$\sum_{v}(6-\deg (v))=4\cdot 3+(k-8)\cdot 2+2\cdot (2-k)=0$, as it must be over a fundamental domain.)

Finally the window density. The degrees $3$ and $4$ both lie in $\{3,\ldots ,k+3\}$, a window of $k+1$
consecutive integers, which therefore catches all but the $N_{k+4}=2$ vertices of degree $k+4$, a fraction
$1-\frac{2}{k-2}=\frac{k-4}{k-2}$. Any other window either misses the degree $k+4$, and then catches at most as
many vertices, or contains $k+4$ and hence begins at some $p\geq 4$, so that it misses all four vertices of
degree three and catches at most a fraction $1-\frac{4}{k-2}$. Thus $w_{k}(S_{k})=\frac{k-4}{k-2}$, and Lemma
\ref{Lemma_periodic} together with Theorem \ref{Tm_deg3exact} completes the proof. Figure \ref{Fig_chain} shows
$S_{9}$ and $S_{12}$.
\end{proof}

\begin{example}
\label{Ex_S9}
The pattern $S_{9}$ in detail. Here $M=k-4=5$, and the four lengths, as equal as possible, are
$m_{U,0}=2$ and $m_{U,1}=m_{D,0}=m_{D,1}=1$: the upward triangles of the even columns carry a chain
of two vertices, and the triangles of the three other classes carry a single stacked vertex. Every
lattice vertex has degree $6+2+M=13=k+4$. A fundamental domain of the translations $(2,0)$ and
$(0,1)$ has $N=k-2=7$ vertices: the two lattice vertices, the four last vertices of the four chains,
which have degree three, and the first vertex of the chain of length two, which has degree four.
The densities are $4/7$ for the degree three, $1/7$ for the degree four and $2/7$ for the degree
$k+4=13$, in agreement with Proposition \ref{Prop_deg3}. The window at $p=3$ is $\{3,4,\ldots ,12\}$; it catches the
five vertices of degree three or four of the domain, a density $5/7$. A window containing the degree
$13$ begins at some $p\geq 4$, misses the four vertices of degree three and catches a density $3/7$
only. Hence $w_{9}(S_{9})=\frac{5}{7}=\frac{k-4}{k-2}$.
Taking the plane modulo the translations $(6,0)$ and $(0,5)$, that is a block of fifteen fundamental
domains, one obtains a triangulation of the torus with $105$ vertices, $315$ edges and $210$ faces,
with $30$ vertices of degree $13$, $60$ of degree three and $15$ of degree four, every edge in
exactly two triangles; this checks the pattern in finite terms.
Cutting a disk out of $S_{9}$ and closing it as in Lemma \ref{Lemma_periodic} therefore gives, for
every large $n$, a maximal planar graph of order $n$ with $\delta =3$ and
$\mathrm{sp}(G,9)=\frac{5}{7}n+O(\sqrt{n})$, against the lower bound
$\left\lceil (5n+16)/7\right\rceil $ of Theorem \ref{Tm_deg3exact}.
\end{example}

Figure \ref{Fig_chain9} draws $S_{9}$ with one fundamental domain shaded.

\begin{figure}[htbp]
\begin{center}
\includegraphics[width=0.7\textwidth]{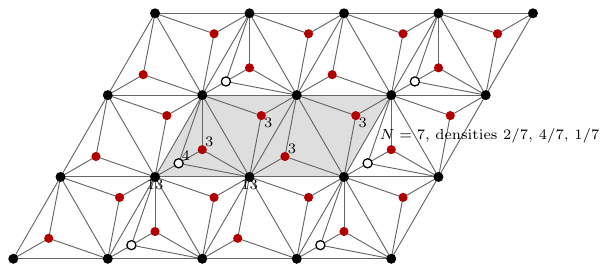}
\end{center}
\caption{The pattern $S_{9}$ of Example \ref{Ex_S9}, with one fundamental domain shaded: the parallelogram
spanned by $(2,0)$ and $(0,1)$ contains two lattice vertices of degree $13$ and one triangle of each of the
four classes, hence one chain of length two and three chains of length one. So $N=7$, and the degrees $3$,
$4$, $13$ occur with densities $4/7$, $1/7$, $2/7$.}%
\label{Fig_chain9}%
\end{figure}

For $k=10$ the four lengths are $2,2,1,1$, the domain has eight vertices and the densities are
$\frac{1}{2}$ for the degree three, $\frac{1}{4}$ for the degree four and $\frac{1}{4}$ for the
degree $14$, so that $w_{10}(S_{10})=\frac{3}{4}$; for $k=12$ they are $2,2,2,2$, the domain has ten
vertices, and the densities are $\frac{2}{5}$, $\frac{2}{5}$ and $\frac{1}{5}$ with
$w_{12}(S_{12})=\frac{4}{5}$. The vertices of degree four are the inner vertices of the chains, so
they appear only from $k=9$ on: in $S_{8}$ all four chains have length one, and $S_{8}$ is the
triangular lattice with a vertex stacked into every face, with the degrees three and twelve in the
proportion $2{:}1$.



Two ranges of orders are worth recording before we leave $\delta =3$, for a reason that is not about
exact values for their own sake. They are the only orders at which $\mathrm{MP}(n,\delta ,k)$ is known
\emph{exactly for all $n$} anywhere in this paper, and they end precisely where Theorem
\ref{Tm_deg3exact} begins to bite: the threshold $\lceil (3k+10)/2\rceil $ below is the order at which a
third vertex of degree at least $k+4$ first becomes affordable, which is what the theorem needs in order to
say anything. The proposition therefore explains the shape of the theorem rather than adding to it.

\begin{proposition}
\label{Prop_deg3small}Let $k\geq 8$. Then
\[
\mathrm{MP}(n,3,k)=n\ \ \text{for }4\leq n\leq k+4,\qquad
\mathrm{MP}(n,3,k)=n-2\ \ \text{for }k+5\leq n\leq \left\lceil \tfrac{3k+10}{2}\right\rceil -1 .
\]
In particular $\mathrm{MP}(n,3,k)$ never takes the value $n-1$.
\end{proposition}

\begin{proof}
If $n\leq k+4$, every degree of a maximal planar graph of order $n$ lies between $3$ and $n-1\leq k+3$, so the window
at $p=3$ contains all $n$ vertices.

Let now $k+5\leq n\leq \lceil (3k+10)/2\rceil -1$ and let $G$ be a maximal planar graph with $\delta (G)=3$; put
$B=\{v:\deg (v)\geq k+4\}$ and $b=|B|$ as in the proof of Theorem \ref{Tm_deg3exact}. If $b\geq 3$ then three
vertices have degree at least $k+4\geq 12$, so the second bound of Lemma \ref{Lemma_excess} gives
$3(k-2)\leq b(k-2)\leq X(G)\leq 2n-16$, that is $n\geq (3k+10)/2$, contradicting the hypothesis. Hence $b\leq 2$
and $\mathrm{sp}(G,k)\geq n-b\geq n-2$.

Equality is attained by the graph that makes Lemma \ref{Lemma_excess} sharp: two adjacent vertices $u,v$ joined
to every vertex of a path $w_{1}\cdots w_{n-2}$. Its degrees are $\deg (u)=\deg (v)=n-1$ and
$4^{\,n-4}3^{2}$, and $n-1\geq k+4$, so the window at $p=3$ contains all $n-2$ vertices other than $u$ and $v$
while no window contains more: a window containing $n-1$ begins at $p\geq 4$ and so misses $w_{1}$ and
$w_{n-2}$. Hence $\mathrm{sp}=n-2$.
\end{proof}


\begin{figure}[htbp]
\begin{center}
\includegraphics[width=0.46\textwidth]{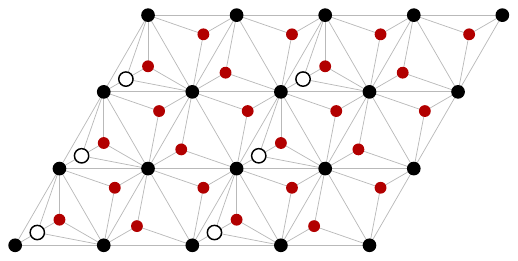}\hfill
\includegraphics[width=0.46\textwidth]{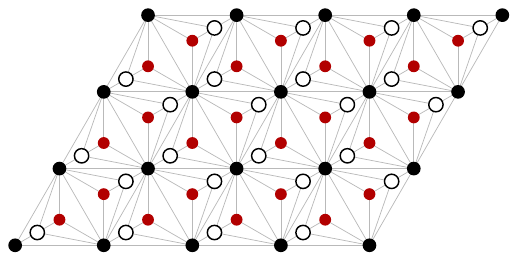}
\end{center}
\caption{The triangulations $S_{k}$ of Proposition \ref{Prop_deg3} for $k=9$, with chain lengths $2,1,1,1$, and
for $k=12$, with chain lengths $2,2,2,2$. The vertices of the lattice, of degree $k+4$, are the large black
dots; inside each chain the vertices of degree four are white and the vertex of degree three at its end is red.
Every triangle carries a chain, whose two ends are the endpoints of the diagonal of the cell and whose tip is
the third corner.}%
\label{Fig_chain}%
\end{figure}

\section{Summary and concluding remarks}

\begin{theorem}
\label{Tm_asym}Let $\delta \in \{3,4,5\}$ and $k\geq 0$, with $k\leq 8$ if $\delta =3$. Then the maximised
lower bound for $\mathrm{MP}(n,\delta ,k)$ is asymptotically tight: with
$c_{\delta ,k}$ its coefficient,
\[
\mathrm{MP}(n,\delta ,k)=c_{\delta ,k}\,n+O(\sqrt{n})
\]
for every sufficiently large $n$, so that $\lim_{n\to \infty }\mathrm{MP}(n,\delta ,k)/n=c_{\delta ,k}$,
and $c_{\delta ,k}$ is the entry of Table \ref{Tab_coeff}. For $\delta =3$ and $k\geq 9$ the maximised bound is
not asymptotically tight and the limit is $\frac{k-4}{k-2}$, by Theorem \ref{Tm_deg3exact} and Proposition
\ref{Prop_deg3}; so for every pair $(\delta ,k)$ the limit exists and is the entry of Table
\ref{Tab_coeff}.
\end{theorem}

\begin{proof}
For $k=0$ the periodic triangulations of \cite[Theorem 2.5]{CaroWest} (for $\delta =5$ also the
flipped triangular lattice of Section \ref{Sec_hex}) have the extremal densities and minimum degree $\delta $, so
Lemma \ref{Lemma_periodic} applies; for $\delta =4$ and
$k\geq 3$ it is Proposition \ref{Prop_chain4}; for $\delta =5$ and $k\geq 2$ it is Proposition
\ref{Prop_wheel}; for the six pairs of
Proposition \ref{Prop_hand} it follows from that proposition and Lemma \ref{Lemma_periodic}; for the remaining pairs
the periodic triangulations of Section \ref{Sec_computer} have the extremal densities and minimum degree $\delta $,
and Lemma \ref{Lemma_periodic} applies.
\end{proof}

Table \ref{Tab_coeff} lists the coefficients themselves.

\begin{table}[htbp]
\begin{center}
\begin{tabular}{|c|c|c|c|}
\hline
$k$ & $\delta =3$ & $\delta =4$ & $\delta =5$\\ \hline
$0$ & $\frac{1}{7}$ & $\frac{1}{5}$ & $\frac{1}{3}$\\[1mm]
$1$ & $\frac{1}{4}$ & $\frac{1}{3}$ & $\frac{1}{2}$\\[1mm]
$2$ & $\frac{1}{3}$ & $\frac{4}{9}$ & $\frac{2}{3}$\\[1mm]
$3$ & $\frac{5}{12}$ & $\frac{1}{2}$ & $\frac{3}{4}$\\[1mm]
$4$ & $\frac{7}{15}$ & $\frac{3}{5}$ & $\frac{4}{5}$\\[1mm]
$5$ & $\frac{1}{2}$ & $\frac{2}{3}$ & $\frac{5}{6}$\\[1mm]
$6$ & $\frac{4}{7}$ & $\frac{5}{7}$ & $\frac{6}{7}$\\[1mm]
$7$ & $\frac{5}{8}$ & $\frac{3}{4}$ & $\frac{7}{8}$\\[1mm]
$8$ & $\frac{2}{3}$ & $\frac{7}{9}$ & $\frac{8}{9}$\\[1mm]
$k\geq 9$ & $\frac{k-4}{k-2}$ & $\frac{k-1}{k+1}$ & $\frac{k}{k+1}$\\ \hline
\end{tabular}
\end{center}
\caption{The coefficient $c=\lim_{n\to \infty }\mathrm{MP}(n,\delta ,k)/n$ for every pair
$(\delta ,k)$. For $\delta =3$ and $k\geq 9$ it is the coefficient of Theorem \ref{Tm_deg3exact},
attained by Proposition \ref{Prop_deg3}; in every other entry it is the coefficient of the maximised
bound, attained by Theorem \ref{Tm_asym}. The entries for $k\geq 3$, $\delta =4$ and for $k\geq 1$,
$\delta =5$ are $\frac{k-1}{k+1}$ and $\frac{k}{k+1}$ throughout, and those for $5\leq k\leq 8$,
$\delta =3$ are $\frac{k-2}{k+1}$.}%
\label{Tab_coeff}%
\end{table}

Each column of the table increases with $k$ and each row increases with $\delta $: a wider window can
only catch more, and a larger minimum degree leaves less room for the small degrees which the best
window collects. For large $k$ the three entries are $1-\frac{2}{k-2}$, $1-\frac{2}{k+1}$ and
$1-\frac{1}{k+1}$, so that the density of the vertices lying outside the best window is $\frac{2}{k-2}$,
$\frac{2}{k+1}$ and $\frac{1}{k+1}$ respectively: with minimum degree five one vertex out of $k+1$
may escape, with minimum degree four two, and with minimum degree three two out of every $k-2$.
All three densities tend to zero, so $\mathrm{MP}(n,\delta ,k)/n\to 1$ as $k\to \infty $ for each
$\delta $, which is what one expects, a window of width $k+1$ eventually covering all degrees that a
maximal planar graph can afford to use.


We conclude with the following three problems, which we consider the most attractive ones. The first
concerns the error term.

\begin{problem}
\label{Pr_error}
For which pairs $(\delta ,k)$ is $\mathrm{MP}(n,\delta ,k)-c_{\delta ,k}n$ bounded, and is it
eventually periodic in $n$? Our constructions give only $O(\sqrt{n})$, the loss coming from the
boundary of the disk and from the cap, whereas every exact value we know exceeds the lower bound by
at most two, and by at most one except at $(\delta ,k,n)=(3,2,6)$ and $(3,3,7)$ (Table
\ref{Tab_MP34small}); this suggests that the exact value is $c_{\delta ,k}n+O(1)$ for every pair. The only
ranges in which the exact value is known for all $n$ are those of Proposition \ref{Prop_deg3small}.
\end{problem}

The second problem asks whether the three minimum degrees are ordered. Table \ref{Tab_coeff} is
increasing in $\delta $ in every row, and so is every column of exact values that we have computed.

\begin{problem}
\label{Pr_monotone}
Is $\mathrm{MP}(n,3,k)\leq \mathrm{MP}(n,4,k)\leq \mathrm{MP}(n,5,k)$ for every $k\geq 0$ and every
$n$ at which all three are defined? The inequality holds for the coefficients, by Table
\ref{Tab_coeff}, and it holds for every $n\leq 12$ and every $k\leq 4$ by the enumeration.
\end{problem}

The last problem leaves planarity behind. The whole method rests on one feature of the class, namely
that its degree sum is a constant of the order, and that feature is not special to planarity. A graph is
\emph{maximal $d$-degenerate} if it is $d$-degenerate and the addition of any edge destroys
$d$-degeneracy; such a graph of order $n\geq d$ has exactly $dn-\binom{d+1}{2}$ edges, so that its degree
sum is $2dn-d(d+1)$, and its minimum degree is $d$ as soon as $n>d$. Every maximal outerplanar graph is
maximal $2$-degenerate, and the maximal planar graphs have the edge count of the case $d=3$
without being $3$-degenerate; observe that the counting argument of Section \ref{Sec_prelim} applies to the
whole family, since all it uses is the order, the size and the minimum degree.

\begin{problem}
\label{Pr_degenerate}Develop an asymptotic for
\[
\mathrm{MD}(n,d,k)=\min \{\mathrm{sp}(G,k):G\text{ maximal }d\text{-degenerate of order }n\} .
\]
\end{problem}

\bigskip\noindent\textbf{Declaration on the use of artificial intelligence.}~~In preparing this paper the
authors used large language models of Anthropic (the Claude~5 family, in particular Claude Opus~5) to lift
and elaborate ideas from \cite{CaroLauriZarb} and \cite{CaroWest}, to overcome many technical details, to
handle exploration and experiments for small values of $n$, to check computations and proposed proofs, and
for drafting and editing the text. The authors take full responsibility for its content.

\bigskip\noindent\textbf{Acknowledgments.}~~R.~\v{S}krekovski acknowledges the partial support by ARIS program P1-0383, and the annual work program of Rudolfovo.

\end{document}